\documentclass[10pt]{amsart}
\usepackage{amsmath, amscd, amsthm} 
\usepackage{amssymb, amsfonts} 
\usepackage{remreal} 
\usepackage[all]{xy} 
\begin{document}
\title{Vanishing Theorems for Real Algebraic Cycles}
\author{Jeremiah Heller}
\email{heller@math.northwestern.edu}
\author{Mircea Voineagu}
\email{voineagu@usc.edu}

\begin{abstract}
We establish the analogue of the Friedlander-Mazur conjecture for Teh's reduced Lawson homology groups of real varieties, which says that the reduced Lawson homology of a real quasi-projective variety $X$ vanishes in homological degrees larger than the dimension of $X$ in all weights. As an application we obtain a vanishing of homotopy 
groups of the mod-2 topological groups of averaged cycles and a characterization in a range of 
indices of the motivic cohomology of a real variety as homotopy groups of the complex of 
averaged equidimensional cycles. We also establish an equivariant Poincare duality between 
equivariant Friedlander-Walker real morphic cohomology and dos Santos' real Lawson homology. We use this together with an equivariant extension of the mod-2  Beilinson-Lichtenbaum conjecture to compute some real Lawson homology groups in terms of Bredon cohomology.

%

\end{abstract}
\maketitle
\tableofcontents
\section{Introduction}
Let $X$ be a quasi-projective  real variety. The Galois group $G = Gal(\C/\R)$ acts on $\mcal{Z}_{q}(X_{\C})$, the topological group of $q$-cycles on the complexification. Cycles on the real variety $X$ correspond to cycles on $X_{\C}$ which are fixed by conjugation. Inside the topological group of $\mcal{Z}_{q}(X_{\C})^{G}$ of cycles fixed by conjugation is the topological group $\mcal{Z}_{q}(X_{\C})^{av}$ of averaged cycles which are the cycles of the form $\alpha + \overline{\alpha}$. The space of reduced cycles on $X$ is the quotient topological group 
\begin{equation*}
 \mcal{R}_{q}(X) = \frac{\mcal{Z}_{q}(X_{\C})^{G}}{\mcal{Z}_{q}(X_{\C})^{av}}.
\end{equation*}
 Homotopy groups of some of the above abelian topological groups are related to classical  topological invariants. For example for $X$ a projective real variety we obtain the singular homology groups $\pi_{*}\mcal{R}_{0}(X)= H _*(X(\mathbb{R}),\mathbb{Z}/2)$  \cite{Teh:real} and   $\pi_*\mcal{Z}_{0}(X_{\C})^{av}= H _*(X_{\C}(\mathbb{C})/G,\mathbb{Z})$  \cite{LLM:real}, as well as Bredon homology $\pi_{*}\mcal{Z}_{0}(X) = H_{*,0}(X_{\C}(\C);\underline{\Z})$ \cite{LF:Gequiv}.  Other homotopy groups are related  to classical algebraic geometry invariants. For example $\pi _0(Z _r(X _\mathbb{C})^G)$  computes the  group of algebraic cycles of dimension r on $X$ modulo real algebraic equivalence \cite{FW:real} and  consequently  with  $\mathbb{Z}/n$ coefficients equals the Chow group $CH _r(X)\tensor\mathbb{Z}/n$ (see Proposition \ref{fcoeffagr}).   However most of  these homotopy groups  remain a mysterious combination of topological and algebraic information of the real variety $X$. 

These homotopy groups are hard to compute and examples are few. Nonetheless an examination of existing computations shows that these homotopy groups are all zero in large degrees. For example, in (\cite{Lam:t}) Lam proves that
$$
\mcal{R}_{q}(\mathbb{P}^n _\mathbb{R})\simeq \prod _{i=0}^{n-q}K(\mathbb{Z}/2,i) .
$$
In particular $\pi _k(\mcal{R}_{q}(\mathbb{P}^n _\mathbb{R}))=0$ for $k>n-q$. Similar 
vanishing results are seen in the computations of \cite{LLM:quat} for a real variety $X$ with the property that its complexification is the quaternionic projective space (see Example \ref{exR}).

In \cite{Teh:HT} Teh  proves a conditional  Harnack-Thom type theorem  for the homotopy 
groups of reduced algebraic cycles on $X$ which holds under the assumption that these homotopy groups are all finitely generated and they are zero in high degrees. In the case of divisors he shows that
$$
\pi _k\mcal{R} _{d-1}(X)=0
$$
when $k\geq 3$ for any smooth projective real variety $X$ of dimension $d$. 

The main theorem of this paper provides this vanishing in general  and should be viewed as a massive generalization of both the classical vanishing of singular homology groups of a manifold in degree larger than the manifold and  of  the  vanishing results discussed above.
\begin{theorem}
\label{mth}
Let $X$ be a quasi-projective real variety. Then
\begin{equation*}
\pi_{k}\mcal{R}_{q}(X) = 0
\end{equation*}
for $k\geq \dim X -q +1$.
\end{theorem}
 In the case of divisors our result improves the previously known vanishing range. The case of real
 projective space described above shows that the theorem's vanishing range is optimal.

The homotopy groups of reduced algebraic cycles $R _q(X)$ define a homology theory for real 
quasi-projective varieties $X$ introduced in \cite{Teh:real} which is defined by $RL _qH _n(X)=\pi _{n-q}(\mcal{R}_{q}(X))$ for $n\geq q$ and called reduced Lawson homology. In 
this notation our vanishing result reads $RL _qH _n(X)=0$ for any $n>\dim(X)$. Thus our 
vanishing result shows that the  Friedlander-Mazur conjecture holds for the reduced Lawson 
homology of real varieties.

 The homotopy groups of $\mcal{R}_{q}(X)$ fit into a long exact sequence
\begin{equation*}
 \cdots \to \pi_{k+1}\mcal{R}_{q}(X) \to \pi_{k}\mcal{Z}_{q}(X_{\C})^{av} \to L_{q}H\R_{q-k,q}(X) \to \pi_{k}\mcal{R}_{q}(X) \to \cdots
\end{equation*}
where $L_{q}H\R_{q-k,q}(X) = \pi_{k}\mcal{Z}_{q}(X_{\C})^{G}$ is the real Lawson homology introduced by dos Santos in \cite{DS:real}.
As a consequence of Suslin rigidity the homotopy groups of $\mcal{R}_{q}(X)$ are also related to motivic cohomology of $X$
\begin{equation*}
 \cdots  \to \pi_{k}z_{equi}(\A^{q}_{\C},0)(X_{\C}\times \Delta^{\bullet}_{\C})^{av} \to H_{\mcal{M}}^{2q-k}(X; \Z(q)) \to \pi_{k}\mcal{R}_{q}(X) \to \cdots .
\end{equation*}

Thus an immediate corollary of the vanishing theorem is an identification of the homotopy groups of the space of averaged cycles and of the complex of averaged equidimensional cycles. 
\begin{corollary}
Let $X$ be a smooth quasi-projective real variety. Then for any $k\geq \dim X -q+1$
\begin{equation*}
L_{q}H\R_{q-k,q}(X) = \pi_{k}\mcal{Z}_{q}(X_{\C})^{av}
\end{equation*}
 and 
\begin{equation*}
H_{\mcal{M}}^{2q-k}(X;\Z(q)) = \pi_{k}z_{equi}(\A^{q}_{\C},0)(X_{\C}\times \Delta^{\bullet}_{\C})^{av}.
\end{equation*}
\end{corollary}
Theorem \ref{mth} also implies that the mod-2 homotopy groups of the topological group of average cycles satisfy an optimal vanishing (see also Example \ref{avopt}).
\begin{corollary} 
 Let $X$ be a smooth projective real variety of dimension $d$. Then
\begin{equation*}
 \pi_{n}\frac{\mcal{Z}_{p}(X_{\C})^{av}}{2\mcal{Z}_{p}(X_{\C})^{av}} = 0
\end{equation*}
for $n\geq 2d-2p+1$.
\end{corollary}

An essential ingredient in the proof of our vanishing theorem is the  Milnor conjecture proved by Voevodsky in \cite{Voev:miln}. The Milnor conjecture relates motivic cohomology and etale cohomology while real morphic cohomology naturally compares with Bredon cohomology. We need to know that these cycle maps are suitably related which is done in Theorem \ref{cyccomp}, 
\begin{theorem}
\label{dg}
Let $X$ be a smooth quasi-projective real variety. The diagram commutes
\begin{equation*}
 \xymatrix{
 L^{q}H\R^{q-k,q}(X;\Z/2) \ar[d]_-{\Phi} && \ar[ll]_{\iso}\H^{2q-k,q}_{\mcal{M}}(X;\Z/2)  \ar[d]^-{cyc}  \\
H^{q-k,q}(X_{\C}(\C);\underline{\Z/2})  \ar[r] & H_{G}^{2q-k}(X_{\C}(\C);\Z/2)\ar[r] &  H^{2q-k}_{et}(X;\mu_{2}^{\otimes q}),
}
\end{equation*}
where $H^{p-q,q}(X_{\C}(\C);\underline{\Z/2})$ denotes Bredon cohomology and $H_{G}^{p}(X_{\C}(\C);\Z/2)$ denotes Borel cohomology.
\end{theorem}
This suggests that there  are possible advantages in replacing the map on Chow groups of real cycles into Borel cohomology with the map into Bredon cohomology since in many respects  Bredon cohomology behaves better than Borel cohomology. An application of this idea will be given in a forthcoming paper. 

Together with the mod-2 Beilinson-Lichtenbaum conjecture for real and complex varieties (which is a consequence of the Milnor conjecture by \cite{SV:BK}) we conclude  an  equivariant Beilinson-Lichtenbaum type theorem for an equivariant extension of Friedlander-Walker's real morphic cohomology groups (see Definition \ref{emco}).
\begin{theorem}
 Let $X$ be a smooth quasi-projective real variety and $k>0$. The cycle map
\begin{equation*}
 \Phi:L^{q}H\R^{r,s}(X; \Z/2^{k}) \to H^{r,s}(X_{\C}(\C); \underline{\Z/2^{k}})
\end{equation*}
is an isomorphism if $r\leq 0$ (and $s\leq q$) and an injection if $r=1$ (and $s\leq q$).
\end{theorem}

Using Friedlander-Voevodsky duality for bivariant cycle theory we show in Corollary \ref{PD} that the equivariant morphic cohomology and real Lawson homology groups are isomorphic through a Poincare duality. As a consequence the equivariant Beilinson-Lichtenbaum says that in a range we may compute the mod-2 real Lawson homology groups in terms of mod-2 Bredon cohomology. This allows a computation for curves with integral coefficients.
\begin{corollary}
 Let $X$ be a smooth real curve. Then 
\begin{equation*}
 L^{q}H\R^{r,s}(X; \Z) \to H^{r,s}(X_{\C}(\C); \underline{\Z})
\end{equation*}
is an isomorphism for any $q\geq 0$, $r\leq q$, and $s\leq q$.
\end{corollary}

The space of reduced cocycles on $X$, related via Poincare duality with the space of reduced cycles, is defined as
\begin{equation*}
\mcal{R}^{q}(X)=\frac{\mcal{Z}^{q}(X_{\C})^{G}}{\mcal{Z}^{q}(X_{\C})^{av}}
\end{equation*}
where $\mcal{Z}^{q}(X_{\C})$ is the space of algebraic cocycles on $X_{\C}$ and $\mcal{Z}^{q}(X_{\C})^{G}$ agrees with the space of real cocycles introduced by Friedlander-Walker in \cite{FW:real} (see Proposition \ref{nqp}). There is a natural comparison map
\begin{equation*}\label{cmap}
\mcal{R}^{q}(X) \to \map{X(\R)}{\mcal{R}_{0}(\A^{q})}
\end{equation*}
and since $\mcal{R}_{0}(\A^{q})=K(\Z/2,q)$ this provides a natural map
\begin{equation}\label{tcycm}
cyc _{k}: \pi_{k}\mcal{R}^{q}(X) \to H^{q-k}_{sing}(X(\R);\Z/2)
\end{equation}
which is the cycle map for reduced morphic cohomology groups defined in \cite{Teh:real}. Via Poincare duality the vanishing theorem is equivalent to the statement that $cyc_{k}$ is an isomorphism for $k>q$.

Via the Milnor conjecture over $\C$ and over $\R$ we can deduce an isomorphism  $\pi_{k}\mcal{R}^{q}(X)\to\pi_{k}\mcal{R}^{q}_{top}(X)$ for $k\geq q$.  Here $\mcal{R}^{q}_{top}(X)$ is the group of ``reduced topological cocycles''. For a precise definition see Section \ref{subdual}, but essentially this is a version of the quotient group $\map{X_{\C}(\C)}{\mcal{Z}_{0}(\A^{q}_{\C})}^{G}/\map{X_{\C}(\C)}{\mcal{Z}_{0}(\A^{q}_{\C})}^{av}$ which has reasonable homotopical properties (such as fitting into a homotopy fiber sequence involving $\map{X_{\C}(\C)}{\mcal{Z}_{0}(\A^{q}_{\C})}^{G}$ and $\map{X_{\C}(\C)}{\mcal{Z}_{0}(\A^{q}_{\C})}^{av}$).

The final ingredient for our vanishing theorem is now provided by Corollary \ref{rdual} which shows that for $X$ projective,
\begin{equation*}
\widetilde{\Phi}:\pi_{k}\mcal{R}^{q}(X)\rightarrow \pi_{k}\mcal{R}^{q} _{top}(X) 
\end{equation*}
agrees with the cycle map $\pi_{k}\mcal{R}^{q}(X) \to H^{q-k}_{sing}(X(\R);\Z/2)$ for $k\geq 2$.

Here is a short outline of the paper. In the second section we review the equivariant homotopy 
used in the paper. The third section is dedicated to introducing the topological spaces of cycles 
we study and proving some basic properties that we use and for which we don't find exact 
references in the literature. In the fourth section we prove a Poincare Duality between 
equivariant morphic cohomology  and real Lawson homology. In the fifth section we discuss the cycle maps from equivariant morphic cohomology and Bredon cohomology and equivariant 
applications of the Beilinson-Lichtenbaum conjecture. The sixth section is devoted to the proof 
of our main vanishing Theorem. One of the main technical ingredients of this proof is left for 
section seven where we reinterpret the cycle map \ref{tcycm} from reduced Lawson homology 
groups to the singular homology in a manner needed to prove our vanishing theorem. The 
paper ends with two appendixes where we prove and recollect a few results on topological 
monoids used in the paper.

The authors would like to thank Eric Friedlander, Christian Haesemeyer and Mark Walker for helpful discussions.

\vskip 0.5cm
\textit{Notation:} By a quasi-projective $k$-variety we mean a reduced and separated quasi-projective scheme of finite type over a field $k$. We write $Sch/k$ for the category of quasi-projective $k$-varieties  and  $Sm/k$ for the subcategory of smooth quasi-projective $k$-varieties. Except in section 2, $G$ always denotes $Gal(\C/\R)$ and $\sigma\in G$ denotes the nontrivial element.

\section{Equivariant Homotopy and Cohomology}

We recall the basic definitions and theorems we need from equivariant  homotopy theory. For more details see \cite{May:equi}. In this paper we will only work with $G=\Z/2$, but since no simplification results in the basic definitions, we let $G$ denote an arbitrary finite group.  The category $\GTop$ of $G$-spaces consists of compactly generated spaces equipped with a left $G$-action and morphisms are continuous $G$-equivariant maps. If $X$ is a $G$-space and $H\subseteq G$ is a subgroup write $X^{H}$ for the subspace of all points fixed by $H$. The category $\GTop_*$ of based $G$-spaces consists of  $G$-spaces $X$ together with a $G$-invariant basepoint $x\in X$ and maps are base-point preserving equivariant maps. A space together with a disjoint, invariant base-point will be denoted $X_{+}$.


\subsection{Equivariant homotopy theory} 
Let $I$ denote the unit interval with trivial $G$-action. A $G$-homotopy between two 
equivariant maps $f,g:X \to Y$ is an equivariant map $F:X\times I \to Y$ such that 
$F|_{X\times\{0\}} = f$ and $F|_{X\times\{1\}} = g$.  An equivariant map $f:X\to Y$ is an \textit{equivariant homotopy equivalence} provided there is an equivariant map $g:Y \to X$ 
such that both $f\circ g $ and $g \circ f$ are $G$-equivalently homotopic to the identity.
An equivariant map $f:X\to Y$ is a \textit{$G$-weak equivalence} provided both 
$f^{H}:X^{H}\to Y^{H}$ is a non-equivariant weak equivalence for all subgroups $H\subseteq G$. Formally inverting the $G$-weak equivalences gives the homotopy category of 
$G$-spaces. Similarly inverting the based $G$-weak equivalences between based $G$-spaces we 
obtain the based $G$-homotopy category. Write $[ X,Y]_{G}$ for classes of based maps  in the homotopy category of based maps.

A $G$-$CW$ complex $X$ is a topological union $X = \cup X_{n}$ of $G$-spaces such that $X_{0}$ is a disjoint union of orbits $G/H$ and $X_{n}$ is obtained from $X_{n-1}$ by attaching cells of the form  $D^{n}\times G/H$ via attaching maps  $\sigma:S^{n-1}\times G/H \to X_{n-1}$.

The equivariant Whitehead theorem holds for $G$-$CW$ complexes. That is, if $f:X\to Y$ is a $G$-equivariant weak equivalence between $G-CW$-complexes then $f$ is a $G$-homotopy equivalence.

A map $A\to X$ is said to have the \textit{homotopy extension property} with respect to $Z$ if   for any  equivariant partial homotopy $H:X\times\{0\}\coprod_{A\times\{0\}}A\times I \to Z$  there is an equivariant map $H'$ making the diagram below commute
\begin{equation*}
 \xymatrix{
X\times\{0\}\coprod_{A\times\{0\}}A\times I \ar[r]^-{H}\ar[d] & Z \\
X\times I \ar@{-->}[ur]^{\exists H'} & .
}
\end{equation*}
 An equivariant \textit{cofibration} $A\hookrightarrow X$ is an equivariant map which has the homotopy extension property with respect to all $Z$ in $\GTop$. Inclusions of sub-$G$-$CW$complexes $A\subseteq X$ are equivariant cofibrations. 

Let $V$ be a real representation of $G$, write $S^{V}$ for the one-point compactification of $V$. The $V$th homotopy group of a based $G$-space $X$ is
\begin{equation*}
 \pi_{V} X = [S^{V}, X]_{G} .
\end{equation*}
Note that $S^{V}$ always has at least two fixed points, $0$ and $\infty$.

When $G=\Z/2$ and $V = \R^{p,q}$, where $V =\R^{p+q}$ with $G$ acting trivially on the first $p$-components and on the last $q$-components the $G$ action is multiplication by $-1$ we use the notation
\begin{equation*}
 \pi_{p,q}X = \pi_{\R^{p,q}}X .
\end{equation*}


\subsection{Borel homology and cohomology}
The \textit{Borel-equivariant cohomology } of $X$ with coefficients in an abelian group $A$ is defined to be the ordinary singular cohomology of the homotopy orbit space of $X$:
\begin{equation*}
 H^{p}_{G}(X; A) = H^{p}((X\times EG)/G ; A).
\end{equation*}
Similarly the \textit{Borel-equivariant homology} is defined to be
\begin{equation*}
 H_{p}^{G}(X; A) = H_{p}((X\times EG)/G ; A).
\end{equation*}

When $X$ has free $G$-action then $ (X\times EG)/G \to X/G$ is a homotopy equivalence and therefore when $X$ has free $G$-action $ H^{p}(X/G; A) \iso H^{p}_{G}(X; A)$
and $H_{p}(X/G; A) \iso H_{p}^{G}(X; A)$.

\subsection{Mackey functors}
Bredon homology and cohomology take Mackey functors as coefficients. There are several equivalent ways to define a Mackey functor \cite{May:equi}. Classically for $G$ a finite group one defines a Mackey functor as follows . Let $\mcal{F}_{G}$ denote the category of finite $G$-sets as objects and with equivariant set maps as morphisms. A \textit{Mackey functor} $\underline{M}$ consists of a pair of abelian-group valued functors $\underline{M}=(M^*, M_*)$ on $\mcal{F}_{G}$, with $M^*$ contravariant and $M_*$ covariant. The functors $M^*$ and $M_*$ satisfy the following requirements.
\begin{enumerate}
 \item $M^{*}$,$M_{*}$ take the same value on objects and convert disjoint unions of $G$-sets into products of abelian groups. 
\item When 
\begin{equation*}
\xymatrix{
S \ar[r]^{\sigma'}\ar[d]^{\beta'} & T \ar[d]^{\beta} \\
U \ar[r]^{\sigma}  & V 
}
\end{equation*}
is a pull-back square of finite $G$-sets then
\begin{equation*}
\xymatrix{
M(S) \ar[r]^{M_{*}(\sigma')} & M(T) \\
M(U) \ar[u]^{M^{*}(\beta')} \ar[r]^{M_{*}(\sigma)} & M(V) \ar[u]^{M^{*}(\beta)}
}
\end{equation*}
is a commutative square of abelian groups.
\end{enumerate}
Given an abelian group $A$, the \textit{constant Mackey functor} $\underline{A}$ is the Mackey functor which on objects $\underline{A}(G/K) = A$ and on a map $f:G/H\to G/K$, $M^*(f) = id$ and $M_*(f)$ is multiplication by the index $[K:H]$.

\subsection{Bredon homology and cohomology}
\textit{Bredon cohomology (homology)} with coefficients in a Mackey functor $\underline{M}$ is a cohomology (homology) theory $H^{*}(-;\underline{M}))$ ($H_{*}(-;\underline{M})$ graded by $RO(G)$. For $V \in RO(G)$ there is an equivariant Eilenberg-Maclane space $K(\underline{M},V)$ which represents the reduced cohomology 
$\tilde{H}^{V}(X; \underline{M})$ for a based $G$-space, 
\begin{equation*}
 \tilde{H}^{V}(X; \underline{M}) = [X , K(\underline{M},V)]_{G} .
\end{equation*}

When $G=\Z/2$ then $RO(G) = \Z\oplus\Z$ with generators $\R^{1,0}$ and $\R^{0,1}$. We use the convention that $H^{p,q}(X;\underline{M}) = H^{\R^{p,q}}(X;\underline{M})$ (and similarly for homology). 

If $A$ is an abelian group  (with trivial $G$-action) then $H^{p,0}(X;\underline{A}) = H^{p}_{sing}(X/G; A)$. More generally, Borel  and Bredon cohomology are related by the natural isomorphism 
\begin{equation*}
 H^{p,q}(X\times EG; \underline{A}) \iso H^{p+q}_{G}(X; A(q))
\end{equation*}
where $A(q)$ is $A$ with $\sigma$ acting by $(-1)^{q}$ (see \cite[Proposition 1.15]{DSLF:quat}).

\subsection{Equivariant Dold-Thom theorem}
Let $X$ be a compactly generated Hausdorff space. The free abelian group on the points of $X$ is defined to be $ \mcal{Z}_{0}(X) = [\coprod_{d}SP^{d}(X)]^{+}$, where $SP^{d}(X)$ is 
the $d$th symmetric product on $X$ and $(-)^{+}$ denotes group completion of the displayed monoid which is topologized via the quotient topology.

The degree homomorphism $\deg:\mcal{Z}_{0}(X) \to \Z$ is defined by $\deg(\sum n_{i}x_i) = \sum n_i$ is a continuous homomorphism. Write $\mcal{Z}_{0}(X)_{0}$ for the kernel of this map. Notice that there is an isomorphism of topological groups $\mcal{Z}_{0}(X) \iso \mcal{Z}_{0}(X_+)_{0}$.

If $X$ is a $G$-space then the action on $X$ induces a $G$-action on $\mcal{Z}_{0}(X)$ and $\mcal{Z}_{0}(X)_{0}$. By \cite[Corollary 2.9]{LF:Gequiv} when $X$ is a $G$-$CW$ complex so is $\mcal{Z}_{0}(X)$. 

The classical Dold-Thom theorem says that $\pi_{n}\mcal{Z}_{0}(X)_{0} = \tilde{H}_{n}(X;\Z)$ and the equivariant Dold-Thom theorem proved by Lima-Filho \cite{LF:Gequiv} and dos Santos \cite{DS:equiDT} says that 
\begin{equation*}
\pi_{V}\mcal{Z}_{0}(X)_{0} = \tilde{H}_{V}(X;\underline{\Z}).
\end{equation*}
In particular $\mcal{Z}_{0}(S^{V})_{0}$ is an Eilenberg-Maclane space $K(\underline{Z}, V)$.

\section{Topological Spaces of Cycles}\label{topspcyc}
\subsection{Group completions of monoids}

Let $M$ be a compactly generated Hausdorff topological abelian monoid. The \textit{naive group completion} of $M$ is the quotient of $M\times M$ by the monoid action of $M$ where $M$ acts by $(a,b) \mapsto (m+a, m+b)$. Write $M^{+}$ for this abelian group, which is 
topologized as the quotient of $M\times M$. Recall that $M$ is said to satisfy the \textit{cancellation property} if $a + m = b+m$ implies that $a = b$ for any $a,b,m\in M$. 
When $M$ satisfies cancellation then the naive group completion can be described as $M^{+}= M\times M/\sim$, where $(a,b) \sim (c,d)$ if $a+d = b +c$. 

Naive group completion does not generally behave well topologically. For example it may 
happen that $M^{+}$ is not a Hausdorff topological group nor is it clear how homotopy 
invariants of $M$ and $M^{+}$ are related. Friedlander-Gabber \cite{FG:cyc} and Lima-Filho \cite{Li:completion} have studied conditions under which the naive group completion of a 
topological monoid is a Hausdorff group and $M\to M^{+}$ is a homotopy group completion. All of the 
topological monoids with which we work are tractable monoids in the sense of 
Friedlander-Gabber (see Appendix \ref{tract}) and in particular the naive group completion of these groups are homotopy group completions. 

Our main objects of interest are the group completions of submonoids of the Chow monoids of effective algebraic cycles on algebraic varieties. Let $k$ be a field of characteristic 0 and $j:Y\subseteq \P^n_k$ be a projective $k$-variety. The Chow variety $\mcal{C}_{q}(Y,j) = \coprod_{d\geq 0}\mcal{C}_{q,d}(Y,j) $ of effective $q$-dimensional cycles on $Y$ is an (infinite, disjoint) union of projective $k$-varieties. See \cite{F:ac} for details.

\subsection{Cycle-spaces over $\C$}

Let $Z$ be a complex variety.  Denote the set of complex points equipped with the analytic topology by  $Z(\C)^{an}$. Since there will be little chance for confusion we will often simply write this space as $Z(\C)$ with the topology understood. If $j:Y\subseteq \P^n_{\C}$ is a projective variety then $\mcal{C}_q(Y,j)(\C)$ is a topological monoid and we will generally omit $j$ from the notation since the homeomorphism type of this space is independent of $j$. 

The monoid $\mcal{C}_q(Y)(\C)^{an}$ is tractable  and therefore the naive group completion is a homotopy group completion. Write  
\begin{equation*}
 \mcal{Z}_q(Y) = (\mcal{C}_q(Y)(\C))^{+} 
\end{equation*}
for the naive group group completion of this monoid.  Define the filtration $\{0\} \subseteq\cdots \subseteq \mcal{Z}_{q,\leq d}(Y) \subseteq \mcal{Z}_{q,\leq d+1}(Y) \subseteq \cdots \subseteq \mcal{Z}_{q}(Y)$ by
\begin{equation*}
\mcal{Z}_{q,\leq d}(Y) = \left( \coprod_{d_1+d_2 \leq d} \mcal{C}_{q,d_1}(Y)(\C)\times \mcal{C}_{q,d_2}(Y)(\C) \right )/ \sim \subseteq \mcal{Z}_{q}(Y). 
\end{equation*}
By \cite{Li:completion} each $\mcal{Z}_{q,\leq d}(Y)$ is a closed, compact Hausdorff space and $\mcal{Z}_{q}(Y)$ has the weak topology with respect to this filtration. 

When $U$ is quasi-projective with projectivization $U\subseteq \overline{U}$ then define $\mcal{Z}_{q}(U) = \mcal{Z}_{q}(\overline{U})/\mcal{Z}_{q}(U_{\infty})$ where $U_{\infty} = \overline{U}\backslash U$. The images of $\mcal{Z}_{q,\leq k}(\overline{U})$ in the quotient $\mcal{Z}_{q}(U)$ give a filtration by compact subspaces (and $\mcal{Z}_{q}(U)$ has the weak topology with respect to this filtration). This definition is independent of choice of projectivization \cite{LF:qproj}, \cite{FG:cyc}.

\subsection{Cycle-spaces over $\R$}

Suppose that $Z$ is a real variety. Write the set of $\R$-points equipped with the analytic topology  as $Z(\R)^{an}$ or simply $Z(\R)$ with the topology understood. 

Let $Y$ be a projective real variety. Consider the topological monoid $\mcal{C}_q(Y)(\R)$. As 
explained in the proof of \cite[Proposition 8.2]{FW:real} (see Proposition \ref{gtract}), the 
topological monoid $\mcal{C}_q(Y)(\R)$ is tractable and therefore its naive group completion 
is a homotopy group completion. Write
\begin{equation*}
 \mcal{Z}_{q}(Y) = (\mcal{C}_{q}(Y)(\R))^{+}
\end{equation*}
for the naive group completion. 

Suppose that $U$ is a quasi-projective real variety with projectivization $U\subseteq \overline{U}$. Define the topological group of $q$-cycles on the quasi-projective real variety $U$ to be
\begin{equation*}
 \mcal{Z}_{r}(U) = \mcal{Z}_{r}(\overline{U})/\mcal{Z}_{r}(\overline{U}\backslash U).
\end{equation*}

If $X$ is a real variety and $\pi: X_{\C} \to X$ is its complexification then $G$ acts on $X_{\C}(\C)$ and  induces a homeomorphism
\begin{equation*}
 X(\R) \xrightarrow{\iso} X_{\C}(\C)^G.
\end{equation*}

In particular if $X$ is a projective real variety then by \cite[Proposition 1.1]{F:ac}  $\mcal{C}_q(X_{\C}) = \mcal{C}_{r}(X)_{\C}$ and so we have the isomorphism of topological monoids
\begin{equation*}
 \mcal{C}_{r}(X)(\R) \xrightarrow{\iso} \mcal{C}_{r}(X_{\C})(\C)^{G} .
\end{equation*}

\begin{proposition}
 Let $U$ be a quasiprojective real variety. Then
\begin{equation*}
\mcal{Z}_{r}(U) \xrightarrow{\iso} \mcal{Z}_{r}(U_{\C})^{G}
\end{equation*}
is an isomorphism of topological abelian groups. In particular the group $\mcal{Z}_{r}(U)$ is independent of projectivization $U\subseteq \overline{U}$.
\end{proposition}
\begin{proof}
For $U$ projective this follows immediately from Proposition \ref{appcycagr}. The quasi-projective case now follows by a comparison of short exact sequences of topological abelian groups
\begin{equation*}
 \xymatrix{
0 \ar[r] & \mcal{Z}_{r}(\overline{U}\backslash U) \ar[r]\ar[d] & \mcal{Z}_{r}(\overline{U}) \ar[r]\ar[d] & \mcal{Z}_{r}(U) \ar[r]\ar[d] & 0 \\
0 \ar[r] & \mcal{Z}_{r}(\overline{U}_{\C}\backslash U_{\C})^{G} \ar[r] & \mcal{Z}_{r}(\overline{U}_{\C})^{G} \ar[r] & \mcal{Z}_{r}(U_{\C})^{G} \ar[r] & 0,
}
\end{equation*}
where the exactness of the bottom row is a consequence of the Lemma \ref{ccom}.
\end{proof}

\begin{remark}\label{addqp}
Let $U$ be a quasi-projective real variety. Since $+:\mcal{Z}_{k}(U_{\C})\times\mcal{Z}_{k}(U_{\C}) \to \mcal{Z}_{k}(U_{\C})$ is closed we see by taking $G$-fixed points that $+:\mcal{Z}_{k}(U) \times \mcal{Z}_{k}(U) \to \mcal{Z}_{k}(U)$ is a closed map for any real variety $U$.
\end{remark}

\begin{lemma} \label{ccom}
 Let $Y\subseteq X$ be a closed subvariety of a real projective variety. Then 
\begin{equation*}
 \mcal{Z}_{r}(X_{\C})^{G} / \mcal{Z}_{r}(Y_{\C})^{G} \xrightarrow{\iso} (\mcal{Z}_{r}(X_{\C})/\mcal{Z}_{r}(Y_{\C}))^{G}
\end{equation*}
and
\begin{equation*}
 \mcal{Z}_{r}(X_{\C})^{av} / \mcal{Z}_{r}(Y_{\C})^{av} \xrightarrow{\iso} (\mcal{Z}_{r}(X_{\C})/\mcal{Z}_{r}(Y_{\C}))^{av}
\end{equation*}
are isomorphisms of topological groups.
\end{lemma}
 \begin{proof}
Consider the quotient maps $\pi: \mcal{Z}_{r}(X_{\C}) \to \mcal{Z}_{r}(X_{\C})/\mcal{Z}_{r}(Y_{\C})$ and $q:\mcal{Z}_{r}(X_{\C})^{G}\to \mcal{Z}_{r}(X_{\C})^{G}/\mcal{Z}_{r}(Y_{\C})^{G}$. Consider the filtration $\{(\pi\mcal{Z}_{r}(X_{\C})_{\leq d}\})^{G}$ of $(\mcal{Z}_{r}(X_{\C})/\mcal{Z}_{r}(Y_{\C}))^{G}$ and the filtration $\{q(\mcal{Z}_{r}(X_{\C})^{G}_{\leq d})\}$ of  $\mcal{Z}_{r}(X_{\C})^{G} / \mcal{Z}_{r}(Y_{\C})^{G}$. 
These spaces have the weak topology given by these filtrations so it is enough to see that 
\begin{equation*}
 q (\mcal{Z}_{r}(X_{\C})_{\leq d}^{G})\xrightarrow{} (\pi\mcal{Z}_{r}(X_{\C})_{\leq d})^{G}
\end{equation*}
is a homeomorphism for all $d$.

First we show that $\mcal{Z}_{r}(X_{\C})_{\leq d}^{G}\to (\pi\mcal{Z}_{r}(X_{\C})_{\leq d})^{G}$ is surjective. 
If $[\eta]\in (\pi\mcal{Z}_{r}(X_{\C})_{\leq d})^{G}$, we can choose a representative $\eta= \sum n_{V} V\in \mcal{Z}_{r}(X_{\C})_{\leq d}$  such that each $V\nsubseteq Y_{\C}$.  Since $\eta -\overline{\eta} \in \mcal{Z}_{r}(Y_{\C})$ (and each $\overline{V} \nsubseteq Y_{\C}$) we see that $\eta = \overline{\eta}$ and therefore the map
\begin{equation*}
 \mcal{Z}_{r}(X_{\C})_{\leq d}^{G} \xrightarrow{} (\pi\mcal{Z}_{r}(X_{\C})_{\leq d})^{G}
\end{equation*}
is surjective. 
The map
\begin{equation*}
 q\mcal{Z}_{r}(X_{\C})_{\leq d}^{G}\xrightarrow{} (\pi\mcal{Z}_{r}(X_{\C})_{\leq d})^{G}
\end{equation*}
is easily seen to be injective and since $\mcal{Z}_{r}(X_{\C})_{\leq d}^{G}$ is compact this map is closed so is a homeomorphism.

The second statement about the topological group of averaged cycles is proved in a similar fashion.
 \end{proof}

\subsection{Spaces of algebraic cocycles}
 In this section we recall the construction of topological monoids of algebraic cocycles \cite{FL:algco, F:algco}. Let $X$, $Y$ be quasi-projective real varieties over $k=\C$ or $\R$. Write $\Mor{k}{X}{Y}$ for the set of continuous algebraic maps between $X$ and $Y$. When $X$ is semi-normal  then $\Mor{k}{X}{Y} = \Hom{}{X}{Y}$.
Friedlander-Walker  construct ``analytic'' topologies on $\Mor{k}{X}{Y}$ in \cite{FW:funcspc} for $k=\C$ and in \cite{FW:real} for $k=\R$. The set of continuous algebraic maps with this topology will be written $\Mor{k}{X}{Y}^{an}$.  By \cite[Lemma 1.2]{FW:real} 
$\Mor{\R}{X}{Y}^{an} = (\Mor{\C}{X_{\C}}{Y_{\C}}^{an})^{G}.$ 

When $X$,$Y$ are projective real varieties and $W$,$Z$ are projective complex varieties then this topology coincides with the subspace topology induced by the inclusions
\begin{equation*}
\Mor{\C}{W}{Z}  \subseteq \map{W(\C)}{Z(\C)}
\end{equation*}
and
\begin{equation*}
 \Mor{\R}{X}{Y}  \subseteq \map{X_{\C}(\C)}{Y_{\C}(\C)}^{G},
\end{equation*}
where $\map{-}{-}$ denotes the space of continuous maps is with compact-open topology. When the domain is only quasi-projective then the analytic topology on the algebraic mapping spaces is no longer the compact-open topology but rather the topology of convergence with bounded degree (see \cite[Appendix A]{FL:dual}).

Let $W$ be a quasi-projective complex variety and $Z$ be a projective complex variety. Write $d= \dim W$. Let $\mcal{C}_{r}(Z)(W)$ denote the monoid of effective cycles on $W\times Z$ equidimensional or relative dimension $r$ on $W$. This is made into a topological monoid via the subspace topology induced by the inclusion
\begin{equation*}
\mcal{C}_{r}(Z)(W) \subseteq \mcal{C}_{d+r}(W \times Z) \eqdef \frac{\mcal{C}_{d+r}(\overline{W}\times Z)}{\mcal{C}_{d+r}(W_{\infty}\times Z)}
\end{equation*}
where $W\subseteq \overline{W}$ a projective closure with closed complement $W_{\infty} = \overline{W}\backslash W$. This topology may also be described as follows. Let 
\begin{equation*}
\mcal{E}_{r}(Z)(W) \subseteq \mcal{C}_{r+d}(\overline{W}\times Z)
\end{equation*}
 denote the constructable submonoid consisting of effective cycles whose restriction to $W\times Z$ is equidimensional of relative dimension $r$ over $W$. By \cite[Proposition 1.8]{F:algco} the topology on $\mcal{C}_{r}(Z)(W)$ (given by the subspace topology above) coincides with the quotient topology given by
\begin{equation*}
 \mcal{C}_{r}(Z)(W) = \frac{\mcal{E}_{r}(Z)(W)}{\mcal{C}_{r+d}(W_{\infty}\times Z)}.
\end{equation*}

Define the topological group of equidimensional cycles of relative dimension r over $W$ as $\mcal{Z}_{r}(Z)(W) = [\mcal{C}_{r}(Z)(W)]^{+}$ (where as usual the naive group completion is given the quotient topology). Since $\mcal{C}_{r}(Z)(W)$ is a tractable monoid the naive group completion is a homotopy group completion. 

In \cite{F:ac} it is shown that a morphism of varieties $f:W\to \mcal{C}_{r}(Z)$ has an associated graph in $\mcal{Z}_{f} \in \mcal{C}_{r}(Z)(W)$. By \cite[Proposition A.1]{FL:dual} this  defines an isomorphism of topological monoids
$ \Gamma: \Mor{\C}{W}{\mcal{C}_r(Z)} \to \mcal{C}_{r}(Z)(W)$ for any normal, quasi-projective complex variety $W$ by \cite[Proposition A.1]{FL:dual}. Therefore the graph map $\Gamma$ also induces an isomorphism of topological abelian groups 
\begin{equation*}
 \Gamma: \Mor{\C}{W}{\mcal{C}_0(Z)}^{+} \to \mcal{Z}_{0}(W)(Z),
\end{equation*}
for any normal, quasi-projective variety $W$ and any projective variety $Z$.
The composite of $\Gamma$ and the continuous inclusion $\mcal{Z}_{0}(W)(Z)\subseteq \mcal{Z}_{\dim W}(W\times Z)$ defines the duality map
\begin{equation*}
\mcal{D}: \Mor{\C}{W}{\mcal{C}_0(Z)}^{+} \xrightarrow{\Gamma} \mcal{Z}_{0}(W)(Z) \subseteq \mcal{Z}_{d}(W\times Z).
\end{equation*}
While this is a continuous injective homomorphism it is not a topological embedding (see \cite{FL:dual}).

\begin{lemma}(c.f. \cite[Proposition 2.9]{Teh:real})
\begin{enumerate}
\item If $W$ is a normal quasi-projective complex variety and $Z _1\subset Z _2$ is a closed subvariety of a complex projective variety then
$$
\Mor{\C}{W}{\mcal{C}_{0}(Z_{1})}^{+}\subseteq \Mor{\C}{W}{\mcal{C}_{0}(Z_{2})}^{+}
$$
 is a closed subspace. 

\item If $U$ is a normal quasi-projective real variety and $Y\subseteq Z$ is a closed subvariety of a projective real variety then
$$
\Mor{\R}{U}{\mcal{C}_{0}(Y)}^{+}\subseteq \Mor{\R}{U}{\mcal{C}_{0}(Z)}^{+}
$$
is a closed subspace.
\end{enumerate}
\end{lemma}
\begin{proof}
For the first statement it is equivalent to show that $\mcal{Z}_{r}(Z_{1})(W) \subseteq \mcal{Z}_{r}(Z_{2})(W)$ is closed. Using Lemma \ref{qmcl} we see that $\mcal{C} _{r+d}(W\times Z _1)\subset \mcal{C} _{r+d}(W\times Z _2)$ is closed and since $\mcal{C}_{r}(Z_1)(W) = \mcal{C} _{r+d}(W\times Z _1)\cap \mcal{C}_{r}(Z_{2})(W)$ we conclude that $\mcal{C}_{r}(Z_1)(W)\subset\mcal{C}_{r}(Z_{2})(W)$ is a closed subspace.
Write $\pi:\mcal{C}_{r}(Z_{2})(W)^{\times 2} \to \mcal{Z}_{r}(Z_{2})(W)$ for the quotient. Then $\pi^{-1}\mcal{Z}_{r}(Z_{1})(W) =\mcal{C}_{r}(Z_{1})(W)^{\times 2} + \Delta(\mcal{C}_{r}(Z_{2})(W)) $ (where $\Delta$ denotes the diagonal) is closed by step (1) in Proposition \ref{Gcocomm}. Therefore $\mcal{Z}_{r}(Z_{1})(W) \subseteq \mcal{Z}_{r}(Z_{2})(W)$ is closed.

The second statement follows immediately from the first statement together with Proposition \ref{appcycagr} and \cite[Lemma 1.2]{FW:real}. 
\end{proof}

\begin{definition}
\begin{enumerate}
 \item 
Let $W$ be a quasi-projective complex variety. The space of algebraic $q$-cocyles is defined to be
\begin{equation*}
 \mcal{Z}^{q}(W) = \frac{\Mor{\C}{W}{\mcal{C}_{0}(\P_{\C}^{q})}^{+}}{\Mor{\C}{W}{\mcal{C}_{0}(\P_{\C}^{q-1})}^{+}}
\end{equation*}

\item 
Let $U$ be a quasi-projective real variety. The space of real algebraic $q$-cocyles is defined to be
\begin{equation*}
 \mcal{Z}^{q}(U) = \frac{\Mor{\R}{U}{\mcal{C}_{0}(\P_{\R}^{q})}^{+}}{\Mor{\R}{U}{\mcal{C}_{0}(\P_{\R}^{q-1})}^{+}}
\end{equation*}
\end{enumerate}
\end{definition}

\begin{proposition}\label{Gcocomm}
 Let $U$ be a normal quasi-projective real variety then
\begin{equation*}
 \frac{(\Mor{\C}{U_{\C}}{\mcal{C}_0(\P^q_{\C})}^{an,+})^{G}}{(
\Mor{\C}{U_{\C}}{\mcal{C}_0(\P^{q-1}_{\C})}^{an,+})^{G}} \xrightarrow{\iso} \left(\frac{\Mor{\C}{U_{\C}}{\mcal{C}_0(\P^q_{\C})}^{an,+}}{
\Mor{\C}{U_{\C}}{\mcal{C}_0(\P^{q-1}_{\C})}^{an,+}}\right)^{G}
\end{equation*}
and
\begin{equation*}
 \frac{(\Mor{\C}{U_{\C}}{\mcal{C}_0(\P^q_{\C})}^{an,+})^{av}}{(
\Mor{\C}{U_{\C}}{\mcal{C}_0(\P^{q-1}_{\C})}^{an,+})^{av}} \xrightarrow{\iso} \left(\frac{\Mor{\C}{U_{\C}}{\mcal{C}_0(\P^q_{\C})}^{an,+}}{
\Mor{\C}{U_{\C}}{\mcal{C}_0(\P^{q-1}_{\C})}^{an,+}}\right)^{av}
\end{equation*}
are isomorphisms of topological groups. 
\end{proposition}

\begin{proof}
By Lemma \ref{lemtop1} and Proposition \ref{appcycagr} it is enough to show that for $Y'\subseteq Y$ a closed subvariety of a projective real variety and $U$ a quasiprojective real variety that
\begin{equation*}
 \frac{\mcal{C}_{r}(Y_{\C})(U_{\C})^{G}}{\mcal{C}_{r}(Y'_{\C})(U_{\C})^{G}}\xrightarrow{} \left(\frac{\mcal{C}_{r}(Y_{\C})(U_{\C})}{\mcal{C}_{r}(Y'_{\C})(U_{\C})}\right)^{G}
\end{equation*}
is an isomorphism of topological monoids.

We proceed in several steps.
\begin{enumerate}
\item The map
\begin{equation*}
+: \mcal{C}_{r}(Y_{\C})(U_{\C})\times\mcal{C}_{r}(Y_{\C})(U_{\C}) \to \mcal{C}_{r}(Y_{\C})(U_{\C})
\end{equation*}
is a proper map. Observe that if $\alpha + \beta$ is equidimensional then both $\alpha$ and $\beta$ are equidimensional and therefore
\begin{equation*}
 \xymatrix{
\mcal{C}_{r}(Y_{\C})(U_{\C})\times\mcal{C}_{r}(Y_{\C})(U_{\C}) \ar[r]^-{+}\ar[d] & \mcal{C}_{r}(Y_{\C})(U_{\C}) \ar[d]\\
 \mcal{C}_{r+k}(U_{\C}\times Y_{\C})\times  \mcal{C}_{r+k}(U_{\C}\times Y_{\C}) \ar[r]^-{+} &  \mcal{C}_{r+k}(U_{\C}\times Y_{\C})
}
\end{equation*}
is a pull-back square. Since addition is a proper map on effective cycles we see that it is a proper map for effective cocycles as well.



\item The map
\begin{equation*}
 \frac{\mcal{C}_{r}(Y_{\C})(U_{\C})^{G}}{\mcal{C}_{r}(Y'_{\C})(U_{\C})^{G}}\xrightarrow{} \left(\frac{\mcal{C}_{r}(Y_{\C})(U_{\C})}{\mcal{C}_{r}(Y'_{\C})(U_{\C})}\right)^{G}
\end{equation*}
is easily seen to be a continuous bijection by an argument similar to the one used in Lemma \ref{ccom}.


\item Finally since $\mcal{C}_{r}(Y_{\C})(U_{\C}) \to \mcal{C}_{r}(Y_{\C})(U_{\C})/\mcal{C}_{r}(Y'_{\C})(U_{\C})$ is a closed map by Lemma \ref{qmcl} we conclude that the continuous bijection
\begin{equation*}
 \frac{\mcal{C}_{r}(Y_{\C})(U_{\C})^{G}}{\mcal{C}_{r}(Y'_{\C})(U_{\C})^{G}}\xrightarrow{} \left(\frac{\mcal{C}_{r}(Y_{\C})(U_{\C})}{\mcal{C}_{r}(Y'_{\C})(U_{\C})}\right)^{G}
\end{equation*}
is a closed map and therefore a topological isomorphism.
\end{enumerate}

The second statement for average cocycles is proved in a similar fashion, using Proposition \ref{avl} and that $C _r(Y _{\C})(U _{\C})^{av}\subseteq C _r(Y_{\C})(U _{\C})$ is closed.

\end{proof}


As with the topological group of cycles on a real variety $X$ we may view the topological group of real cocycles as the topological group of cycles on the complexification which are fixed by the Galois action. 
\begin{proposition} \label{nqp}
Let $X$ be a normal quasi-projective real variety. Then
\begin{equation*}
\mcal{Z}^{q}(X) = \mcal{Z}^{q}(X_{\C})^{G}.
\end{equation*}
\end{proposition}
\begin{proof}
This follows from Proposition \ref{appcycagr} together with the previous proposition since 
$\Mor{\R}{X}{\mcal{C}_0(\P^q_{\R})} = \Mor{\C}{X_{\C}}{\mcal{C}_0(\P^q_{\C})}^{G}$. 
\end{proof}
\begin{remark}
 The space $\mcal{Z}^{q}(X_{\C})$ has the equivariant homotopy type of a $G$-$CW$-complex (see Corollary \ref{ehtype}).
\end{remark}

When $W=X_{\C}$ is the complexification of a quasi-projective real variety and $Z = Y_{\C}$ is the complexification of a projective real variety the graph map is an equivariant morphism. In particular the duality map
\begin{equation*}
 \mcal{D}:\mcal{Z}^{q}(X_{\C})\to \mcal{Z}_{d}(X_{\C}\times \A^{q}_{\C})
\end{equation*}
is an equivariant continuous map. 

\begin{definition}
\begin{enumerate}
\item Let $X$ be a projective real variety. Define the topological group of \textit{averaged cocycles} to be 
\begin{equation*}
\mcal{Z}^{q}(X_{\C})^{av} = \{f + \sigma \cdot f | f\in \mcal{Z}^{q}(X_{\C})\}\subseteq \mcal{Z}^{q}(X_{\C}).
\end{equation*}

\item Let $X$ be a normal projective real variety. Define the topological group of \textit{reduced cocycles} to be the quotient topological group
\begin{equation*}
 \mcal{R}^{q}(X) = \frac{\mcal{Z}^{q}(X_{\C})^{G}}{\mcal{Z}^{q}(X_{\C})^{av}} .
\end{equation*}
Lemma \ref{avl} shows that $\mcal{R}^{q}(X)$ is a Hausdorff topological group.
\end{enumerate}
\end{definition}
In \cite{Teh:real} Teh defines 
$$ 
\mcal{R}_{0}(Y)(X)=\frac{\Mor{\C}{X_{\C}}{\mcal{C}_{0}(Y_{\C})}^{+,G}}{ \Mor{\C}{X_{\C}}{\mcal{C}_{0}(Y_{\C})}^{+,av}}
$$ 
and defines the reduced cocycles are defined as 
\begin{equation*}
\mcal{R}^{q}(X) = \frac{\mcal{R}_{0}(\P^{q})(X)}{\mcal{R}_{0}(\P^{q-1})(X)}
\end{equation*} 
for any real normal projective variety X and real projective variety Y.
 
By Proposition \ref{Gcocomm} this definition and the one above give isomorphic topological groups.

\begin{lemma}\label{avl}(c.f. \cite[Proposition 2.4]{Teh:real})
Let $X$ be a real projective variety. The subset of averaged cocycles $\mcal{Z}^{q}(X_{\C})^{av} \subseteq \mcal{Z}^{q}(X_{\C})$ is a closed subgroup. 
\end{lemma}
\begin{proof}
Write $\overline{f}$ for $\sigma\cdot f$ and $\overline{V}$ for $\sigma\cdot V$. 

Suppose that $\{[f_n]+\overline{[f_{n}]}\}$ is a sequence in $\mcal{Z}^{q}(X_{\C})^{av}$ which converges in $\mcal{Z}^{q}(X_{\C})$. Write $[\gamma] = \lim_{n\to \infty} [f_{n}] + \overline{[f_{n}]}$ for its limit. We need to conclude that $[\gamma]$ is an averaged cocycle. 

The set $\{[f_n] + [\overline{f_n}]\} \cup \{[\gamma]\}\subseteq \mcal{Z}^{q}(X_{\C})$ is compact. Applying the duality map to this set yields the compact subset 
\begin{equation*}
\{\Gamma([f_n]) + \Gamma([\overline{f_n}])\} \cup \{\Gamma([\gamma])\} \subseteq 
\mcal{Z}_{d}( X_{\C}\times \A^{q}_{\C}). 
\end{equation*}
Since this is a compact subset  it lies in  $\mcal{Z}_{d,\leq k}(X_{\C}\times 
\A^{q}_{\C})$ for some $k$. 

The sequence $\{[g_{n}]\}\subseteq \mcal{Z}_{d,\leq k}(X_{\C}\times 
\A^{q}_{\C})$ has a convergent subsequence.  Write $\{[g_{n_{i}}] \}$ for this convergent subsequence and write $\lim_{n_{i}\to \infty} [g_{n_{i}}]= [g] \in \mcal{Z}_{d,\leq k}(X_{\C}\times \A^{q}_{\C})$ for its limit. Note that $[g]$ satisfies $[g] +\overline{[g]} = \Gamma(\gamma)$. Since $\Gamma$ is injective and its image consists precisely of equidimensional cycles, we are done if we can find an equidimensional cycle $[g']$ such that $[g'] + \overline{[g']} = [g] + \overline{[g]}$.


Choose a representative $\gamma\in \Mor{\C}{X_{\C}}{\mcal{C}_0(\P^q_{\C})}^{+}$ of 
$[\gamma]$. Choose a representative $g= \sum n_V V\in \mcal{Z}_{d}(X_{\C}\times\P^{q}_{\C})$ of $[g]$ such that if $n_V \neq 0$ 
then $V\nsubseteq\P^{q-1}\times X$. Since $[g] + \overline{[g]} = \Gamma([\gamma]) \in 
\mcal{Z}_{d}(X_{\C}\times \A_{\C}^{q})$ we see that $g + \overline{g} = \sum (n_V + 
n_{\overline{V}} )V = \Gamma(\gamma) + h$ where $h\in \mcal{Z}_{d}(X_{\C}\times \P_{\C}^{q-1})$. Write $h 
= \Sigma m_{W} W$. Since  $V\nsubseteq \P^{q-1}$ whenever $n_{V}\neq 0$ we see that if 
$m_{W} \neq 0$ then a term of $-m_{W}W$ must appear in $\Gamma(\gamma)$. In particular $h$ is equidimensional. Consequently $g + \overline{g}$ is equidimensional.

If $n_{V} + n_{\overline{V}} \neq 0$ then $V$ is equidimensional. Define 
\begin{equation*}
g' = \sum_{n_{V} + n_{\overline{V}}\neq 0} n_{V} V.
\end{equation*}
 Since $g'$ is an equidimensional cycle there is an $f\in (\Mor{\C}{X_{\C}}{\mcal{C}_0(\P^q_{\C})}^{+})^{G}$ 
such that $\Gamma(f) = g'$. Since $\Gamma([f] + \overline{[f]}) = [g'] + \overline{[g']} = [g] + \overline{[g]} = \Gamma([c])$ and $\Gamma$ is injective, we conclude that $[c] = [f] + \overline{[f]}$.
\end{proof}

A continuous algebraic map $f:W\to V$ between two complex varieties induces a continuous map $f:W(\C)\to V(\C)$. Friedlander-Lawson \cite[Proposition 4.1]{FL:algco} show that this defines a continuous map
\begin{equation}\label{morphcomp}
\Phi:\mcal{Z}^{r}(W) \to \map{W(\C)}{\mcal{Z}_{0}(\A^{r}_{\C})},
\end{equation}
where the mapping space between two topological spaces is given with the compact-open topology. If $Y$ is a real variety this provides a continuous equivariant comparison map
\begin{equation*}
 \Phi:\mcal{Z}^{r}(Y_{\C}) \to \map{Y_{\C}(\C)}{\mcal{Z}_{0}(\A^{r}_{\C})}
\end{equation*}
of topological abelian groups.

\begin{definition}(Real Morphic Cohomology)
Friedlander-Walker \cite{FW:real} define real morphic cohomology of a quasi-projective real variety by 
\begin{equation*}
 L^qH\R^{n}(X)=\pi_{2q-n}\mcal{Z}^{q}(X)
\end{equation*}
for $2q-n\geq 0$.
\end{definition}

 We will be using  an equivariant extension of their theory for normal quasi-projective real varieties defined below.
 \begin{definition}(Equivariant Morphic Cohomology) 
 \label{emco}
 Let $X$ be a normal quasi-projective variety. Then the equivariant morphic cohomology is (in equivariant homotopy indexing notation)
$$L^{q}H\R^{k,r}(X)=\pi_{q-k,q-r}\mcal{Z}^{q}(X_{\C}),$$
for $q-k, q-r \geq 0$.
\end{definition} 
By Proposition \ref{nqp} we see that 
$$L^qH\R^{n}(X)=\pi_{2q-n}\mcal{Z}^{q}(X_{\C})^{G}$$
so Friedlander-Walker's real morphic cohomology groups are a part of the equivariant morphic cohomology, $L^{q}H\R^{q-r,q}(X)=\pi_{r,0}\mcal{Z}^{q}(X_{\C})=L^{q}H\R^{2q-r}(X)$.

In \cite{DS:real} dos Santos defines real Lawson homology.
\begin{definition}(Real Lawson Homology) For any quasi-projective real variety $X$, the real Lawson homology is defined by
$$L_qH\R_{n,m}(X) = \pi_{n-q,m-q}\mcal{Z}_q(X_{\C}),$$
for $n-q,m-q\geq 0$. 
\end{definition}

\begin{definition}
Let $X$ be a quasi-projective real variety. The following definitions are taken from \cite{LLM:real}.
\begin{enumerate}

\item Define the \textit{space of averaged cycles} $\mcal{Z}_q(X)^{av}$ to be
\begin{equation*}
 \mcal{Z}_q(X_{\C})^{av} =\im(N) \subseteq \mcal{Z}_q(X_{\C})^{G},
\end{equation*}
so  $\mcal{Z}_q(X_{\C})^{av} \subseteq \mcal{Z}_q(X_{\C})^{G}$ is the subgroup generated by cycles of the form $Z + \overline{Z}$ and given the subspace topology.  By Remark \ref{addqp} this is a closed subgroup. Here $N:  \mcal{Z}_q(X_{\C}) \to \mcal{Z}_q(X_{\C})$ is defined by $N(Z)=Z+\overline{Z}$.

\item Define the \textit{space of reduced cycles}  $\mcal{R}_q(X)$ to be the quotient group
\begin{equation*}
 \mcal{R}_q(X) = \frac{\mcal{Z}_q(X_{\C})^{G}}{\mcal{Z}_q(X_{\C})^{av}}.
\end{equation*}
\end{enumerate}
\end{definition}

\begin{remark}
 These spaces all have the homotopy type of a $CW$-complex (see Corollary \ref{ehtype}).
\end{remark}

Teh \cite{Teh:real} defines the \textit{reduced real Lawson homology} of $X$ to be
\begin{equation*}
 RL_qH_{n}(X) = \pi_{n-q}\mcal{R}_q(X),
\end{equation*}
for $n\geq q$. According to Lemma \ref{ccom} this definition coincides with the definition given in \cite{Teh:real} in the case of a quasi-projective variety.

\begin{example}
Let $X$ be a projective real variety.
\begin{enumerate}
\item \cite[Lemma 8.4]{LLM:real} The space of averaged zero-cycles computes the singular homology of the quotient of analytic space of complex points
 \begin{equation*}
\pi_k \mcal{Z}_{0}(X_{\C})^{av} = H_k(X(\C)/G;\Z) .
\end{equation*}
\item By the equivariant Dold-Thom theorem \cite{DS:equiDT} the space of fixed zero-cycles computes (a portion of) Bredon cohomology 
\begin{equation*}
\pi_k\mcal{Z}_{0}(X_{\C})^{G} = H_{k,0}(X(\C);\underline{\Z}) .
\end{equation*} 

\item \cite[Proposition 2.7]{Teh:real} The space of reduced real cycles computes the singular homology with $\Z/2$ coefficients of the analytic space of real points
\begin{equation*}
\pi_{k}\mcal{R}_{0}(X) = H_{k}(X(\R);\Z/2).
\end{equation*}
\end{enumerate}
\end{example}

\section{Poincare Duality}\label{realdual}

In this section we use the duality  for bivariant cycle homology in \cite{FV:biv} to establish a duality between Lawson homology of a real variety and real morphic cohomology. This together with the 
duality between Lawson homology and morphic cohomology \cite{FL:dual} gives an 
equivariant duality between the algebraic cocycle spaces and algebraic cycle spaces for the complexification of a real variety. 

The material and methods used here closely parallel  \cite[Section 3]{FW:ratisos} where Friedlander-Walker  reformulate Lawson homology and morphic cohomology for complex varieties.

\subsection{Recognition Principle}

Let $F(-)$ be a presheaf sets (respectively simplicial sets, or abelian groups) on $Sch/\R$. If $T$ is a topological space then define $F(T)$ by the filtered colimit
\begin{equation*}
 F(T) = \colim_{T\to V(\R)} F(V).
\end{equation*}

In particular we obtain a simplicial set (respectively a bisimplicial set, or simplicial abelian group) by 
\begin{equation}
 d\mapsto F(\Delta^{d}_{top}).
\end{equation}

We record an analogue of the recognition principle \cite[Theorem 2.3]{FW:ratisos} which is needed to move the duality for bivariant cycle homology to a duality for real Lawson homology and morphic cohomology. 
Friedlander-Walker's proof in the complex case uses the $uad$-topology which is essentially due to Deligne. 
\begin{defn}
\begin{enumerate}
\item A continuous map of topological spaces $f: S\to T$ is said to satisfy \textit{cohomological descent} if for any sheaf $A$ of abelian groups on $T$ the natural map 
\begin{equation*}
 H^*(T, A) \to H^{*}(N_T(S), f^*A)
\end{equation*}
is an isomorphism. Here $N_T(S)\to T$ is the Cech nerve of $f$, i.e. $N_T(S)$ is the simplicial space which in degree $n$ is the $n+1$-fold fiber product of $S$ over $T$.
A map $f:S\to T$ is said to be of \textit{universal cohomological descent} provided the pullback $S\times_{T}T'\to T'$ along any continuous map $T'\to T$ is again of cohomological descent.

\item The $uad$-topology on $Sch_{\R}$ is the Grothendieck topology associated to the pretopology generated by collections $\{U_i\to X\}$ such that $\coprod U_i(\R)^{an}\to X(\R)^{an}$ is a surjective map of universal cohomological descent.
\end{enumerate}
\end{defn}
\begin{example}
\begin{enumerate}
\item A  proper and surjective map of real varieties $X\to Y$ which induces a surjective map of real points is a $uad$-cover. Indeed, in this case $X(\R)^{an}\to Y(\R)^{an}$ is a proper surjective map of topological spaces, and therefore is a map of universal cohomological descent (see \cite[5.3.5]{Deligne:hodgeIII}).
\item Any Nisnevich cover is a $uad$-cover. Any $cdh$-cover is a $uad$-cover. In particular every real variety $X$ is locally smooth in the $uad$ topology because resolution of singularities implies there is a $cdh$-cover $X' \to X$, with $X'$ smooth.
\item Unlike the complex case not every etale-cover is a $uad$-cover (e.g. $\spec \C \to \spec \R$ is an etale cover but not a $uad$-cover).
\end{enumerate}
\end{example}
Here is the recognition principle.
\begin{theorem}[{\cite[Theorem 2.2]{FW:ratisos}}]\label{recog}
Suppose that $F\to G$ is a natural transformation of presheaves of abelian groups on $Sch_{\R}$. If $F_{uad} \xrightarrow{} G_{uad}$ is an isomorphism of $uad$-sheaves, then 
\begin{equation*}
 F(\Delta^{\bullet}_{top}) \to G(\Delta^{\bullet}_{top})
\end{equation*}
 is a homotopy equivalence of simplicial abelian groups.
\end{theorem}
\begin{proof}
Friedlander-Walker's proof given in \cite{FW:ratisos} works by changing the space $X(\C)^{an}$ associated with a complex variety with the space $Y(\R)^{an}$ associated to a real variety together with the fact that $Y(\R)^{an}$ may be triangulated.
\end{proof}

\begin{corollary} \label{saqiso}

Suppose that $f:F \to G$ is a map of presheaves of simplicial abelian 
 groups such that  $F(V) \to G(V)$ is a homotopy equivalence for any smooth $V$. Then the map of simplicial abelian groups
$\diag F(\Delta^{\bullet}_{top}) \to \diag G(\Delta^{\bullet}_{top})$ is a homotopy equivalence. 
\end{corollary}

\subsection{Poincare Duality}
Let $X$ be a variety over a field $k$ of characteristic zero. Recall the presheaf $z_{equi}(X,r)(-)$ of equidimensional $r$-cycles. This is the unique $qfh$-sheaf on $Sch/k$ such that for a normal variety $U$ the group $z_{equi}(X,r)(U)$ is the free abelian group on closed, irreducible subvarieties $V\subseteq U\times_{k} X$ which are equidimensional of relative dimension $r$ over some irreducible component of $U$.

If $X$ and $Y$ are real varieties then $G=Gal(\C/\R)$ acts on the group $z_{equi}(X_{\C}, r)(U_{\C})$ by $\sigma\cdot [V\subseteq U_{\C}\times_{\C} X_{\C}] = [\sigma V \subseteq U_{\C}\times_{\C} X_{\C}]$.

 \begin{lemma}
Let $X$ and $U$ be real varieties. Then
\begin{equation*}
z_{equi}(X, r)(U)\xrightarrow{\pi^{*}}(z_{equi}(X_{\C}, r)(U_{\C}))^{G}
\end{equation*}
is a natural isomorphism where $\pi:(U\times_{\R} X)_{\C} \to U\times_{\R} X$.
\end{lemma}
\begin{proof}
It suffices to check this for $U$ normal, since normalization is a $qfh$-cover. By \cite[Lemma 2.3.2]{SV:rel} $\pi^*:Cycl(U\times X) \to Cycl((U\times X)_{\C})^{G}$ is an isomorphism, where $Cycl(W)$ denotes the group of cycles on $W$. 
We are done if we see that $f:V\to U$ is equidimensional if and only if $\tilde{f}:V_{\C}\to U_{\C}$ is equidimensional. By \cite[Proposition 13.3.8]{EGAIVpt3} if $f$ is equidimensional then so is $\tilde{f}$. Suppose that $\tilde{f}:V_{\C}\to U_{\C}$ is equidimensional. Since  $U_{\C}$ is normal,  $\tilde{f}:V_{\C}\to U_{\C}$ is an open mapping and for all $v'\in V$ the local rings $\mcal{O}_{V_{\C},v'}$ are equidimensional by \cite[Corollaire 14.4.6]{EGAIVpt3}. By \cite[Corollaire 2.6.4, Proposition 7.1.3]{EGAIVpt2} the map $f:V\to U$ is open and $\mcal{O}_{V,v}$ is equidimensional for all $v\in V$ since $U_{\C}\to U$ is faithfully flat and therefore $f$ is equidimensional.
\end{proof}

In the proof of \cite[Proposition 2.4]{FW:real} it is shown that for any presheaf $F(-)$ of sets on $Sch/\C$ and any topological space $T$ the natural map
\begin{equation*}
 \colim_{T\to V(\R)} F(V_{\C}) \xrightarrow{\iso} \colim_{T\to U(\C)} F(U)
\end{equation*}
is an isomorphism.  In the first indexing set $V$ ranges over real varieties and in the second  $U$ ranges over complex varieties.

In particular $ z_{equi}(X_{\C}, r)(Y_{\C}\times_{\C} T)$ may be computed via the filtered colimit
\begin{equation*}
 z_{equi}(X_{\C}, r)(Y_{\C}\times_{\C} T) = \colim_{T\to V(\R)}z_{equi}(X_{\C}, r)(Y_{\C}\times_{\C} V_{\C}),
\end{equation*}
which equips $z_{equi}(X_{\C},r)(Y_{\C}\times_{\C}T)$ with an action of $G$. 
Filtered colimits  commute with  fixed points and so 
\begin{equation*}
 z_{equi}(X,r)(Y\times_{\R} T) \to (z_{equi}(X_{\C}, r)(Y_{\C}\times_{\C} T))^{G} = \colim_{T\to V(\R)}(z_{equi}(X_{\C}, r)(Y_{\C}\times_{\C} V_{\C}))^{G}
\end{equation*}
is an isomorphism.

For X projective we have the natural isomorphism of presheaves (in fact of $qfh$-sheaves) of abelian groups on $Sch_{\R}$ (see \cite[Lemma 4.4.14]{SV:rel}
\begin{equation*}
z_{equi}(X,r)(-) \iso \Mor{\R}{-}{\mcal{C}_r(X)}^{+}.
\end{equation*}

The following is the real analogue of \cite[Proposition 3.1]{FW:ratisos}.
\begin{proposition}\label{sing}
Let $T$ be a compactly generated Hausdorff topological space and $X$ a quasi-projective real 
variety. There is a natural map of abelian groups
\begin{equation*}
 z_{equi}(X,r)(T) \to \Hom{cts}{T}{\mcal{Z}_{r}(X)}
\end{equation*}
given by sending $(f:T\to U(\R),\alpha\in z_{equi}(X,r)(U))$ to the function $t\mapsto \alpha|_{f(t)}$.

This map is contravariant for continuous maps of compactly-generated Hausdorff spaces $T'\to 
T$, covariant for proper maps $X\to X'$ and contravariant for flat maps $X'\to X$ (with a shift 
in dimension).

When $X$ is a projective real variety the induced map of simplicial abelian groups
\begin{equation*}
 z_{equi}(X,r)(\Delta^{\bullet}_{top}) \to \sing_{\bullet}\mcal{Z}_r(X)
\end{equation*}
is the natural homotopy equivalence
\begin{equation*}
 \left[ \sing_{\bullet}(\mcal{C}_{r}(X)^{an})\right]^{+} \xrightarrow{\wkeq} \sing_{\bullet}\mcal{Z}_{r}(X) .
\end{equation*}

More generally, for a quasi-projective real variety $X$ with projectivization $X\subset \overline{X}$ this map fits into a comparison of homotopy fiber sequences
\begin{equation} \label{tdgrm}
 \xymatrix{
z_{equi}(\overline{X}\sm X, r)(\Delta^{\bullet}_{top}) \ar[d]\ar[r] & z_{equi}(\overline{X},r)(\Delta_{top}^{\bullet}) \ar[r]\ar[d] & z_{equi}(X, r)(\Delta^{\bullet}_{top}) \ar[d] \\
\sing_{\bullet}\mcal{Z}_{r}(\overline{X}\sm X) \ar[r] & \sing_{\bullet}\mcal{Z}_{r}(\overline{X}) \ar[r] & \sing_{\bullet}\mcal{Z}_r(X).
}
\end{equation}
Therefore the map 
\begin{equation*}
 z_{equi}(X,r)(\Delta_{top}^{\bullet}) \to \sing_{\bullet}\mcal{Z}_r(X)
\end{equation*}
is a natural weak equivalence for any quasi-projective real variety $X$.
\end{proposition}
\begin{proof}
The map 
\begin{equation*}
z_{equi}(X_{\C},r)( T)=\colim_{T\to W(\C)} z_{equi}(X_{\C}, r)(W) \to \Hom{cts}{T}{\mcal{Z}_{r}(X_{\C})}
\end{equation*}
 given sending $(f:T\to W(\C), \alpha\in z_{equi}(X_{\C}, r)(W))$ to the function $t\mapsto \alpha_{|{f(t)}}$ is shown to be well-defined in \cite[Proposition 3.1]{FW:ratisos} and to satisfy the stated naturality properties. Observe that if $W=V_{\C}$ is the complexification of a real variety then $\overline{\alpha_{|{f(t)}}} =\overline{\alpha}_{|{\overline{f(t)}}}$. Therefore composing with the natural isomorphism
\begin{equation*}
 \colim_{T\to V(\R)}z_{equi}(X_{\C},r)(V_{\C}) \xrightarrow{\iso} \colim_{T\to W(\C)}z_{equi}(X_{\C},r)(W)
\end{equation*}
gives a well-defined equivariant map
\begin{equation*}
 z_{equi}(X_{\C},r)(T) \to \Hom{cts}{T}{\mcal{Z}_{r}(X_{\C})}
\end{equation*}
which by taking fixed points induces the map
\begin{equation*}
 z_{equi}(X,r)(T)= z_{equi}(X_{\C},r)(T)^{G} \to \Hom{cts}{T}{\mcal{Z}_{r}(X_{\C})}^{G} = \Hom{cts}{T}{\mcal{Z}_{r}(X)},
\end{equation*}
which is the map of the proposition and satisfies the stated naturality properties.

When $X$ is a projective real variety and $T$ is a compact Hausdorff space, the map $z_{equi}^{eff}(X_{\C},r)(T) \to \Hom{cts}{T}{\mcal{C}_r(X_{\C})}$ is an isomorphism by \cite[Corollary 4.3]{FW:sstfct}. Since this is an equivariant map, taking fixed points yields the isomorphism of monoids
\begin{equation*}
 z_{equi}^{eff}(X,r)(T) \xrightarrow{\iso} \Hom{cts}{T}{\mcal{C}_r(X)}.
\end{equation*}

Therefore the map
\begin{equation*}
z_{equi}(X, r)({\Delta^{\bullet}_{top}})\xrightarrow{\iso} [\Hom{cts}{\Delta^{\bullet}_{top}}{\mcal{C}_{r}(X)^{an}}]^{+} \to \sing_{\bullet}\mcal{Z}_{r}(X)
\end{equation*}
is a homotopy equivalence by Quillen's theorem \cite[App Q]{FM:filt} on homotopy group completions of simplicial abelian monoids.

Finally the diagram (\ref{tdgrm}) commutes by the naturality properties of the map $z_{equi}(X,r)(T) \to \Hom{cts}{T}{\mcal{Z}_{r}(X)}$ . By \cite[5.12,8.1]{FV:biv}, Proposition \ref{htpyinv} (homotopy invariance), and  Theorem \ref{recog} (recognition principle)  the upper row  of the diagram (\ref{tdgrm}) is a homotopy fiber sequence. Comparing the upper and lower homotopy fiber sequence  yields the final statement of the proposition. 

\end{proof}

\begin{proposition}\label{singmor}
For a quasi-projective real variety $U$, projective real variety $Y$, and compact Hausdorff space $T$,  there is a natural map of abelian groups
\begin{equation*}
 z_{equi}(Y,0)(U\times_{\R} T) \to \Hom{cts}{T}{\Mor{\R}{U}{\mcal{C}_{0}(Y)}^{+}}
\end{equation*}
given by sending $(f,\alpha)$ to the function $t\mapsto \alpha_{|{f(t)}}$. 
\end{proposition}
\begin{proof}
The map 
\begin{equation*}
z_{equi}(Y_{\C},0)(U_{\C}\times_{\C} T) \to \Hom{cts}{T}{\Mor{\C}{U_{\C}}{\mcal{C}_{0}(Y_{\C})}^{+}}
\end{equation*}
from \cite[Proposition 3.3]{FW:ratisos}  is equivariant and therefore taking fixed points induces the natural map of abelian groups
\begin{equation*}
z_{equi}(Y,0)(U\times_{\R} T) \to  \Hom{cts}{T}{\Mor{\R}{U}{\mcal{C}_{0}(Y)}^{+}}.
\end{equation*}
\end{proof}

Let $Y$ be a projective real variety and $U$ a normal quasi-projective real variety of dimension $d$ with projectivization $U\subseteq X$ and closed complement $X_{\infty}=X\setminus U$. Write $\mcal{E}_{r}(Y)(U)\subseteq \mcal{C}_{r+d}(Y\times_{\R} X)$ for the submonoid consisting of those cycles of dimension $r+d$ on $Y\times X$ whose restriction to $U$ is equidimensional of relative dimension $r$ over $U$. This is a constructable embedding. This can be seen by arguing as in \cite{F:algco} for the complex case. The subspace topology on this monoid agrees with the quotient topology $\mcal{C}_{r}(Y)(U) = \mcal{E}_{r}(Y)(U)/\mcal{C}_{r+d}(Y\times X_{\infty})$ by the same reasoning as in \cite[Proposition 1.8]{F:algco}. The topological group of equidimensional cycles is the naive groups completion $\mcal{Z}_{r}(Y)(U) = \mcal{C}_{r}(Y)(U)^{+}$. Since these are tractable monoids, they are related to equidimensional cocycles via the homotopy fiber sequence
\begin{equation*}
 \mcal{Z}_{r+d}(Y\times X_{\infty}) \to (\mcal{E}_{r}(Y)(U))^{+} \to \mcal{Z}_{r}(Y)(U).
\end{equation*}
%

Define the presheaf $e(U,Y,r)(-)$ to be the pull-back of presheaves
\begin{equation*}
 \xymatrix{
e(U,Y,r)(-) \ar@{^{(}->}[r]\ar[d] & z_{equi}^{eff}(Y\times X, r+d)(-) \ar[d]\\
z_{equi}^{eff}(Y, r)(U\times - ) \ar@{^{(}->}[r] & z^{eff}_{equi}(Y\times U, r+d)(-) .
}
\end{equation*}
We have for each quasi-projective real variety $V$ the short exact sequence of abelian groups
\begin{equation*}
 0\to z_{equi}(Y\times X_{\infty}, r+d)(V) \to (e(U,Y,r)(V))^{+} \to z_{equi}(Y,r)(U\times V) \to 0.
\end{equation*}

\begin{proposition}\label{eiso}
Let $Y$ be a projective real variety, $U$ a normal quasi-projective real variety, and $T$ a compact Hausdorff space. Then
\begin{equation*}
 e(U,Y,r)(T)\xrightarrow{\iso} \uphom{T}{\mcal{E}_{r}(Y)(U)}
\end{equation*}
is an isomorphism.
\end{proposition}
\begin{proof}
 Observe that if $V$ is a quasi-projective real variety then the isomorphism
$\Mor{\R}{V}{\mcal{C}_{r+d}(Y\times X)} \iso z_{equi}^{eff}(Y\times X,r+d)(V)$ restricts to give the isomorphism $\Mor{\R}{V}{\mcal{E}(Y)(U)}\iso e(U,Y,r)(V)$. Here if $E\subseteq W$ is a constructable subset then $\Mor{}{V}{E}\subseteq \Mor{}{V}{W}$ is the subset consisting of those continuous algebraic maps whose image is contained in $E$.

The isomorphism $e(U,Y,r)(T)\xrightarrow{\iso} \uphom{T}{\mcal{E}_{r}(Y)(U)}$ now follows as in \cite[Corollary 4.3]{FW:funcspc}.
\end{proof}

\begin{proposition}\label{smheq}
Let $U$ be a normal quasi-projective real variety and $Y$ a projective real variety. The map  of simplicial abelian groups from Proposition \ref{singmor}
\begin{equation*}
 z_{equi}(Y,0)(U\times \Delta^{\bullet}_{top}) \to \sing_{\bullet}(\Mor{\R}{U}{\mcal{C}_{0}(Y)}^{+})
\end{equation*}
is a homotopy equivalence.
\end{proposition}
\begin{proof}
%
By proposition \ref{eiso} we have $e(U,Y,r)(\Delta^{\bullet}_{top}) \iso \sing_{\bullet}\mcal{E}_{r}(Y)(U)$. Now by taking group completions,  tractability of the monoid $\mcal{E}_{r}(Y)(U)$ and Quillen's theorem \cite[App Q]{FM:filt} we conclude that 
\begin{equation*}
e(U,Y,r)(\Delta^{\bullet}_{top})^{+} \xrightarrow{\wkeq} \sing_{\bullet}(\mcal{E}_{r}(Y)(U)^{+})
\end{equation*}
is a homotopy equivalence. 

We conclude the proposition by comparing homotopy fiber sequences of simplicial abelian groups
\begin{equation*}
 \xymatrix{
z_{equi}(Y\times X_{\infty}, r+d)(\Delta^{\bullet}_{top}) \ar[r]\ar[d] & (e(U,Y,r)(\Delta^{\bullet}_{top}))^{+} \ar[r]\ar[d] & z_{equi}(Y,r)(U\times \Delta^{\bullet}_{top}) \ar[d]  \\
\sing_{\bullet}\mcal{Z}_{r+d}(Y\times X_{\infty}) \ar[r] & \sing_{\bullet}(\mcal{E}_{r}(Y)(U)^{+}) \ar[r] & \sing_{\bullet}\mcal{Z}^{r}(Y)(U) .
}
\end{equation*}
The left arrow is a homotopy equivalence by Proposition \ref{sing}, we have just seen that the middle map is a homotopy equivalence, the right horizontal maps induce a surjection on $\pi_{0}$ and so we conclude that $z_{equi}(Y,r)(U\times\Delta^{\bullet}_{top}) \to \sing_{\bullet}\mcal{Z}^{r}(Y)(U)$ is a homotopy equivalence.
\end{proof}

\begin{proposition}\label{htpyinv}
The presheaves $z_{equi}(X,r)(\Delta^{\bullet}_{top}\times -)$ are homotopy invariant in the sense that the map of complexes
\begin{equation*}
z_{equi}(X,r)(\Delta^{\bullet}_{top}) \to z_{equi}(X,r)(\Delta^{\bullet}_{top}\times_{\R} \Delta_{\R}^1) 
\end{equation*}
 is a quasi-isomorphism.
\end{proposition}
\begin{proof}
 The same argument as in \cite[Lemma 1.2]{FW:compK}.
\end{proof}

The duality theorem for bivariant cycle theory \cite[Theorem 7.4]{FV:biv} says that for real varieties $X$, $U$ with $U$ smooth of dimension $d$, the natural inclusion
\begin{equation}\label{FVdual}
 \mcal{D}:z_{equi}(X, r)(U\times_{\R} - ) \hookrightarrow z_{equi}(X\times_{\R} U, r+d)(-)
\end{equation}
induces a quasi-isomorphism of complexes
\begin{equation*}
 \mcal{D}:z_{equi}(X,r)(U\times_{\R} \Delta^{\bullet}_{\R}) \xrightarrow{\simeq} 
z_{equi}(X\times_{\R} U, r+d)(\Delta^{\bullet}_{\R}).
\end{equation*}

\begin{proposition}\label{fvsstdual}
 For a smooth real variety $U$ and a quasi-projective real variety $X$ the map
\begin{equation*}
z_{equi}(X,r)(U\times_{\R}\Delta^{\bullet}_{top}) \xrightarrow{\mcal{D}} z_{equi}(X\times_{\R} U, r+d)(\Delta^{\bullet}_{top}) 
\end{equation*}
is a quasi-isomorphism.
\end{proposition}
\begin{proof}
Consider the commutative diagram
\begin{equation*}
\begin{CD}
z_{equi}(X,r)(U\times_{\R}\Delta^{\bullet}_{top}) @>{\mcal{D}}>> z_{equi}(X\times_{\R} U, r+d)(\Delta^{\bullet}_{top}) \\
@V{\pi^*}VV @VV{\pi^{*}}V \\
 z_{equi}(X,r)(U\times_{\R} \Delta^{\bullet}_{\R}\times_{\R}\Delta^{\bullet}_{top}) @>{\mcal{D}}>> z_{equi}(X\times_{\R} U, r+d)(\Delta^{\bullet}_{\R}\times_{\R}\Delta^{\bullet}_{top}) .
\end{CD}
\end{equation*}
 The vertical arrows are quasi-isomorphisms by homotopy invariance. The bottom right arrow is a quasi-isomorphism by Corollary \ref{saqiso} since 
\begin{equation*}
 z_{equi}(X,r)(U\times_{\R} \Delta^{\bullet}_{\R}\times W) \to z_{equi}(X\times_{\R} U, r+d)(\Delta^{\bullet}_{\R}\times_{\R} W).
\end{equation*}
is a quasi-isomorphism for all smooth real varieties $W$ by \cite[Theorem 7.4]{FV:biv} and therefore the top horizontal map is a quasi-isomorphism as well.
\end{proof}

\begin{lemma}\label{Dcomp}
 Let $Y$ be a projective real variety and $U$ a smooth real variety. The following diagram commutes
\begin{equation*}
 \begin{CD}
  z_{equi}(Y,0)(U\times_{\R} \Delta^{\bullet}_{top}) @>{\mcal{D}}>> z_{equi}(U\times_{\R} Y,d)(\Delta^{\bullet}_{top}) \\
@VVV @VVV \\
\sing_{\bullet}\Mor{\R}{U}{\mcal{C}_{0}(Y)}^{+} @>{\mcal{D}}>> \sing_{\bullet}\mcal{Z}_{d}(U\times_{\R} Y)
 \end{CD}
\end{equation*}
where the vertical maps are the ones from Proposition \ref{singmor} and Proposition \ref{sing}.
\end{lemma}
\begin{proof}
 By \cite[Proposition 3.3]{FW:ratisos} the diagram of equivariant maps of simplicial sets
\begin{equation*}
 \begin{CD}
  z_{equi}(Y_{\C},0)(U_{\C}\times_{\C} \Delta^{\bullet}_{top}) @>>{\mcal{D}}> z_{equi}(U_{\C}\times_{\C} Y_{\C},d)(\Delta^{\bullet}_{top}) \\
@VVV @VVV \\
\sing_{\bullet}\Mor{\C}{U_{\C}}{\mcal{C}_{0}(Y_{\C})}^{+} @>>{\mcal{D}}> \sing_{\bullet}\mcal{Z}_{d}(U_{\C}\times_{\C} Y_{\C})
 \end{CD}
\end{equation*}
commutes. Taking fixed points yields the result.
\end{proof}

Write  
\begin{equation*}
z_{equi}(\P^{q/q-1}_{\R},0)(U) = coker ( z_{equi}(\P^{q-1}_{\R},0)(U)\to z_{equi}(\P^{q}_{\R},0)(U) )
\end{equation*}
for the cokernel of the map of presheaves induced by $\P^{q-1}_{\R}\subseteq \P^{q}_{\R}$.
\begin{proposition}\label{qisoseq}
 Let $U$ be a smooth real variety of dimension $d$. The sequence of natural maps of complexes below consist of quasi-isomorphisms.
\begin{multline}\label{morsst}
z_{equi}(\A^q_{\R}, 0)(U\times_{\R}\Delta^{\bullet}_{top})  \leftarrow z_{equi}(\P^{q/q-1}_{\R},0)(U\times_{\R}\Delta^{\bullet}_{top}) \to \\
\to  \frac{\sing_{\bullet}(\Mor{\R}{U}{\mcal{C}_0(\P^{q}_{\R}})^{+}}{\sing_{\bullet}(\Mor{\R}{U}{\mcal{C}_0(\P^{q-1}_{\R}})^{+}} 
\to \sing_{\bullet}\mcal{Z}^{q}(U).
\end{multline}
\end{proposition}
\begin{proof}
 That the first map of diagram (\ref{morsst}) is a quasi-isomorphism follows from consideration of the comparison diagram
\begin{equation*}
 \xymatrix@-1pc{
z(\P^{n-1}, r)(U\times\Delta^{\bullet}_{top}) \ar[d]^{\mcal{D}}\ar[r] & z(\P^{n},r)(U\times\Delta_{top}^{\bullet}) \ar[r]\ar[d]^{\mcal{D}} & z(\A^{n}, r)(U\times\Delta^{\bullet}_{top}) \ar[d]^{\mcal{D}} \\
z(\P^{n-1}\times U, r+d)(\Delta^{\bullet}_{top}) \ar[d]\ar[r] & z(\P^{n}\times U,r+d)(\Delta_{top}^{\bullet}) \ar[r]\ar[d] & z(\A^{n}\times U , r+d)(\Delta^{\bullet}_{top}) \ar[d] \\ 
\sing_{\bullet}\mcal{Z}_{r+d}(\P^{n-1}\times U) \ar[r] & \sing_{\bullet}\mcal{Z}_{r+d}(\P^{n}\times U) \ar[r] & \sing_{\bullet}\mcal{Z}_{r+d}(\A^{n}\times U).
}
\end{equation*}
The vertical arrows are all quasi-isomorphisms by Proposition \ref{fvsstdual} and by Proposition \ref{sing}. Because $\mcal{C}_{k}(V)$ is a tractable monoid, the bottom row is homotopy equivalent to a short exact sequence of simplicial abelian groups and therefore the top rows are as well. It follows immediately that the first arrow of diagram \ref{morsst} is a quasi-isomorphism. The second arrow of diagram (\ref{morsst}) is a quasi-isomorphism by Proposition \ref{smheq} and the last arrow of the diagram is a quasi-isomorphism because $\Mor{\R}{U}{\mcal{C}_{0}(\P_{\R}^{n})}$ is a tractable monoid.
\end{proof}

\begin{definition}
 If $k<0$ then define $\mcal{Z}_{k}(X)$ to be $\mcal{Z}_{0}(X\times \A^{-k})$. 
\end{definition}

We can now conclude the duality for real morphic cohomology and real Lawson homology. 

\begin{corollary}
 Let $U$ be a smooth real variety of dimension $d$. Then
\begin{equation*}
 \mcal{Z}^{q}(U_{\C})^{G} \xrightarrow{\mcal{D}} \mcal{Z}_{d}(\A^{q}_{\C}\times_{\C} 
U_{\C})^{G} \xleftarrow{\simeq} \mcal{Z}_{d-q}(U_{\C})^{G}
\end{equation*}
is a natural homotopy equivalence.

In particular it induces the natural isomorphism
\begin{equation*}
L^{q}H\R^{n}(U)\xrightarrow{\iso} L_{d-q}H\R_{d-n,d}(U). 
\end{equation*}
\end{corollary}
\begin{proof}
%
This follows from Proposition \ref{fvsstdual}, Lemma \ref{Dcomp}, Proposition \ref{sing}, Proposition \ref{qisoseq}, and homotopy invariance \cite[Proposition 4.15]{DS:real}. Indeed these show that the following diagram is commutative and the left hand maps are homotopy equivalences, 
\begin{equation*}
\xymatrix@-1pc{
z_{equi}(\A_{\R}^{q},0)(U\times\Delta^{\bullet}_{top}) \ar[d]^{\mcal{D}} & z_{equi}(\P_{\R}^{q/q-1},0)(U\times\Delta^{\bullet}_{top}) \ar[l]\ar[d]^{\mcal{D}}\ar[r] & \\
z_{equi}(\A^{q}_{\R}\times U, d)(\Delta^{\bullet}_{top}) & z_{equi}(\P^{q/q-1}_{\R}\times U, d)(\Delta^{\bullet}_{top}) \ar[l]\ar[r] & \\
\ar[r] & \frac{\sing_{\bullet}(\Mor{\R}{U}{\mcal{C}_0(\P^{q}_{\R}})^{+}}{\sing_{\bullet}(\Mor{\R}{U}{\mcal{C}_0(\P^{q-1}_{\R}})^{+}} \ar[r]\ar[d]^{\mcal{D}} & \sing_{\bullet}\mcal{Z}^{q}(U) \ar[d]^{\mcal{D}} \\
\ar[r] & \frac{\sing_{\bullet}\mcal{Z}_{d}(\P^{q}_{\R}\times U)}{\sing_{\bullet}\mcal{Z}_{d}(\P^{q-1}_{\R}\times U)} \ar[r] & \sing_{\bullet}\mcal{Z}_{d}(\A^{q}\times U) .
}
\end{equation*}
 Therefore the right hand map is also a homotopy equivalence
\end{proof}

 Combining Friedlander-Lawson's duality between Lawson homology and morphic cohomology over $\C$ 
and the duality over $\R$ immediately gives an equivariant duality theorem.

\begin{corollary} \label{PD}
 Let $U$ be a smooth real variety of dimension $d$. The sequence of maps 
\begin{equation*}
\mcal{Z}^{q}(U_{\C}) \to \mcal{Z}_{d}(U_{\C}\times_{\C} \A^{q}_{\C}) \leftarrow \mcal{Z}_{d-q}(X_{\C})
\end{equation*}
consists of $G$-equivariant homotopy equivalences. In particular
\begin{equation*}
L^{q}H\R^{n,m}(U)\xrightarrow{\iso} L_{d-q}H\R_{d-n,d-m}(U). 
\end{equation*}
for all smooth quasi-projective real varieties $U$.
\end{corollary}


\begin{remark}\label{fpa}
A smooth $G$-manifold $M$ equipped such that the action of $G$ on its tangent bundle makes it into a real $n$-bundle satisfies an equivariant Poincare duality,
\begin{equation*}
 \mcal{P}: H^{p,q}(M;\underline{\Z}) \xrightarrow{\iso} H_{n-p,n-q}(M;\underline{\Z}).
\end{equation*}
In a forthcoming paper we prove that the duality $\mcal{D}$ is compatible under the cycle maps with the duality $\mcal{P}$. 
\end{remark}

\section{Compatibility of Cycle Maps}
\subsection{Generalized cycle maps}
Let $X$ be a smooth real variety. The generalized cycle map relates motivic cohomology and etale cohomology,
\begin{equation*}
cyc:\H^{2q-k,q}_{\mcal{M}}(X;\Z/2) \to H^{2q-k}_{et}(X;\mu_{2}^{\otimes q}).
\end{equation*}
By \cite{Cox:real} the etale cohomology of a real variety is equal to the Borel equivariant cohomology of its space of complex points,
\begin{equation*}
H^{2q-k}_{et}(X;\mu_{2}^{\otimes q}) \iso H^{2q-k}_{\Z/2}(X_{\C}(\C);\Z/2).
\end{equation*}

On the other hand morphic cohomology and motivic cohomology agree with finite coefficients (see Proposition \ref{fcoeffagr}). 

Combining the generalized cycle map in morphic cohomology and the comparison map between Bredon and Borel equivariant cohomology
 \begin{equation*}
L^{q}H\R^{q-k,q}(X;\Z/2) \to H^{q-k,q}(X_\R(\C);\underline{\Z/2}) \to H^{2q-k}_{\Z/2}(X(\C);\Z/2)
\end{equation*}
together with the isomorphism
$ \H^{2q-k,q}_{\mcal{M}}(X;\Z/2) \xrightarrow{\iso}  L^{q}H\R^{q-k,q}(X;\Z/2)$
gives another  map

%
\begin{equation*}
\H^{2q-k,q}_{\mcal{M}}(X;\Z/2)\iso L^{q}H\R^{q-k,q}(X;\Z/2)\to  H^{2q-k}_{\Z/2}(X_{\C}(\C);\Z/2)\iso H^{2q-k}_{et}(X;\mu_{2}^{\otimes q}).
\end{equation*}
In this section we verify that these two potentially different cycle maps are equal and we explore a few consequences. In particular this allows compatibility of cycle maps allows us to conclude that $L^{q}H\R^{q-p,q}(X;\Z/2^{k}) \to H^{q-p,q}(X_{\R}(\C);\underline{\Z/2^{k}})$ is an isomorphism for $p\geq q$ and for any smooth $X$. 

Before continuing, we show that motivic cohomology and morphic cohomology for real varieties agree with finite coefficients. This is a well-known to the experts, but because of the lack of a good reference we prove it below.
\begin{proposition}\label{fcoeffagr}
Let $X$ be a smooth real variety. Then for any $n>0$
\begin{equation*}
 \H^{2q-k,q}_{\mcal{M}}(X;\Z/n) \xrightarrow{\iso}  L^{q}H\R^{q-k,q}(X;\Z/n).
\end{equation*}
\end{proposition}
\begin{proof}
We show that the natural map of simplicial abelian groups 
$$
z_{equi}(\A^{q},0)(X\times\Delta^{\bullet}_{\R})\otimes\Z/n \to z_{equi}(\A^{q},0)(X\times\Delta^{\bullet}_{\R}\times\Delta^{\bullet}_{top})\otimes\Z/n
$$
is a quasi-isomorphism which implies the result by Proposition \ref{smheq} and Proposition \ref{htpyinv}.
Write $F(U)$ for the presheaf 
$$
U\mapsto \pi_{k}(z_{equi}(\A^{q},0)(X\times\Delta^{\bullet}_{\R}\times U)\otimes\Z/n)
$$
 on $Sch/\R$ and $F_{0}(U)$ for the constant presheaf 
$$
U\mapsto \pi_{k}(z_{equi}(\A^{q},0)(X\times\Delta^{\bullet}_{\R})\otimes\Z/n).
$$ 

Restricted to $Sm/\R$ these are homotopy invariant presheaves with transfers. Recall \cite[Lemma 3.8]{FW:real} that if $F(-)$ is a homotopy invariant presheaf with transfers, $Y$ is smooth, and $y\in Y(\R)$ then $F(\spec\mcal{O}_{Y,y}^{h}) \to F(\R)$ is an isomorphism, where $\mcal{O}_{Y,y}^{h}$ is the Henselization of the local ring $\mcal{O}_{Y,y}$. 

Let $H(-)$ denote either the kernel or the cokernel of the natural transformation $F_{0}(-)\to F(-)$. Let $Y$ be a quasi-projective real variety and $\gamma\in H(Y)$. Let $\tilde{Y}\to Y$ be a $cdh$-cover with $\tilde{Y}$ smooth (in particular it is a $uad$-cover). Since $H(\spec\mcal{O}_{\tilde{Y},y}^{h}) = 0$ for any $y\in \tilde{Y}(\R)$ there are finitely many etale maps $\tilde{Y}_{k}\to \tilde{Y}$ such that $\gamma|_{\tilde{Y}_{k}} = 0$ and $\coprod \tilde{Y}_{k} \to \tilde{Y}$ is a $uad$-cover. 

Therefore $H_{uad} = 0$ and  $(F_{0})_{uad} \to F_{uad}$ is an isomorphism. By Theorem \ref{recog} we conclude that 
$$
\pi_{k}(z_{equi}(\A_{\R}^{q},0)(X\times\Delta^{\bullet}_{\R})\otimes\Z/n) \to \pi_{k}(z_{equi}(\A^{q}_{\R},0)(X\times \Delta^{\bullet}_{\R}\times \Delta^{\bullet}_{top})\otimes\Z/n)
$$
 is an isomorphism. 

An application of the Bousfield-Friedlander spectral sequence finishes the proof.
\end{proof}

Friedlander-Walker introduce in \cite{FW:real} the equivalence relation of \textit{real algebraic equivalence}. Briefly two cycles $\alpha$, $\beta$ on a real variety $X$ are real algebraically equivalent provided there is a smooth real curve $C$, two real points $c_{0}$, $c_{1}$ in the same analytic connected component of $C(\R)$, and a cycle $\gamma$ on $X\times C$ such that $\alpha=\gamma|_{c_{0}}$ and $\beta = \gamma|_{c_{1}}$. Since $L^{q}H\R^{q,q}(X)$ is the group of codimension $q$ cycles on $X$ modulo real algebraic equivalence we obtain the following corollary.

\begin{corollary}
 Let $X$ be a smooth real variety and $0\leq r\leq dim(X)$. Rational equivalence and real algebraic equivalence yield the same equivalence relation on the group of $r-$cycles on $X$ with finite coefficients.
\end{corollary}

Recall that $z_{equi}(\P^{q/q-1},0)(U) = z_{equi}(\P^{q},0)(U)/z_{equi}(\P^{q-1},0)(U)$. Write
\begin{align*}
\Z/2(q)(X) & = (z_{equi}(\P_{\R}^{q/q-1},0)(X\times_{\R}\Delta^{\bullet}_{\R})\otimes\Z/2)[-2q] \\
\Z/2(q)^{sst}(X)   & = \sing_{\bullet} (\mcal{Z}^{q}/2(X_{\C})^{G})[-2q] \\
\Z/2(q)^{top}(X)  & = \Hom{cts}{X_{\C}(\C)\times\Delta^{\bullet}_{top}}{\mcal{Z}/2_{0}(\A^{q}_{\C})}^{G}[-2q] \\
\Z/2(q)^{Bor}(X) & = \Hom{cts}{X_{\C}(\C)\times EG\times\Delta^{\bullet}_{top}}{\mcal{Z}/2_{0}(\A^{q}_{\C})}^{G}[-2q]
\end{align*}
where we identify a simplicial abelian group with its associated bounded above cochain complex. These form presheaves of cochain complexes  on $Sm/\R$. These chain complexes compute respectively motivic cohomology, real morphic cohomology, Bredon cohomology, and Borel cohomology. Note  that $\Z/2(q)$, $\Z/2(q)^{top}$ and $\Z/2(q)^{Bor}$ are in fact complexes of etale sheaves on $(Sm/\R)$.

There are natural maps between these complexes,
\begin{equation*}
\Z/2(q)(X)\xrightarrow{\rho} \Z/2(q)^{sst}(X)\xrightarrow{\Phi} \Z/2(q)^{top}(X) \xrightarrow{\psi} \Z/2(q)^{Bor}(X)
\end{equation*}
obtained as follows. 
From Proposition \ref{singmor} and the projection $\Delta^{\bullet}_{\R} \to \spec \R$ we obtain
$$
z_{equi}(\P^{q/q-1},0)(X\times_{\R}\Delta^{\bullet}_{\R}) \to  \sing\mcal{Z}^{q}(X\times_{\R}\Delta^{\bullet}_{\R}) = \sing\mcal{Z}^{q}(X_{\C}\times_{\C}\Delta^{\bullet}_{\C})^{G}
$$
which induces $\Z/2(q)(X) \to \Z/2(q)^{sst}(X)$. The second map $\Phi$ is the map (\ref{morphcomp}) and the third map $\psi$ is induced by the projection $X_{\C}(\C)\times EG \to X_{\C}(\C)$.

\subsection{Nisnevich hypercohomology and descent}
These cohomology theories may be computed as Nisnevich  hypercohomology  groups. This allows us to view these cycle maps as maps in a derived category where we can use a computation of \cite{SV:BK}.

Say that a cartesian square 
\begin{equation}\label{Nisdist}
 \xymatrix{
  V \ar[r]\ar[d] & Y\ar[d]^{f} \\
U\ar@^{{(}->}[r]^{i} & X
 }
\end{equation}
is a \textit{distinguished Nisnevich square} provided the map $Y\xrightarrow{f} X$ is etale, 
$i:U\subseteq X$ is an open embedding, and $f: (Y\backslash V) \to (X\backslash U)$ is an 
isomorphism. The Nisnevich topology is the Grothendieck topology on $Sm/k$ generated by 
covers of the form $U \coprod Y \to X$ where $U\subseteq X$ and $f:Y\to X$ form part of a 
distinguished square as above.

Given a presheaf of chain complexes $F$ and a closed $i:A\subseteq B$ and open complement $j:U \subseteq B$ define
\begin{equation*}
F(B)_{A} = \cone(F(B) \xrightarrow{j^{*}} F(U))[-1], 
\end{equation*}
which fits into the exact triangle
\begin{equation*}
 F(B)_{A} \to F(B) \xrightarrow{j^{*}} F(U).
\end{equation*}

Say that a presheaf $F(-)$ of chain complexes satisfies \textit{Nisnevich descent} provided that for a distinguished square as in (\ref{Nisdist})  the square
\begin{equation*}
 \xymatrix{
   F(X) \ar[r]\ar[d] & F(Y)\ar[d] \\
F(U)\ar[r] & F(V)
 }
\end{equation*}
is homotopy cartesian. Recall this means that this square induces the Mayer-Vietoris exact triangle (in the derived category of abelian groups):
\begin{equation*}
 F(X) \to F(Y)\oplus F(U) \to F(V) .
\end{equation*} 
Equivalently, it means that $ F(Y)_{Z'}\to F(X)_{Z} $ is an isomorphism in the derived category of abelian groups where $Z=X\backslash U$ and $Z'=Y\backslash V$.

When a presheaf of chain complexes $F(-)$ (with $F(\emptyset)=0$) satisfies Nisnevich descent then the Nisnevich hypercohomology of a smooth $X$ with coefficients in $F$ is computed as 
$$H^{p}(F(X)) = H^{p}(F _{Nis}(X))=\H^{p}_{Nis}(X; F_{Nis})$$
(see for example \cite[Theorem 7.5.1]{BO} for presheaves of chain complexes, \cite{Nis} for descent in the case of presheaves of spectra, \cite{BG} for descent in the Zariski topology).

Note that 
\begin{align*}
\H^{2q-p}_{Nis}(X;(\Z/2(q)^{sst})_{Nis}) & 
= L^qH^{q-p,q}(X;\Z/2), \\
\H^{2q-p}_{Nis}(X;\Z/2(q)^{top}) & 
= H^{q-p,q}(X _\R(\C);\underline{\Z/2}), \\
 \H^{2q-p}_{Nis}(X;\Z/2(q)^{Bor}) & 
= H^{2q-p}_{\Z/2}(X(\C);\Z/2).
\end{align*}
In the first case this follows because the motivic complex $\Z/2(q)$ satisfies Nisnevich descent and $\Z/2(q)(X) \to \Z/2(q)^{sst}(X)$ is a quasi-isomorphism of chain complexes for all smooth $X$.

Given  $i:A\subseteq B$ a closed subvariety with open complement $j:U\subseteq B$ and write $C(j)$ for the mapping cone of $j:U(\C) \subseteq B(\C)$. Then by a comparison of exact triangles we see that
$$
\Z/2(q)^{top}(B)_{A} \wkeq \phom{C(j)\wedge\Delta^{\bullet}_{top,+}}{\mcal{Z}/2_{0}(\A^{q}_{\C})}^{G}[-2q]
$$ 
and 
$$
 \Z/2(q)^{Bor}(B)_{A} \wkeq \phom{C(j)\wedge EG_+\wedge\Delta^{\bullet}_{top,+}}{\mcal{Z}/2_{0}(\A^{q}_{\C})}^{G}[-2q] .
$$

Let $F(-)$ denote either $\Z/2(q)^{top}(-)$ or $\Z/2(q)^{Bor}(-)$ and let
\begin{equation*}
 \xymatrix{
  V \ar[r]^{j'}\ar[d] & Y \ar[d]\\
U \ar[r]^{j} & X
 }
\end{equation*}
be a distinguished Nisnevich square in $Sm/\R$ then
 \begin{equation*}
\xymatrix{ 
V(\C)\ar[r]^{j'}\ar[d] & Y(\C)\ar[d] \\
U(\C) \ar[r]^{j} & X(\C)
}
\end{equation*}
 is an equivariant homotopy pushout diagram of $G$-spaces (see for example \cite{DI:hyp} ).
 Therefore $C(j') \xrightarrow{\wkeq} C(j)$ is an equivariant homotopy equivalence.
 Consequently $F(X)_{Z}\to F(Y)_{Z'}$ is an isomorphism in the derived category of abelian groups and therefore
\begin{equation*}
 \xymatrix{
  F(X) \ar[r]\ar[d] & F(U) \ar[d]\\
F(Y) \ar[r] &  F(V) 
 }
\end{equation*}
is homotopy cartesian. This means that both $\Z/2(q)^{top}(-)$ and $\Z/2(q)^{Bor}(-)$ satisfy Nisnevich descent.

\subsection{Compatibility of cycle maps}
We are now ready to show that two cycle maps discussed in the beginning of this section are the same map.

\begin{lemma} \label{Borel}
 Suppose that $V$ is a quasi-projective complex variety considered as a real variety. Then $\pi_{k}\Z/2(q)^{Bor}(V) = H_{sing}^{2q-k}(V(\C);\Z/2)$.
\end{lemma}
\begin{proof}
 If $V$ is a complex variety then $(V\times_{\R}\C)(\C) = V(\C)\amalg V(\C)$ and $G$ acts by interchanging the factors. In particular $G$ acts freely on $V_{\C}(\C)$ and $(V_{\C}(\C)\times EG)/G \to V_{\C}(\C)/G = V(\C)$ is a vector-bundle which immediately implies that $\pi_{k}\Z/2(q)^{Bor}(V) = H^{2q-k} _{sing}(V(\C);\Z/2)$.
\end{proof}

Write $\pi_{0}:(Sm/\R)_{et} \to (Sm/\R)_{Nis}$ for the canonical map of sites.

\begin{proposition} \label{etpb}
 The complex of etale sheaves $\pi_{0}^*\Z/2(q)^{Bor}$ on $(Sm/\R)_{et}$ is canonically quasi-isomorphic to $\mu_{2}^{\otimes q}$.
\end{proposition}
\begin{proof}
 Write $\mcal{H}^{i}$ for the etale sheafification of the $i$th cohomology presheaf of $\Z/2(q)^{Bor}$. First we show that $\mcal{H}^{i} = 0$ for $i\neq 0$.
It is enough to show that for each real variety $X$ and $\gamma \in  H^{i}( [X_{\C}(\C)\times EG]/G ; \Z/2)$ that we can find an etale covering $(U_j\to X)$ such that $\gamma|_{U_j} = 0$ for each $j$.  The map $Y=X_{\C}\to X$ is an etale cover for any real variety $X$. Write $\gamma' = \gamma|_{Y}$. By the previous Lemma 
$H^{i}( Y _\C(\C)\times EG]/G ; \Z/2) = H^{i}_{sing}( Y(\C) ; \Z/2)$. Since $Y$ has an etale cover $U_j\to Y$ such that  $\gamma'|_{U_j} = 0$ for each $j$ (see e.g. \cite[Lemma III.3.15]{Milne:etale} ) we conclude that $\mcal{H}^{i} = 0$ for $i\neq 0$.

When $X_{\C}$ is connected then $H^{0}( [X_{\C}(\C)\times EG]/G ; \Z/2) = \Z/2$. More generally if $X = \amalg X_{i}$ is the disjoint union of $c$ connected real varieties then $H^{0}( [X_{\C}(\C)\times EG]/G ; \Z/2) =\Z/2^{\times c}$. This shows $\mcal{H}^{0} = \Z/2 = \mu_{2}^{q}$.

Finally since $\mcal{H}^{i} = 0 $ for $i\neq 0 $ we have canonical isomorphisms
\begin{equation*}
 \Z/2 = \Hom{et}{\mcal{H}^{0}}{\mu_{2}^{\otimes q}} = \Hom{D^{-}((Sm/\R)_{et})}{\pi^*\Z/2(q)^{Bor}}{ \mu_{2}^{\otimes q}}
\end{equation*}

\end{proof}
Recall \cite[Section 6]{SV:BK} that there is an injective etale resolution $0\to\mu_2^{\otimes q} \to J^{\bullet}$ such that $\pi_{0*}J^{\bullet}$ is a complex of Nisnevich sheaves with transfers with homotopy invariant cohomology sheaves. Proposition \ref{etpb} gives a canonical map $\pi_{0}^{*}\Z/2(q)^{Bor}\to J^{\bullet}$ and by adjointness we obtain a map 
\begin{equation*}
\Z/2(q)^{Bor} \to \mathbb{R}(\pi_{0})_{*}\mu_{2}^{\otimes q}=(\pi_{0})_*J^{\bullet}  .
\end{equation*}

Consider the following sequence of maps of complexes of Nisnevich sheaves
\begin{equation}\label{maps}
 \Z/2(q) \to (\Z/2(q)^{sst})_{Nis} \xrightarrow{\Phi} \Z/2(q)^{top} \to \Z/2(q)^{Bor} \to \mathbb{R}(\pi_{0})_{*}\mu_{2}^{\otimes q}.
\end{equation}

The complex of Nisnevich sheaves with transfers $B_{2}(q)$ is defined in \cite[Section 6]{SV:BK} to be the truncation 
$$
 B_{2}(q) = \tau_{\leq q}(\pi_{0})_*J^{\bullet} = \tau_{\leq q}(\mathbb{R}\pi_{0*}\mu_{2}^{\otimes q}), 
$$
in particular  $\H^{p}_{Nis}(X; B_{2}(q)) = H^{p}_{et}(X;\mu_2^{\otimes q})$ for $p\leq q$ and all smooth $X$.
Since the cohomology sheaves of $\Z/2(q)$ (and therefore of $\Z/2(q)^{sst}$ as well) vanish in degrees $i> q$ and so the sequence of maps (\ref{maps}) factors  through the truncations, 
\begin{equation}\label{tmaps}
 \Z/2(q) \to (\Z/2(q)^{sst})_{Nis} \to \tau_{\leq q} \Z/2(q)^{top} \to \tau_{\leq q} \Z/2(q)^{Bor} \to B_{2}(q).
\end{equation}
\begin{remark}\label{notriv}
 It is important to note that the composites \ref{maps} and \ref{tmaps} are non-trivial. 
This can be seen, for example, by evaluating on $\spec\C$. 
The map $\Z/2(q)(X) \to  (\Z/2(q)^{sst})_{Nis}(X)$ is a quasi-isomorphism for any smooth real variety $X$ by Proposition \ref{fcoeffagr}. The comparison map $(\Z/2(q)^{sst})_{Nis}(\C) \to  
\Z/2(q)^{top}(\C)$ is an equality. By Proposition \ref{binj} below, for any $X$, the 
map $\Z/2(q)^{top}(X) \to  
\Z/2(q)^{Bor}(X)$ induces an isomorphism on cohomology in degrees $p \leq q$. Finally since $\Z/2(q)^{Bor} \to \mathbb{R}\pi_{0*}\mu_{2}^{\otimes q}$ is obtained as the adjoint of a quasi-isomorphism and $\spec\C$ is an etale point (of $(Sm/\R)_{et}$) the map $\Z/2(q)^{Bor}(\C) \to \mathbb{R}\pi_{0*}\mu_{2}^{\otimes q}(\C)$ cannot be zero. 
\end{remark}

 Write $D^-(Nis)$ (respectively $D^-(NSwT/\R)$) for the derived category of bound above complexes of Nisnevich sheaves (respectively Nisnevich sheaves with transfers). Write $DM^-(\R)\subseteq D^-(NSwT/\R)$ for the full subcategory consisting of complexes with homotopy invariant Nisnevich cohomology sheaves. 
\begin{theorem}\label{cyccomp}
Let $X$ be a smooth real variety. The diagram commutes
\begin{equation}\label{tricom}
 \xymatrix{
 L^{q}H\R^{q-k,q}(X;\Z/2) \ar[d]_-{\Phi} && \ar[ll]_{\iso}\H^{2q-k,q}_{\mcal{M}}(X;\Z/2)  \ar[d]^-{cyc}  \\
H^{q-k,q}(X_{\C}(\C);\underline{\Z/2})  \ar[r] & H_{G}^{2q-k}(X_{\C}(\C);\Z/2)\ar[r] &  H^{2q-k}_{et}(X;\mu_{2}^{\otimes q}),
}
\end{equation}
\end{theorem}

\begin{proof}
 
By \cite[Corollary 6.11.1]{SV:BK} and the vanishing of the cohomology sheaves of $\Z/2(q)$  above degree $q$, 
$$
\Z/2= \Hom{DM^{-}(\R)}{\Z/2(q)}{B_{2}(q)}=\Hom{DM^{-}(\R)}{\Z/2(q)}{\pi_{0*}J^{\bullet}}.
$$ 

Also by \cite[Lemma 6.5]{SV:BK} the inclusion of bi-complexes 
$$
\Hom{Nis}{\Z/2(q)}{\pi_{0*}J^{\bullet}} \subseteq \Hom{NSWT}{\Z/2(q)}{\pi_{0*}J^{\bullet}}
$$
 is an equality. Therefore 
\begin{equation*}
\Z/2 =\Hom{DM^{-}(\R)}{\Z/2(q)}{\pi_{0*}J^{\bullet}}=\Hom{D^{-}(Nis)}{\Z/2(q)}{\pi_{0*}J^{\bullet}}.
\end{equation*}
Finally, the map  $\Z/2(q)\to \pi_{0*}J^{\bullet}$ obtained from (\ref{maps}) is not trivial by Remark \ref{notriv} and so we conclude that it must be the cycle map. 

\end{proof}

\subsection{Applications and computations}
As a result of the compatibility of cycle maps we can conclude some Beilinson-Lichtenbaum type theorems for morphic cohomology which we need to prove the vanishing theorem. We also use these to make a few computations of equivariant morphic cohomology. 

\begin{corollary}
\label{MCR}
Let $X$ be a smooth real variety. The map
\begin{equation*}
\Phi:L^{q}H\R^{q-k,q}(X;\Z/2^n) \to H^{q-k,q}(X _\C(\C);\underline{\Z/2^n})  
\end{equation*}
is an isomorphism for $q\leq k$ and a monomorphism for $q=k+1.$
\end{corollary}
\begin{proof}
Consider the commutative diagram (\ref{tricom}). By \cite{SV:BK} the Milnor conjecture, proved by Voevodsky \cite{Voev:miln}, implies that 
$cyc$ is an isomorphism for $k\geq q$ and an injection for $k=q-1$. This immediately implies 
the statement for injectivity. If $k\geq q$ then since $cyc$ is an isomorphism we conclude that $H_{G}^{2q-k}(X_{\C}(\C);\Z/2) \to H^{2q-k}_{et}(X;\mu_{2}^{\otimes q})$ is a 
surjective map between finitely dimensional $\Z/2$-vector spaces. By \cite{Cox:real} these are 
isomorphic $\Z/2$-vector spaces  and therefore the map is an isomorphism. Since $H^{q-k,q}(X_{\C}(\C);\underline{\Z}) \to H^{2q-k}_{G}(X_{\C}(X);\Z/2)$ is also an 
isomorphism for $k\geq q$ by Proposition \ref{binj} we conclude that $\Phi$ is also an isomorphism for $k\geq q$.  This yields the result for $\Z/2$-coefficients.  

We have the following diagram of distinguished triangles in $D^-(Nis):$
\begin{equation*}
 \xymatrix{
\Z/2(q)^{sst} _{Nis} \ar[r]\ar[d] & \Z/4(q)^{sst} _{Nis} \ar[r]\ar[d] & \Z/2(q)^{sst} _{Nis} \ar[r]\ar[d] & \Z/2(q)^{sst} _{Nis}[1]\ar[d] \\
 \tau_{\leq q}\Z/2(q)^{top} \ar[r] & \tau_{\leq q}\Z/4(q)^{top} \ar[r] & \tau_{\leq q}\Z/2(q)^{top}\ar[r] &\tau_{\leq q}\Z/2(q)^{top}[1] ,
}
\end{equation*}
To see that the bottom row is a triangle in $D^-(Nis)$ it is enough to check that 
the map on cohomology sheaves $\mcal{H}^{q}(\tau_{\leq q}\Z/4(q)^{top}) \to \mcal{H}^{q}(\tau_{\leq q}\Z/2(q)^{top})$ is a surjection.
 This follows from the surjectivity of the composition
 $$
\mcal{H}^q(\Z(q))\to\mcal{H}^q(\tau_{\leq q}\Z/4(q)^{top})\to\mcal{H}^q(\tau_{\leq q}\Z/2(q)^{top}).
$$
 which is a consequence of the local vanishing of $\Z(q)$ and the quasi-isomorphisms $\Z/2(q)\to B_{2}(q)$ and $\tau_{\leq q}\Z/2(q)^{top}\to B_{2}(q)$. Now the conclusion follows from the long exact sequence in hypercohomology associated to the diagram. Using induction on $n$ we conclude that $\Phi:(\Z/2^{n}(q)^{sst})_{Nis}\to \tau_{\leq q}\Z/2^{n}(q)^{top}$ is a quasi-isomorphism. 
\end{proof}
\begin{corollary}
\label{MCC}
Let $X$ be a smooth complex variety. For any $n>0$ the map 
$$
\Phi: L^{q}H^{2q-k}(X;\Z/2^n)\to H^{2q-k} _{sing}(X(\C);\Z/2^n)
$$ 
is an isomorphism for any $q\leq k$ and a monomorphism for $q=k+1$.
\end{corollary}
\begin{proof} 
Follows immediately from Corollary \ref{MCR} by viewing $X$ as real variety.
\end{proof}

\begin{corollary}\label{eBL}
 Let $X$ be a smooth real variety and $k>0$. The cycle map
\begin{equation*}
 \Phi:L^{q}H\R^{r,s}(X; \Z/2^{k}) \to H^{r,s}(X_{\C}(\C); \underline{\Z/2^{k}})
\end{equation*}
is an isomorphism if $r\leq 0$ (and $s\leq q$) and an injection if $r=1$ (and $s\leq q$).
\end{corollary}
\begin{proof}
 Write $F  _q = \hofib(\mcal{Z}^{q}/2^{k}(X_{\C}) \to \map{X_{\C}(\C)}{\mcal{Z}/2^{k}_{0}(\A^{q}_{\C})})$ for the homotopy fiber of the cycle map. The homotopy fiber construction is equivariant and yields an equivariant homotopy fiber sequence
\begin{equation*}
 F _q\to\mcal{Z}^{q}/2^{k}(X_{\C}) \to \map{X_{\C}(\C)}{\mcal{Z}/2^{k}_{0}(\A^{q}_{\C})}.
\end{equation*}
By Corollary \ref{MCC} and Corollary \ref{MCR} both $\pi_{k}(F _q) = 0$ and $\pi_{k}(F^{G} _q) =0$ for $k\geq q-1$. Therefore $\Omega^{q-1}F _q$ is equivariantly weakly contractible for $q\geq 1$ and if $q=0$ then $F _0$ is equivariantly contractible. The result follows now from the long exact sequence of homotopy groups applied to the equivariant homotopy fiber sequence
\begin{equation*}
\Omega^{q-1}F _q\to\Omega^{q-1}\mcal{Z}^{q}/2^{k}(X_{\C}) \to \Omega^{q-1}\map{X_{\C}(\C)}{\mcal{Z}/2^{k}_{0}(\A^{q}_{\C})}.
\end{equation*}
\end{proof}

\begin{corollary}
\label{cds}
 Let $X$ be a smooth real curve. Then 
\begin{equation*}
 L^{q}H\R^{r,s}(X; \Z) \to H^{r,s}(X_{\C}(\C); \underline{\Z})
\end{equation*}
is an isomorphism for any $q\geq 0$, $r\leq q$, and $s\leq q$.
\end{corollary}
\begin{proof}
 By Poincare duality for real Lawson homology and equivariant morphic cohomology and Remark \ref{fpa}, $L^{q}H^{r,s}(X;\Z) \xrightarrow{\iso} H^{r,s}(X_{\C}(\C);\underline{\Z})$ for $q\geq 1$. By Corollary \ref{eBL},  $L^{0}H\R^{r,s}(X;\Z/2^{k}) \to H^{r,s}(X;\underline{\Z/2^{k}})$ is an isomorphism for $r,s\leq 0$. When $A$ is an abelian group and $2$ is invertible in $A$ then a transfer argument shows that 
$$L^{0}H\R^{r,s}(X;A)\stackrel{\iso}{\to} H^{r,s}(X;\underline{A}).$$ 
  This isomorphism and the one with  mod-$2^k$ coefficients give the result of the corollary.
\end{proof}

\begin{corollary}
 Let $X$ be a smooth real surface. Then for any $k>0$
\begin{equation*}
 L^{q}H\R^{r,s}(X;\Z/2^{k}) \to H^{r,s}(X_{\C}(\C);\underline{\Z/2^{k}})
\end{equation*}
is an isomorphism for $q=0$ and $r,s\leq 0$ and it is an injection for $r=1$ and $s\leq 1$. Moreover $L^{1}H\R^{1,s}(X;\Z/2^{k}) = 0$ for $s\leq -2$.
\end{corollary}

Recall that $\pi_{0}\mcal{Z}^{q}(X_{\C})^{G}$ is the group of codimension $q$ cycles on $X$ modulo real algebraic equivalence. 
\begin{corollary}
 Let $X$ be a smooth real variety of dimension $d$. Then for any $k>0$
\begin{equation*}
 L^{1}H\R^{r,s}(X;\Z/2^{k}) \to H^{r,s}(X; \underline{\Z/2^{k}})
\end{equation*}
is an isomorphism for any $r\leq 0$ and $s\leq 1$ and it is an injection for $r=1$ and $s\leq 1$. Moreover
\begin{align*}
 L^{1}H\R^{1,1}(X;\Z/2^{k}) & = CH^{1}(X)\otimes\Z/2^{k}\subseteq H^{1,1}(X_{\C}(\C);\underline{\Z/2^{k}}) \\
L^{1}H\R^{1,s}(X;\Z/2^{k}) & = 0 \;\;\;\;\; \textrm{for} \;\; s\leq -2
\end{align*}
\end{corollary}
\begin{proof}
 All statements follow immediately from Corollary \ref{eBL} except the last one. The first part of the last statement follows from Proposition \ref{fcoeffagr}. For rest of the last statement, by Corollary \ref{eBL} together with Proposition \ref{binj} we have 
\begin{equation*}
 L^{1}H\R^{1,s}(X;\Z/2^{k}) \hookrightarrow H^{1,s}(X_{\C}(\C); \underline{\Z/2^{k}}) \hookrightarrow H^{1+s}_{G}(X_{\C}(\C);A) = 0
\end{equation*}
for $1+ s <0$.
\end{proof}

We finish this section with the computation used in Corollary \ref{MCR} that Bredon and Borel cohomology agree in the range relevant to the Beilinson-Lichtenbaum conjecture.
\begin{proposition}\label{binj}
 Let $W$ be a $G$-$CW$ complex, $M$ a $G$-module, and $\underline{M}$ the associated constant Mackey functor.  The map 
$$
H^{m,q}(W;\underline{M}) \to H^{m,q}(W\times EG;\underline{M}) 
$$ 
is an isomorphism for $m\leq 0$ and it is an injection for $m = 1$.

In particular for $q\leq p$ the map 
$$
H^{q-p,q}(W;\underline{\Z/2}) \to H^{2q-p}_{\Z/2}(W,\Z/2)
$$
 is an isomorphism and an injection for $p=q-1$.
\end{proposition}
\begin{proof}
Define $\tilde{E}G = \colim_{n} S^{0,n}$. This space fits into a homotopy cofiber sequence
\begin{equation*}
 EG_{+}\to S^{0} \to \tilde{E}G
\end{equation*}
and $\tilde{E}G/G \wkeq S^{1,0}\wedge BG$.

First we consider the case that $G$ acts trivially on $W$. From the previous homotopy cofiber sequence we obtain the homotopy cofiber sequence
\begin{equation}\label{EGcof}
 W_+\wedge EG_+ \to W_+ \to W_+\wedge \tilde{E}G.
\end{equation}
 Since $G$ acts trivially on $W$ we have that $(W_+\wedge \tilde{E}G)/G = W_+ \wedge \tilde{E}G/G\wkeq W_+ \wedge S^{1,0}\wedge BG$ and
 therefore $\tilde{H}^{k,0}(W_+\wedge \tilde{E}G;\underline{M}) = 
\tilde{H}_{sing}^{k-1}(W_+\wedge BG;M) = 0$ if $k\leq 1$. 
Notice that $S^{0,1}\wedge \tilde{E}G = S^{0,1}\wedge \colim_{n}S^{0,n} = \colim_{n} S^{0,n+1}\iso \tilde{E}G$. Since $\tilde{E}G \iso S^{0,1}\wedge \tilde{E}G$ is an 
equivariant equivalence this induces an isomorphism $$\tilde{H}^{k,s}(W_+\wedge 
\tilde{E}G;\underline{M}) \iso \tilde{H}^{k,0}(W_+\wedge 
\tilde{E}G;\underline{M})$$ for all $s$ and therefore $\tilde{H}^{k,s}(W_+\wedge 
\tilde{E}G, \underline{M})= 0$ for $k\leq 1$ for all $s$. Now from the long exact sequence associated to the 
cofiber sequence (\ref{EGcof}) it follows that for all $q$ the map $H^{m,q}(W;\underline{M}) \to 
H^{m,q}(W\times EG;\underline{M})$ is an isomorphism for $m\leq 0$ and an injection for $m=1$.

Consider now a general $G$-$CW$ complex $W$ and consider the quotient $W/W^G$. Since $G$ acts freely on the based space $W/W^{G}$ we have the isomorphism $$\tilde{H}^{s,t}(X/X^{G};\underline{M})\xrightarrow{\iso} \tilde{H}^{s,t}(X/X^{G}\wedge EG_{+};M).$$ Applying the five lemma to the comparison of long exact sequences obtained from the cofiber sequences $W_{+}^{G}\to W_{+} \to W/W^{G}$  and $(W^{G}\times EG)_{+} \to (W\times EG)_{+} \to W/W^{G}\wedge EG_{+}$ yields the result.

\end{proof}

\section{Vanishing Theorem}

Let $X$ be a  real variety. Continue to write $G = \Z/2$ and $\sigma\in G$ for the nontrivial element. Recall that the reduced real cycle group is defined to be the quotient topological group
\begin{equation*}
 \mcal{R}_{q}(X) = \frac{\mcal{Z}_{q}(X_{\C})^{G}}{\mcal{Z}_{q}(X_{\C})^{av}}.
\end{equation*}
This is the free $\Z/2$-module generated by closed subvarieties $Z\subseteq X$ such that both $Z$ and $Z_{\C}$ are irreducible. In particular, if $X$ is a complex variety viewed as a variety over $\R$ then $\mcal{R}_{q}(X) = 0$. 

Reduced real Lawson homology of $X$ is defined by the homotopy groups of this space,
\begin{equation*}
 RL_{q}H_{n}(X) = \pi_{n-q}\mcal{R}_{q}(X).
\end{equation*}

In this section we prove our main theorem which we state now. 
\begin{varthm}[\textbf{Theorem \ref{main}}]
 Let $X$ be a quasi-projective real variety.   Then 
\begin{equation*}
\pi_k\mcal{R}_n(X) = RL_{n}H_{k+n}(X) = 0 
\end{equation*}
for $k\geq \dim X -n +1$.
\end{varthm}

To avoid difficulties with point-set topology below we work simplicially. Note that if $X$ is a $G$-space then $\sing_{\bullet}X$ is a $G$-simplicial set and $\sing_{\bullet}(X^{G}) = (\sing_{\bullet} X)^{G}$. If $A_{\bullet}$ is a $G$-simplicial set then $|A^G_{\bullet}| = |A_{\bullet}|^{G}$ (see for example \cite[Lemma A.5]{dug:kr}). 

 \begin{definition}
\begin{enumerate}
\item Let $W$ be a $G$-space. Write
\begin{equation*}
 \widetilde{\mcal{Z}}^q_{top}(W) = \uphom{W\times \Delta^{\bullet}_{top}}{\mcal{Z}_0(\A^{q}_{\C})},
\end{equation*}
 and 
\begin{equation*}
 \widetilde{\mcal{Z}}^{q}/2_{top}(W) = \uphom{W\times\Delta^{\bullet}_{top}}{\mcal{Z}/2_0(\A^{q}_{\C})}. 
\end{equation*}
These are simplicial abelian groups and $G$ acts on them by the standard formula $(\sigma f)(x) = \sigma f(\sigma  x)$.

\item Let $X$ be a normal quasi-projective real variety. Write
$$
\widetilde{\mcal{Z}}^q(X_{\C}) = \sing_{\bullet}\mcal{Z}^q(X_{\C})
$$
 and 
$$ 
\widetilde{\mcal{Z}}^{q}/2(X_{\C}) = \sing_{\bullet}\mcal{Z}^{q}/2(X_{\C}).
$$ 

\end{enumerate}
\end{definition}
We have 
\begin{equation*}
  \pi_k\widetilde{\mcal{Z}}^{q}_{top}(X _\C(\C))^G \iso H^{q-k,q}(X_{\C}(\C),\underline{\Z}) 
 \end{equation*}
and
\begin{equation*}
\pi_k\widetilde{\mcal{Z}}^{q}(X _\C)^G \iso L^qH\R^{2q-k}(X;\Z)
\end{equation*}
and similarly for the mod-$2$ groups. In particular if $X$ is a complex variety viewed as a real variety then $X_{\C}(\C) = X(\C)\coprod X(\C)$ (with $G$-action switching the factors) and so 
 \begin{equation*}
  \pi_k\widetilde{\mcal{Z}}^{q}_{top}(X _\C(\C))^G \iso H^{2q-k}_{sing}(X(\C),\underline{\Z}) \;\;\;\;\textrm{and}\;\;\;\;\ \pi_k\widetilde{\mcal{Z}}^q(X _\C)^G \iso L^qH^{2q-k}(X;\Z)
 \end{equation*}
and similarly for the mod-$2$ groups.

The comparison maps (\ref{morphcomp}) of simplicial abelian groups 
\begin{equation*}
\Phi:\widetilde{\mcal{Z}}^q(X_{\C}) \to \widetilde{\mcal{Z}}^{q}_{top}(X _\C(\C))
\hskip 0.5cm \text{and} \hskip 0.5cm
\Phi:\widetilde{\mcal{Z}}^{q}/2(X_{\C}) \to \widetilde{\mcal{Z}}^{q}/2_{top}(X _\C(\C))
\end{equation*} 
are $G$-equivariant for any quasi-projective real variety $X$.


If $M_{\bullet}$ is a simplicial $G$-module write 
$N=1+\sigma:M_{\bullet}\to M_{\bullet}$
and define $M_{\bullet}^{av}$ to be 
\begin{equation*}
M_{\bullet}^{av} = Im(N) =Im(1+\sigma: M_{\bullet} \to M_{\bullet}).
\end{equation*}

\begin{definition}
 Let $W$ be a $G$-space. Define the group of \textit{reduced topological cocycles} (of codimension $q$) to be the quotient simplicial abelian group
\begin{equation*}
 \widetilde{\mcal{R}}^{q}_{top}(W) = \frac{\uphom{W\times\Delta^{\bullet}_{top}}{\mcal{Z}/2_{0}(\A_{\C}^{q})}^G}
{\uphom{W\times\Delta^{\bullet}_{top}}{\mcal{Z}/2_{0}(\A_{\C}^{q})}^{av}} = \frac{\widetilde{\mcal{Z}}^{q}/2_{top}(W)^{G}}{\widetilde{\mcal{Z}}^{q}/2_{top}(W)^{av}}.
\end{equation*}
\end{definition}

To relate the space of reduced algebraic cocycles with the reduced topological cocycles we introduce the following  auxiliary simplicial set for $X$ a quasi-projective normal real variety:
\begin{equation*}
 \widetilde{\mcal{R}}^{q}(X) = \frac{\widetilde{\mcal{Z}^{q}}/2(X_{\C})^{G}}{\widetilde{\mcal{Z}}^{q}/2(X_{\C})^{av}}.
\end{equation*}

\begin{proposition}
Let $X$ be a normal quasi-projective real variety. The following diagrams commute and the horizontal rows are short exact sequences of simplicial abelian groups (and therefore in particular the horizontal rows are homotopy fiber sequences of simplicial sets)
\begin{equation}\label{fib1}
 \xymatrix{
\widetilde{\mcal{Z}}^{q}/2(X_{\C})^{G} \ar[r]\ar[d]^{\Phi^{G}} & \widetilde{\mcal{Z}}^{q}/2(X_{\C}) \ar[r]^{N}\ar[d]^{\Phi} & \widetilde{\mcal{Z}}^{q}/2(X_{\C})^{av} \ar[d]^{\Phi^{av}} \\
\widetilde{\mcal{Z}}^{q}/2_{top}(X_{\C}(\C))^{G} \ar[r] & \widetilde{\mcal{Z}}^{q}/2_{top}(X_{\C}(\C)) \ar[r]^{N} & \widetilde{\mcal{Z}}^{q}/2_{top}(X_{\C}(\C))^{av} ,
}
\end{equation}
and
\begin{equation}\label{fib2}
 \xymatrix{
\widetilde{\mcal{Z}}^{q}/2(X_{\C})^{av} \ar[r]\ar[d]^{\Phi^{av}} & \widetilde{\mcal{Z}}^{q}/2(X_{\C})^{G} \ar[r]\ar[d]^{\Phi^{G}} & \widetilde{\mcal{R}}^{q}(X) \ar[d]^{\overline{\Phi}} \\
\widetilde{\mcal{Z}}^{q}/2_{top}(X_{\C}(\C))^{av} \ar[r] & \widetilde{\mcal{Z}}^{q}/2_{top}(X_{\C}(\C))^{G} \ar[r] & \widetilde{\mcal{R}}^{q}_{top}(X _{\C}(\C)) .
}
\end{equation}
\end{proposition}
\begin{proof}
These diagrams commute because $\Phi$ is a $G$-homomorphism. 

Whenever $M$ is a $G$-module whose underlying abelian group is $2$-torsion then the sequence of abelian groups $ 0\to M^{G} \to M \xrightarrow{N} M^{av} \to 0$ is exact.

In particular the underlying sequences of simplicial abelian $G$-modules in the first diagram form short exact sequences of simplicial abelian groups.

 In the second diagram the horizontal rows  form short exact 
sequences by definition of $\widetilde{R}^{q}(-)$ and $\widetilde{R}^{q}_{top}(-)$.
\end{proof}

 By definition we have 
\begin{equation*}
 (\sing_{\bullet} \mcal{Z}^{q}/2(X_{\C}))^{av}=\im(\sing_{\bullet}\mcal{Z}^{q}/2(X_{\C})
 \xrightarrow{N} \sing _{\bullet}\mcal{Z}^{q}/2(X_{\C})) . 
\end{equation*}

There is a natural map $i: (\sing_{\bullet}\mcal{Z}^{q}/2(X_{\C}))^{av} \to \sing_{\bullet}(\mcal{Z}^{q}/2(X_{\C})^{av})$  which is simply  
\begin{equation*}
i(f+\overline{f})=f+\overline{f}
\end{equation*} 
for a continuous map $f:\Delta^d _{top}\to \mcal{Z}^{q}/2(X_{\C})$. The map $ i:(\sing_{\bullet}\mcal{Z}^{q}/2(X_{\C}))^{av} \to \sing_{\bullet}(\mcal{Z}^{q}/2(X_{\C})^{av})$ induces a map

\begin{equation}\label{rvcomp}
 \overline{i}: \widetilde{\mcal{R}}^{q}(X) \to \sing_{\bullet}\mcal{R}^{q}(X).
\end{equation}

\begin{lemma}\label{sra}
Let $X$ be a normal real projective variety. The map (\ref{rvcomp}) of simplicial abelian groups
\begin{equation*}
 \widetilde{\mcal{R}}^{q}(X) \to \sing_{\bullet}\mcal{R}^{q}(X)
\end{equation*}
is a homotopy equivalence.
\end{lemma}
\begin{proof}
By Proposition \ref{qimav} the maps $\mcal{Z}^{q}/2(X_{\C})/\mcal{Z}^{q}/2(X_{\C})^{G}\xrightarrow{} \mcal{Z}^{q}/2(X_{\C})^{av}$ and $\mcal{Z}^{q}(X_{\C})^{av}/2\mcal{Z}^{q}(X)^{G} \xrightarrow{} \mcal{Z}^{q}/2(X_{\C})^{av}$ are isomorphisms of topological groups. Therefore both
\begin{equation*}
0\to\mcal{Z}^{q}/2(X_{\C})^{G} \to \mcal{Z}^{q}/2(X_{\C}) \to \mcal{Z}^{q}/2(X_{\C})^{av} \to 0
\end{equation*}
and
\begin{equation*}
0\to\mcal{Z}^{q}/2(X_{\C})^{av} \to \mcal{Z}^{q}/2(X_{\C})^{G} \to \mcal{R}^q(X)\to 0
\end{equation*}
are short exact sequences of topological abelian groups. These groups all have the homotopy type of a $CW$-complex and therefore these sequences are homotopy fiber sequences \cite{Teh:real}.  Applying $\sing_{\bullet}$ to these homotopy fiber sequence and comparing with the homotopy fiber sequences of the top rows of \ref{fib1} and \ref{fib2} gives commutative diagrams of homotopy fiber sequences of simplicial sets:
\begin{equation}\label{fibra1}
 \xymatrix{
\widetilde{\mcal{Z}}^{q}/2(X_{\C})^{G} \ar[r]\ar[d]^{\simeq} & \widetilde{\mcal{Z}}^{q}/2(X_{\C}) \ar[r]\ar[d]^{\simeq} & \widetilde{\mcal{Z}}^{q}/2(X_{\C})^{av} \ar[d]^{i} \\
\sing_{\bullet}{\mcal{Z}/2}^q(X_{\C})^{G} \ar[r] & \sing _{\bullet}{\mcal{Z}/2}^q(X_{\C}) \ar[r] & \sing_{\bullet}{\mcal{Z}/2}^q(X_{\C})^{av} ,
}
\end{equation}
and
\begin{equation}\label{fibra2}
 \xymatrix{
\widetilde{\mcal{Z}}^{q}/2(X_{\C})^{av} \ar[r]\ar[d]^{i} & \widetilde{\mcal{Z}}^{q}/2(X_{\C})^{G} \ar[r]\ar[d]^{\simeq} & \widetilde{\mcal{R}}^q(X) \ar[d]^{\overline{i}} \\
\sing_{\bullet}{\mcal{Z}/2}(X_{\C})^{av} \ar[r] & \sing_{\bullet}{\mcal{Z}/2}^q(X_{\C})^{G} \ar[r] & \sing_{\bullet}{\mcal{R}}^q(X ) .
}
\end{equation}
From the first diagram we see that  $i:\widetilde{\mcal{Z}}^{q}/2(X_{\C})^{av}\xrightarrow{\wkeq}\sing_{\bullet}(\mcal{Z}^{q}/2(X_{\C})^{av})$ is a weak equivalence of simplicial sets and consequently from the second diagram we conclude that 
$$
\widetilde{\mcal{R}}^{q}(X)\xrightarrow{\wkeq} \sing_{\bullet}\mcal{R}^{q}(X)
$$ 
is a weak equivalences of simplicial abelian groups and therefore is a homotopy equivalence of simplicial sets.
\end{proof}

We now prove our main theorem.
\begin{theorem}\label{main}
 Let $X$ be a quasi-projective real variety of dimension $d$.   Then 
\begin{equation*}
RL_{n}H_{n+k}(X) = \pi_k\mcal{R}_n(X) = 0
\end{equation*}
 for $k\geq d -n +1$.
\end{theorem}
\begin{proof}
We first consider the case when $X$ is a smooth projective real variety.

In case $n=d=dim(X)$ we have that 
$$
 \mcal{R} _d(X)=\frac{\mcal{Z} _d(X_\C)^G}{\mcal{Z} _d(X_\C)^{av}}=\mathbb{Z}/2^{\times c},
$$ 
where $c$ denotes the number of irreducible components of $X$ which are not defined over $\C$.

Therefore $\pi _0(\mcal{R} _d(X))=\mathbb{Z}/2^{\times c}$ and $\pi _i(\mcal{R} _d(X))=0$ for $i>0$.  

Consider the comparison of homotopy fiber sequences (\ref{fib1}) for $q>0$. 
\begin{equation}
 \xymatrix{
\widetilde{\mcal{Z}}^{q}/2(X_{\C})^{G} \ar[r]\ar[d]^{\Phi^{G}} & \widetilde{\mcal{Z}}^{q}/2(X_{\C}) \ar[r]^{N}\ar[d]^{\Phi} & \widetilde{\mcal{Z}}^{q}/2(X_{\C})^{av} \ar[d]^{\Phi^{av}} \\
\widetilde{\mcal{Z}}^{q}/2_{top}(X_{\C}(\C))^{G} \ar[r] & \widetilde{\mcal{Z}}^{q}/2_{top}(X_{\C}(\C)) \ar[r]^{N} & \widetilde{\mcal{Z}}^{q}/2_{top}^q(X_{\C}(\C))^{av} ,
}
\end{equation}
By the Milnor conjecture  (see Corollary \ref{MCC}) the comparison map $\Phi$ induces an isomorphism on $\pi_k$ for $k\geq q$ and induces an injection for $k =q-1$. By Corollary \ref{MCR} the map $\Phi^{G}$ induces an isomorphism on $\pi_k$ for $k\geq q$ and induces an injection for $k =q-1$. We now conclude by the $5$-lemma that $\Phi^{av}$ induces an isomorphism on $\pi_k$ for $k\geq q+1$. When $k=q$ we have the comparison diagram:
\begin{equation}
 \xymatrix{
\pi_q\widetilde{\mcal{Z}}^{q}/2^{G} \ar[r]\ar[d]^{\iso} & \pi_q\widetilde{\mcal{Z}}^{q}/2 \ar[r]\ar[d]^{\iso} & \pi_q\widetilde{\mcal{Z}}^{q}/2^{av} \ar[r]\ar[d]^{\Phi^{av}} & \pi_{q-1}\widetilde{\mcal{Z}}^{q}/2^{G} \ar@^{(->}[d] \\
\pi_q\widetilde{\mcal{Z}}^{q}/2_{top}^{G} \ar[r] & \pi_q\widetilde{\mcal{Z}}^{q}/2_{top} \ar[r] & \pi_q\widetilde{\mcal{Z}}^{q}/2_{top}^{av} \ar[r] & \pi_{q-1}\widetilde{\mcal{Z}}^{q}/2_{top}^{G}
}
\end{equation}
and so $\Phi^{av}$ induces an injection for $k=q$.

Considering now the comparison diagram (\ref{fib2}) and using the five-lemma we have that  $\overline{\Phi}$ induces an isomorphism on $\pi_{k}$ for $k\geq q+2$ and an injection for $k =q+1$.

By Corollary \ref{rdual}, $\pi_k\widetilde{\mcal{R}}^{q}_{top}(X_{\C}(\C)) =  H^{q -k}(X(\R), \Z/2)$ for $k\geq 2$. In particular $\pi_k\widetilde{\mcal{R}}^{q}(X) = 0$ for $k\geq q+1$, when $q\geq 1$.
By the homotopy equivalences $\widetilde{\mcal{R}}^{q}(X)\wkeq \sing_{\bullet}\mcal{R}^{q}(X)$ (see Lemma \ref{sra})
and the duality \cite[Theorem 5.14]{Teh:real} between reduced cycle and reduced cocycle spaces $\mcal{R}^{q}(X) \xrightarrow{\wkeq} \mcal{R}_{d-q}(X)$  the vanishing $\pi_k\widetilde{\mcal{R}}^{q}(X) = 0$ for $k\geq q+1$  is equivalent to the vanishing $\pi_k\mcal{R}_{n}(X) = 0$ for $k\geq \dim X -n+1$. 

Now let $X$ be a smooth quasi-projective variety and let $X\subseteq \overline{X}$ be a projective closure with closed complement $Z= \overline{X}\backslash X$. The result follows from the projective case and the long exact sequence in homotopy groups induced by the homotopy fiber sequence
\begin{equation*}
 \mcal{R}_{n}(Z) \to \mcal{R}_{n}(\overline{X}) \to \mcal{R}_{n}(X). 
\end{equation*}

Finally let $X$ be an arbitrary quasi-projective variety. There is an increasing filtration of closed subvarieties
\begin{equation*}
 \varnothing = X_{-1} \subseteq X_{0} \subseteq X_{1} \subseteq \cdots \subseteq X_{d} = X
\end{equation*}
such that $X_{i+1}\backslash X_{i}$ is smooth and $\dim X_{i} = i$. We proceed by induction, the case $i=0$ is done. Consider the long exact sequence which arises from the homotopy fiber sequence
\begin{equation*}
 \mcal{R}_{n}(X_{i}) \to \mcal{R}_{n}(X_{i+1}) \to \mcal{R}_{n}(X_{i+1}\backslash X_{i}).
\end{equation*}
Since the result holds for $X_{i}$ by induction and for $X_{i+1}\backslash X_{i}$ because it is smooth we obtain the result for $X_{i+1}$.
\end{proof}

\begin{remark} 
If $X$ is a projective smooth real variety of dimension $d$ it is proved in 
\cite[Theorem 6.7]{Teh:real} that $\pi _k (\mcal{R}_{d-1}(X))= 0$ for any $k\geq 3$. Theorem \ref{main} in case $n=d-1$ improves this vanishing bound. 
\end{remark}

\begin{example} \label{exR}
 If $n=0$ and if $X$ has no real points then $\mcal{R}_{0}(X)=0$ and so $$RL_{0}H_{*}(X) = \pi_{*}\mcal{R}_{0}(X)= H_*^{sing}(X(\R),\Z/2)=0.$$

Let $\P(\H)$ denote the space of complex lines in the quaternions $\H = 
\C \oplus j\C$ where $j^2 = -1$. Multiplication by $j$ defines an involution on $\P(\H)$ and write $Q$ for the corresponding 1-dimensional real curve. We know that $Q$ is the smooth real curve $X^2+Y^2+Z^2=0$ in $\mathbb{P}^2 _\mathbb{R}$. This is the Severi-Brauer variety corresponding to the non-trivial element of $Br(\R) = \Z/2$ and has no real points. This means $\mcal{R}_{0}(Q) = 0 $ and  $\mcal{R}_{1}(Q) = \Z/2$. Thus in this case, 
\begin{equation*}
 0 = RL_{0}H_{0}(Q) = RL_{0}H_{1}(Q) = H_0(Q(\R),\Z/2)
\end{equation*}
and 
\begin{equation*}
 \Z/2 = RL_1H_1(Q). 
 \end{equation*}
 
 Let $X=SP^{2d+1}(Q)$  be the smooth projective real variety given by an odd symmetric power of $Q$. Because $X _\mathbb{C}=\mathbb{P} _\mathbb{C}(\mathbb{H}^{d+1})$, we have $\mcal{R}^{2q}(X)=\mathbb{Z}/2 $  and $\mcal{R}^{2q+1}(X)=0$ for any $2q\leq 2d+1$ (see \cite[Theorem 2.3]{LLM:quat}). This implies that the only nonzero reduced Lawson homology groups of $X$ are $RL _{2r+1}H _{2r+1}(X)=\mathbb{Z}/2$ for any $r\leq d$. Notice that in this case $dim(X)=2d+1$.
  
  These computations show that the vanishing in the above theorem is best possible, even in the case of a real variety with no real points.  
\end{example}
\begin{example} According to \cite{Lam:t}, $RL _rH_n(\P^d _\R)=\Z/2$ for any $0\leq r\leq n\leq d$ and $RL_rH_n(\P^d _\R)=0$ for any $n>d$.
\end{example}
We also obtain the following vanishing result.
\begin{corollary}
\label{av2}
 Let $X$ be a smooth projective real variety of dimension $d$. Then
\begin{equation*}
 \pi_{n}\frac{\mcal{Z}_{p}(X_{\C})^{av}}{2\mcal{Z}_{p}(X_{\C})^{av}} = 0
\end{equation*}
for $n\geq 2d-2p+1$.
\end{corollary}
\begin{proof}
 
By the Corollaries \ref{MCC}, \ref{MCR}, \ref{qimav} and \ref{PD} we conclude that 
$$
\pi_n\frac{\mcal{Z}_{p}(X_{\C})^{av}}{2\mcal{Z}_{p}(X_{\C})^{G}} = 0
$$
 for $n\geq 2d- 2p+1$ from the long exact sequence in homotopy groups induced by the short exact sequence \cite[Proposition 4.3]{Teh:HT}
\begin{equation}
\label{ses1}
 0 \to \frac{\mcal{Z}_{p}(X_{\C})^{G}}{2\mcal{Z}_{p}(X_{\C})^{G}} \to \frac{\mcal{Z}_{p}(X_{\C})}{2\mcal{Z}_{p}(X_{\C})} \to \frac{\mcal{Z}_{p}(X_{\C})^{av}}{2\mcal{Z}_{p}(X_{\C})^{G}} \to 0 .
\end{equation}

Consider the short exact sequences of topological abelian groups
\begin{equation*}
 0 \to \frac{2\mcal{Z}_{p}(X_{\C})^{G}}{2\mcal{Z}_{p}(X_{\C})^{av}} \to 
\frac{\mcal{Z}_{p}(X_{\C})^{av}}{2\mcal{Z}_{p}(X_{\C})^{av}} \to 
\frac{\mcal{Z}_{p}(X_{\C})^{av}}{2\mcal{Z}_{p}(X_{\C})^{G}} \to 0 .
\end{equation*}

Multiplication by $2$ induces a homeomorphism
\begin{equation*}
\mcal{R}_{p}(X) = \frac{\mcal{Z}_{p}(X_{\C})^{G}}{\mcal{Z}_{p}(X_{\C})^{av}} \iso \frac{2\mcal{Z}_{p}(X_{\C})^{G}}{2\mcal{Z}_{p}(X_{\C})^{av}} 
 \end{equation*}
and plugging the vanishing for homotopy groups of $\mcal{R}_{p}(X)$ into the above exact sequence yields the result.
\end{proof}
\begin{remark}
 Using the same arguments as in Theorem \ref{main} shows that the vanishing in Corollary \ref{av2} holds for any quasi-projective real variety.
\end{remark}

\begin{example}
\label{avopt}
Let $X=\mathbb{P}^{d} _{\R}$. Then 
$$
\pi_{n}\left(\frac{\mcal{Z}_{p}(\P^{d}_{\C})^{av}}{2\mcal{Z}_{p}(\P^{d}_{\C})^{av}}\right) = 0
$$ 
for any $n\geq 2d-2p+1$. If $p=d$, then
$$\pi_{2d-2p}\left(\frac{\mcal{Z}_{d}(\P^{d}_{\C})^{av}}{2\mcal{Z}_{d}(\P^{d}_{\C})^{av}}\right) = \Z/2$$
If $p=0$ then, for any real projective variety $X$,  
$$
\pi _* \left(\frac{\mcal{Z}_{p}(X_{\C})^{av}}{2\mcal{Z}_{p}(X_{\C})^{av}}\right)=H _*(X(\C)/G,\Z/2).
$$ 
These computations show that the vanishing bound of Corollary \ref{av2} is the best possible. For these computations see \cite{LLM:quat}.
\end{example}
The following corollary shows that in a range the morphic cohomology of a real variety $X$ can be computed by the homotopy groups of average cycles on $X$. 
\begin{corollary}
\label{m=av}
 Let $X$ be a real quasi-projective variety. Then
$$\pi _q(\mcal{Z} _p(X _\mathbb{C})^G)\simeq \pi _q(\mcal{Z} _p(X _\mathbb{C})^{av})$$
for any $q\geq dim(X)-p+1$.
\end{corollary}
\begin{proof} 
This follows from Theorem \ref{main} together with the long exact sequence of homotopy groups associated to the homotopy fiber sequence
$$
0\rightarrow \mcal{Z} _{p}(X _\mathbb{C})^{av}\rightarrow \mcal{Z} _{p}(X _\mathbb{C})^G\rightarrow \mcal{R} _p(X)\rightarrow 0.
$$

\end{proof}
\begin{example}
\begin{enumerate}
 \item In case of divisors $p=dim(X)-1$, Corollary \ref{m=av} and \cite[Proposition 6.2]{Teh:HT} show that 
$$
\pi _q(\mcal{Z} _p(X _\mathbb{C})^{av})=0
$$ 
for any $q\geq 2$. 
\item In the case of zero-cycles $p=0$, we get 
$$
H _{k,0}(X(\mathbb{C}),\underline{ \mathbb{Z}})\simeq H _k(X(\mathbb{C})/G,\mathbb{Z})
$$ 
for any $k\geq dim(X)+1$.
\end{enumerate}
\end{example}

We conclude this section by observing  that the vanishing theorem also  shows that motivic cohomology of a real variety can be computed in a range via the complex of averaged equidimensional cycles on the complexification. 

Let $X$ and $Y$ be a quasi-projective real varieties. The group of reduced equidimensional cycles is defined to be the quotient group
\begin{equation*}
 r_{equi}(Y, r)(X) = \frac{z_{equi}(Y_{\C}, r)(X_{\C})^{G}}{z_{equi}(Y_{\C}, r)(X_{\C})^{av}}.
\end{equation*}

It is essentially a consequence of Suslin rigidity that the complex of reduced equidimensional cycles computes the reduced Lawson homology.

\begin{proposition}
Let $X$ be a quasi-projective real variety.
\begin{enumerate} 
\item The diagram
\begin{equation*}
 r_{equi}(\A^{q},q)(X\times\Delta_{top}^{\bullet}) \xleftarrow{\wkeq} r_{equi}(\P^{q/q-1},q)(X\times\Delta^{\bullet}_{top}) \xrightarrow{\wkeq} \sing_{\bullet}\mcal{R}^{q}(X)
\end{equation*}
consists of homotopy equivalences of simplicial sets. 
\item The map 
\begin{equation*}
 r_{equi}(\P^{q/q-1}, 0)(X\times\Delta^{\bullet}) \xrightarrow{\wkeq} r_{equi}(\P^{q/q-1}, 0)(X\times\Delta^{\bullet}\times\Delta^{\bullet}_{top})
\end{equation*}
is a homotopy equivalence of simplicial sets.
\end{enumerate}
\end{proposition}
\begin{proof}
The proof is similar to other proofs in this paper so we only provide a sketch.
First observe that the simplicial abelian group of reduced equidimensional cycles may  be computed as
\begin{equation*}
 r_{equi}(Y, r)(X) = \frac{(z_{equi}(Y_{\C}, r)(X_{\C})\otimes\Z/2)^{G}}{(z_{equi}(Y_{\C}, r)(X_{\C})\otimes\Z/2)^{av}}.
\end{equation*}
 Using Proposition \ref{smheq} and the appropriate analogues of the homotopy fiber sequences (\ref{fib1}) and (\ref{fib2}) we see that 
\begin{equation*}
 r_{equi}(\P^{q},0)(X\times\Delta^{\bullet}_{top})\to 
\frac{(\sing_{\bullet}\Mor{\C}{X_{\C}}{\mcal{C}_{0}(\P^{q}_{\C})}^{+}/2)^{G}}{(\sing_{\bullet}\Mor{\C}{X_{\C}}{\mcal{C}_{0}(\P^{q}_{\C})}^{+}/2)^{av}}
\end{equation*}
is a homotopy equivalence. The first part follows now in a similar fashion as Proposition \ref{qisoseq}. The second part follows from the fact that both over $\C$ and over $\R$ with finite coefficients motivic cohomology agrees with morphic cohomology.
\end{proof}

\begin{corollary}
 Let $X$ be a quasi-projective real variety. Then
\begin{equation*}
 H_{\mcal{M}}^{p}(X;\Z(q)) = \pi_{2q-p}z_{equi}(\A^{q}_{\C}, 0)(X_{\C}\times_{\C}\Delta^{\bullet}_{\C})^{av}
\end{equation*}
for $q-1 \geq p$
\end{corollary}

\section{Reduced Topological Cocycles}\label{subdual}

For typographical simplicity throughout this section we write $\mcal{Z} = \mcal{Z}/2_{0}(S^{q,q})_{0}$. This section is devoted to the computation that $\pi_{k}\widetilde{\mcal{R}}^{q}_{top}(W)= H^{q-k}_{sing}(W^{G};\Z/2)$, for $k\geq 2$, where $\widetilde{\mcal{R}}^{q}_{top}(W)$ is the quotient simplicial abelian group
$$
\displaystyle{\widetilde{\mcal{R}}^{q}_{top}(W)= \frac{\phom{W\wedge\Delta^{\bullet}_{top,+}}{\mcal{Z}}^{G}}{\phom{W\wedge\Delta^{\bullet}_{top,+}}{\mcal{Z}}^{av}}}
$$ 
and $W$ is a based finite $G$-$CW$ complex.

The idea is to reduce to the case of trivial action. Before doing this we sketch what happens when $G$ acts trivially on $W$. By \cite[Proposition 8.3]{LLM:real} the short exact sequence
\begin{equation*}
0\to \mcal{Z}/2_0(S^{q,q})_{0}^{av} \to \mcal{Z}/2_0(S^{q,q})_{0}^{G} \to \mcal{R}_{0}(S^{q,q})_{0}\to 0
\end{equation*}
is a fibration sequence (in fact principle fibration sequence) of topological spaces. Applying $\phom{W\wedge\Delta^{\bullet}_{top,+}}{-}$ to this sequence yields a homotopy fiber sequence of simplicial sets. Now we compare the homotopy fiber sequences of simplicial abelian groups
\begin{equation*}
 \xymatrix@-1pc{
\Hom{}{W\wedge\Delta^{\bullet}_{top,+}}{\mcal{Z}}^{av} \ar[r]\ar[d]&
\Hom{}{W\wedge\Delta^{\bullet}_{top,+}}{\mcal{Z}}^{G} \ar[r]\ar[d]^{\iso}& 
\widetilde{\mcal{R}}^q_{top}(W)\ar[d] \\
\Hom{}{W\wedge\Delta^{\bullet}_{top,+}}{\mcal{Z}^{av}} \ar[r]&
\Hom{}{W\wedge\Delta^{\bullet}_{top,+}}{\mcal{Z}^{G}} \ar[r]& 
\mcal{H}^{q}(W),
} 
\end{equation*}
where $\mcal{H}^{q}(W) = \phom{W\wedge\Delta^{\bullet}_{top,+}}{\mcal{R}_{0}(S^{q,q})_{0}}$. 
We will see that when $W$ has trivial $G$-action then the left vertical arrow induces an isomorphism on $\pi_{k}$ for $k\geq 1$ (see Corollary \ref{zeroav}). Therefore $\pi_{i}\widetilde{\mcal{R}}^q_{top}(W) =\pi_{i}\mcal{H}^{q}(W) = H^{q-i}_{sing}(W;\Z/2)$ for $i\geq 2$ when $W$ has trivial $G$-action. 

For a based $G$-$CW$ complex $W$ and a topological $G$-module $Z$, write 
$$
\phom{W}{Z}_{0}^{G}
$$
 
for the set of based equivariant maps which are equivariantly homotopic to the $0$-map (via a based homotopy).

\begin{lemma}\label{conncomp}
 Let $W$ be a based $G$-$CW$ complex and let $Z$ be a topological $G$-module. The simplicial set
\begin{equation*}
 d\mapsto\phom{W\wedge \Delta^{d}_{top,+}}{Z}_{0}^{G} 
\end{equation*}
is the path-connected component of the vertex $0\in \phom{W\wedge \Delta^{\bullet}_{top,+}}{Z}^{G}$.
\end{lemma}
\begin{proof}
A vertex $g \in \phom{W\wedge\Delta^{\bullet}_{top,+}}{Z}^{G}$ is in the same path component as the $0$-map if and only if there is a $1$-simplex $F\in  \phom{W\wedge\Delta^{1}_{top,+}}{Z}^{G}$ such that $F(0) = 0$ and $F(1) = g$. This happens exactly when $g\in \phom{W}{Z}_{0}^{G}$. 

A $d$-simplex, $f:W\wedge \Delta^{d}_{top,+} \to Z$ is in the path-component of $0$ if and only if its restriction to a vertex is in the path component of $0$. Since $\Delta^{d}_{top,+}$ is equivariantly contractible and the restriction  $f|_{W\wedge\{v\}_+}$ to a vertex is equivariantly homotopic to the constant map $0$ we conclude that $f$ itself is equivariantly homotopic to $0$.

 \end{proof}

\begin{definition} Let $W$ be a based $G$-$CW$ complex.
 \begin{enumerate}
  \item Define $\phom{W\wedge \Delta^{\bullet}_{top,+}}{\mcal{Z}}_{0}^{av}$ to be the path-connected component of the vertex $0$ in the simplicial set $\phom{W\wedge\Delta^{\bullet}_{top,+}}{\mcal{Z}}^{av}$.
\item Define 
\begin{equation*}
\widetilde{\mcal{R}}^{q}_{top}(W)_{0} =\frac{\phom{W\wedge \Delta^{\bullet}_{top,+}}{\mcal{Z}}_{0}^{G}}{\phom{W\wedge \Delta^{\bullet}_{top,+}}{\mcal{Z}}_{0}^{av}},
\end{equation*}
here the quotient is in the category of simplicial abelian groups. 
 \end{enumerate}
\end{definition}

Restricting to $W^{G}$ gives rise to the comparison map
\begin{equation*}
 \widetilde{\mcal{R}}^{q}_{top}(W)_{0} \to \widetilde{\mcal{R}}^{q}_{top}(W) \to \phom{W^{G}\wedge\Delta^{\bullet}_{top,+}}{\mcal{Z}/2_{0}(S^{q})}.
\end{equation*}

Note that $\pi_{i}\widetilde{\mcal{R}}^{q}_{top}(W)_{0} \to \pi_{i}\widetilde{\mcal{R}}^{q}_{top}(W)$ is an isomorphism for $i\geq 2$ and an 
injection for $i=0,1$. To compute $\pi_{i}\widetilde{\mcal{R}}^{q}_{top}(W)$ we will show 
that $\widetilde{\mcal{R}}^{q}_{top}(W)_{0}\to\widetilde{\mcal{R}}^{q}_{top}(W^{G})_{0}$ is an isomorphism. The surjectivity is easy but the injectivity takes some work.

\begin{proposition}
 Let $i:A\hookrightarrow W$ be an equivariant cofibration between based $G$-$CW$-complexes and let $Z$ be a topological $G$-module. Then
\begin{equation*}
 i^*:\phom{W}{Z}^{G}_{0}\to \phom{A}{Z}^{G}_{0}
\end{equation*}
is surjective.
\end{proposition}
\begin{proof}
 Suppose that $f:A\to Z$ is a based equivariant map which is based equivariantly homotopic to the $0$-map. Let $H:A\wedge I_+\to Z$ be an equivariant homotopy such that $H(-,0) = 0$ and $H(-,1) = f$. 

By the homotopy extension property of cofibrations, an equivariant map $H'$ making the diagram below commute exists
\begin{equation*}
 \xymatrix{
W\wedge\{0\}_+\coprod_{A\wedge\{0\}_+} A\wedge I_+ \ar[r]^-{0\amalg H}\ar@{^{(}->}[d] & Z\\
W\wedge I_+ \ar@{-->}[ur]_-{H'} & .
}
\end{equation*}
 The restriction of $f'=H'(-,1)$ to $A$ is equal to $f$ and $H'$ is an equivariant homotopy between $f'$ and the $0$-map.
\end{proof}

\begin{corollary}\label{rsurj}
 Let  $i:A\hookrightarrow W$ be an equivariant cofibration between based $G$-$CW$-complexes. The induced map
\begin{equation*}
 i^*:\rtop{q}(W)_{0} \to \rtop{q}(A)_{0}
\end{equation*}
is a surjection. 
\end{corollary}
\begin{proof}
 Consider the square
\begin{equation*}
 \xymatrix{
\phom{W\wedge\Delta^{\bullet}_{top,+}}{\mcal{Z}}_{0}^{G} \ar@{->>}[r]\ar@{->>}[d] & \phom{A\wedge\Delta^{\bullet}_{top,+}}{\mcal{Z}}_{0}^{G} \ar@{->>}[d]\\
\rtop{q}(W)_{0} \ar[r] & \rtop{q}(A)_{0}.
}
\end{equation*}
By the previous proposition, the top horizontal arrow is surjective. The vertical arrows are surjective by definition and therefore the bottom horizontal arrow is also surjective.
\end{proof}


For a based $CW$-complex $W$ and a topological abelian group $Z$ we will write $\phom{W}{Z}_{0}$ for the set of  based continuous maps which are based homotopic to the $0$-map. Note that the simplicial set 
\begin{equation*}
d\mapsto \phom{W\wedge\Delta^{d}_{top,+}}{Z}_{0}
\end{equation*}
is the path-connected component of the vertex $0\in\phom{W\wedge\Delta^{\bullet}_{top,+}}{\mcal{Z}}$ (for example consider $W$ and $Z$ with trivial $G$-action and apply Lemma \ref{conncomp}). If $Z$ and $W$ have a $G$-action write $\left(\phom{W}{Z}_{0}\right)^{av}$ for the image of $N=1+\sigma$.  This set consists of maps $h:W\to Z$ which can be written as  $h = f+\overline{f}$ where $f$ is a continuous map which is nonequivariantly homotopic to $0$.

We now justify the use of similar notation for two potentially different simplicial sets. Previously we wrote $\phom{W\wedge\Delta^{\bullet}_{top,+}}{\mcal{Z}}^{av}_{0}$ for the path-component of the vertex $0\in \phom{W\wedge\Delta^{\bullet}_{top,+}}{\mcal{Z}}^{av}$. We now verify that this path-component can be described explicitly as 
the image under $N=1+\sigma$ of the path-component of $0\in\phom{W\wedge\Delta^{\bullet}_{top,+}}{Z}$. In otherwords $\phom{W\wedge\Delta^{\bullet}_{top,+}}{\mcal{Z}}_{0}^{av}= (\phom{W\wedge\Delta^{\bullet}_{top,+}}{\mcal{Z}}_{0})^{av}$. This explicit description will be fundamental to our proof of Proposition \ref{kernsurj} below.

\begin{proposition}
 Let $W$ be a based $G$-$CW$-complex.  The simplicial set
\begin{equation*}
 d\mapsto(\phom{W\wedge\Delta^{d}_{top,+}}{\mcal{Z}}_{0})^{av} 
\end{equation*}
is the path-connected component of $0 \in \phom{W\wedge\Delta^{\bullet}_{top,+}}{\mcal{Z}}^{av}$
\end{proposition}
\begin{proof}
First we identify $(\phom{W}{\mcal{Z}}_{0})^{av}$ as the set of vertices of the path-connected component of $0\in \phom{W}{\mcal{Z}}^{av}$. Any $f+\overline{f}\in (\phom{W}{\mcal{Z}}_{0})^{av}$ is in the path component of $0$. Suppose the vertex $h+\overline{h}\in \phom{W}{\mcal{Z}}^{av}$ is in the same path component as $0$. This means there is a map of simplicial sets
\begin{equation*}
 F:I \to \phom{W\wedge\Delta^{\bullet}_{top,+}}{\mcal{Z}}^{av}
\end{equation*}
such that $F(0) = 0$ and $F(1) = h + \overline{h}$. Consider the diagram of simplicial sets,
\begin{equation*}
 \xymatrix{
\{0\}\ar[r]^-{0} \ar@{^{(}->}[d]^-{\wkeq} & \phom{W\wedge\Delta^{\bullet}_{top,+}}{\mcal{Z}} \ar@{->>}[d]^{N} \\
I \ar[r]^-{F}\ar@{-->}[ur]^-{F'} & \phom{W\wedge\Delta^{\bullet}_{top,+}}{\mcal{Z}}^{av} .
}
\end{equation*}
A surjection between simplicial abelian groups is a fibration and therefore an $F'$ exists to make the above square commute. 

The map $F'(1):W\to \mcal{Z}$ is in $\phom{W}{\mcal{Z}}_{0}$ and satisfies 
\begin{equation*}
F'(1) + \overline{F'(1)}= N(F'(1)) = F(1) = h + \overline{h}.
\end{equation*}
 We conclude that $(\phom{W}{\mcal{Z}}_{0})^{av}$ is the set of vertices of  the path-connected component of the $0 \in \phom{W}{\mcal{Z}}^{av}$. 

Now to conclude that the simplicial set 
$d\mapsto(\phom{W\wedge\Delta^{d}_{top,+}}{\mcal{Z}}_{0})^{av}$  is the 
path-connected component of $0$ we need to see that if the restriction $g|_{W\wedge\{v\}_+}$ of $g\in \phom{W\wedge\Delta^{n}_{top,+}}{Z}^{av}$ to a 
vertex $v\in \Delta^{n}_{top}$ lies in $(\phom{W}{Z}_{0})^{av}$  then $g\in (\phom{W\wedge\Delta^{n}_{top,+}}{Z}_{0})^{av}$. That is, if $g\in \phom{W\wedge\Delta^{n}_{top,+}}{Z}^{av}$ and that there is a map $f:W\to \mcal{Z}$ which is homotopic to $0$ such that the restriction of $g$ to some vertex
 $v\in \Delta^{n}_{top}$ satisfies $g|_{W\wedge \{v\}_{+}} = f+\overline{f}$ then we need to see that $g$ can be written $g = f'+\overline{f'}$ for some $f':W\wedge \Delta^{n}_{top,+} \to\mcal{Z}$ which is homotopic to $0$. 
For this we consider the lift $f{'}$ of $g$,
\begin{equation*}
 \xymatrix{
\{v\} \ar@{^{(}->}[d]^-{\wkeq}\ar[r]^-{f} & \phom{W\wedge\Delta^{\bullet}_{top,+}}{\mcal{Z}} \ar@{->>}[d]^{N} \\
\Delta^{n} \ar[r]^-{g}\ar@{-->}[ur]^-{f'}& \phom{W\wedge\Delta^{\bullet}_{top,+}}{\mcal{Z}}^{av}.
}
\end{equation*}
The map $f{'}:W\wedge \Delta^{n}_{top,+}\to \mcal{Z}$ satisfies $f{'}+\overline{f{'}}=N(f{'})=g$, the restriction of $f{'}$ to $v \in \Delta^{n}_{top}$ is 
homotopic to the $0$-map and, since $\Delta^{n}_{top}$ is contractible, $f{'}$ is homotopic to the $0$-map as well. 

Therefore we conclude that 
\begin{equation*}
 d\mapsto (\phom{W\wedge\Delta^{d}_{top,+}}{\mcal{Z}}_{0})^{av} 
\end{equation*}
is the path-component of $0$ in $\phom{W\wedge\Delta^{d}_{top,+}}{\mcal{Z}}^{av}$.
\end{proof}

\begin{proposition}\label{kernsurj}
 Let $i:A\hookrightarrow W$ be an equivariant cofibration. Then 
\begin{equation*}
 \rtop{q}(W/A)_{0}\twoheadrightarrow \ker[\rtop{q}(W)_{0} \xrightarrow{i^*} \rtop{q}(A)_{0}].
\end{equation*}
\end{proposition}
\begin{proof}
 Suppose that $[f]\in \ker(\rtop{q}(W)_{0} \xrightarrow{i^*} \rtop{q}(A)_{0})$ is a $d$-simplex. Then $[f]$ is represented by an equivariant map $f:W\wedge \Delta^{d}_{top,+} \to \mcal{Z}$ which is equivariantly homotopic to $0$. Since $i^*[f]=0$ this means that $i^{*}f \in \phom{A\wedge\Delta^{d}_{top,+}}{\mcal{Z}}_{0}^{av}$. Thus there is a continuous map $h:A\wedge \Delta^{d}_{top,+} \to \mcal{Z}$, (nonequivariantly homotopic to $0$), such that $i^*f=f|_{A\wedge\Delta^{d}_{top,+}} = h + \overline{h}$. 
Since $h$ is (nonequivariantly) homotopic to the $0$-map, $h: A\wedge \Delta^{d}_{top,+}\to \mcal{Z}$ extends to a continuous map $h':W\wedge \Delta^{d}_{top,+}\to \mcal{Z}$ which is (nonequivariantly) homotopic to the $0$-map.

Explicitely, let $H:A\wedge\Delta^{d}_{top,+}\wedge I_+\to \mcal{Z}$ be a homotopy such that $H(-,0) = 0$ and $H(-,1) = h$. By the homotopy extension property for cofibrations, the dotted arrow exists in the diagram
\begin{equation*}
\xymatrix{
 A\wedge\Delta^{d}_{top+}\wedge I_+ \coprod_{A\wedge\Delta^{d}_{top,+}\wedge\{0\}_+} W\wedge \Delta^{d}_{top,+}\wedge \{0\}_+ \ar@{^{(}->}[d]\ar[r]^-{H\amalg 0} & \mcal{Z}  \\
W\wedge \Delta^{d}_{top,+}\wedge I_+ \ar@{-->}[ur]^{H'} . & 
}
\end{equation*}
Now $H'(-,1)=h'$ is the desired extension of $h$, $H'$ is a homotopy between $h'$ and the $0$-map and $F \eqdef f-(h'+\overline{h'})$ represents the same class as $[f]$. Since $F|_{A\wedge \Delta^{d}_{top,+}} = 0$ the map $F$ defines the map $F':W/A\wedge \Delta^{d}_{top,+} \to \mcal{Z}$ such that 
\begin{equation*}
F =p^{*}F' : W \xrightarrow{p} W/A\wedge\Delta^{d}_{top,+} \to \mcal{Z}.
\end{equation*}

Therefore 
\begin{equation*}
 \rtop{q}(W/A)_{0} \twoheadrightarrow \ker(i^{*}:\rtop{q}(W)_{0} \to \rtop{q}(A)_{0}),
\end{equation*}
because $p^{*}[F'] =[p^*F']=[F] = [f]$.
\end{proof}

\begin{lemma}(c.f. \cite[Lemma 8.8]{LLM:real})
Suppose that $A_{\bullet}$ is a simplicial $G$-module. Then 
\begin{equation*}
 |A_{\bullet}^{av}| = |A_{\bullet}|^{av}
\end{equation*}
\end{lemma}
\begin{proof}
 Let $f_{\bullet}:B_{\bullet} \to C_{\bullet}$ be a  map between simplicial sets, then $|\im f_{\bullet}| = \im |f_{\bullet}|$. The lemma follows since $(-)^{av}$ is defined to be the image of the map $N=1+\sigma$.
\end{proof}

\begin{proposition}\label{avfib}
 Suppose that $Y=|Y_{\bullet}|$ is the realization of a based $G$-simplicial set. Then $\mcal{Z}_{0}(Y)_{0} \to \mcal{Z}_{0}(Y)^{av}_{0}$ and $\mcal{Z}/2_{0}(Y)_{0} \to \mcal{Z}/2_{0}(Y)_{0}^{av}$ are Serre fibrations.
\end{proposition}
\begin{proof}
 If $Y$ is a based set and $A$ is an abelian group then define $A\otimes Y = \oplus_{y\in Y\backslash \{*\}} A$. If $Y_{\bullet}$ is a based $G$-simplicial set then  $A\otimes Y_{\bullet}$ is a $G$-simplicial set. In case $A= \Z$ or $A=\Z/2$ we have $\mcal{Z}(|Y_{\bullet}|)_{0} =|\Z\otimes Y_{\bullet}|$ and $\mcal{Z}/2_{0}(|Y_{\bullet}|)_{0} =|\Z/2\otimes Y_{\bullet}|$ (see \cite{DS:equiDT, MC:class}).
The map $\Z\otimes Y_{\bullet} \to (\Z\otimes Y_{\bullet})^{av}$ is a surjection between simplicial abelian groups and so is a fibration of simplicial sets and similarly for $\Z/2\otimes Y_{\bullet} \to (\Z/2\otimes Y_{\bullet})^{av}$. 

The realization of a Kan fibration is a Serre fibration and therefore both $\mcal{Z}_{0}(Y)_{0} =|\Z\otimes Y_{\bullet}| \xrightarrow{N} |(\Z\otimes Y_{\bullet})^{av}|=\mcal{Z}_{0}(Y)_{0}^{av}$ and $\mcal{Z}/2_{0}(Y)_{0} =|\Z/2\otimes Y_{\bullet}| \xrightarrow{N} |(\Z/2\otimes Y_{\bullet})^{av}|=\mcal{Z}/2_{0}(Y)_{0}^{av}$  are Serre fibrations of topological spaces. 
\end{proof}

Below we apply this proposition to the cases $Y= S^{q,q}$ and $Y=S^{q,q}\wedge \Z/2_{+}$ so we make explicit that these are realizations of $G$-simplicial sets. We consider $\Z/2$ as a simplicial set in the usual way. The simplicial set $S^{1,0}$ is the ordinary $S^{1}$ with trivial action. The simplicial set 
$S^{0,1}$ is  the simplicial whose nondegenerate simplices are two vertices $\{0\}$ and 
$\{\infty\}$ and two $1$-simplices. The $G$-action fixes the vertices and switches the 
$1$-simplices. The realization of this simplicial set is the usual $S^{0,1}$.  Now $S^{p,q}$ is 
the $G$-simplicial set $S^{p,q} = (S^{1,0})^{\wedge p}\wedge(S^{0,1})^{\wedge q}$ and its realization is the usual $S^{p,q}$.

\begin{lemma}\label{zeroav}
 Let $W$ be a based $G$-space with trivial action.  Suppose that $Z$ is a topological $G$-module such that $Z\xrightarrow{N} Z^{av}$ is a Serre fibration. Then
\begin{equation*}
\phom{W\wedge\Delta^{\bullet}_{top,+}}{Z}^{av}_{0} = \phom{W\wedge\Delta^{\bullet}_{top,+}}{Z^{av}}_{0} . 
\end{equation*}
\end{lemma}
\begin{proof}
 Since $W$ and $\Delta^{d}_{top}$ have trivial actions,
\begin{equation*}
 \phom{W\wedge\Delta^{\bullet}_{top,+}}{\mcal{Z}/2_0(Y)_{0}}^{av}_{0} \subseteq \phom{W\wedge\Delta^{\bullet}_{top,+}}{\mcal{Z}/2_0(Y)_{0}^{av}}_{0} ,
\end{equation*}
and we wish to see that it is onto.

Suppose that $f:W\wedge \Delta^{d}_{top,+} \to Z^{av}$ is a map which is homotopic to $0$. Let $H$ be a homotopy between $0$ and $f$ and let $H'$ be a lift of $H$,
\begin{equation*}
 \xymatrix{
W\wedge\Delta^{d}_{top,+}\wedge \{0\}_{+}\ar[r]^-{0}\ar@{^{(}->}[d]^-{\wkeq}  & Z\ar@{->>}[d]^-{N} \\ 
W\wedge\Delta^{d}_{top,+}\wedge I_{+} \ar[r]^-{H}\ar@{-->}[ur]^-{H'} &  Z^{av} ,
}
\end{equation*}
which exists because the right-hand vertical map is a fibration.
Finally the map $f'(-) = H'(-,1)$ satisfies $f' + \overline{f}' = f$ and $H'$ is a homotopy between $0$ and $f'$. 
\end{proof}

\begin{proposition}
 Suppose that $W$ has trivial $G$-action. Then for all $n\geq 0$ and all $q\geq 0$,
\begin{equation*}
 \widetilde{\mcal{R}}^q_{top}(W\wedge\Z/2_+)_{0} = \{0\} .
\end{equation*}
\end{proposition}
\begin{proof}
First recall \cite[Lemma 2.4]{DS:equiDT} that given a finite $G$-set $Z$ then there is a $G$-homeomorphism
\begin{equation*}
 \Hom{*}{Z_+}{\mcal{Z}_{0}(Y)_{0}} \xrightarrow{\iso}\mcal{Z}_{0}(Y\wedge Z_+)_{0}
\end{equation*}
defined by $f\mapsto \sum_{z\in Z} f(z) \wedge z$.
This yields
 \begin{align*}
 \widetilde{\mcal{R}}^q_{top}(W\wedge \Z/2_+)_{0}  = & \frac{\phom{W\wedge\Z/2_+\wedge\Delta^{\bullet}_{top,+}}{\mcal{Z}/2_{0}(S^{q,q})_{0}}_{0}^G}
{\phom{W\wedge\Z/2_+\wedge\Delta^{\bullet}_{top,+}}{\mcal{Z}/2_0(S^{q,q})_{0}}_{0}^{av}}  = \\ = & \frac{\phom{W\wedge\Delta^{\bullet}_{top,+}}{\mcal{Z}/2_{0}(S^{q,q}\wedge\Z/2_+)_{0}}_{0}^{G}}
{\phom{W\wedge\Delta^{\bullet}_{top,+}}{\mcal{Z}/2_0(S^{q,q}\wedge\Z/2_+)_{0}}_{0}^{av}} = \\ = & 
\frac{\phom{W\wedge\Delta^{\bullet}_{top,+}}{\mcal{Z}/2_{0}(S^{q,q}\wedge\Z/2_+)_{0}^{G}}_{0}}
{\phom{W\wedge\Delta^{\bullet}_{top,+}}{\mcal{Z}/2_0(S^{q,q}\wedge\Z/2_+)_{0}^{av}}_{0}},
\end{align*}
where the last equality follows from Lemma \ref{zeroav} and Proposition \ref{avfib} because $W$ has trivial $G$-action.
But since the action of $G$ on $S^{q,q}\wedge \Z/2_+$ is free we have the isomorphism 
\begin{equation*}
 \mcal{Z}/2_{0}(S^{q,q}\wedge \Z/2_+)_{0}^{av} \xrightarrow{\iso} \mcal{Z}/2_{0}(S^{q,q}\wedge\Z/2_+)_{0}^{G}
\end{equation*}
 and therefore 
\begin{equation*}
 \widetilde{\mcal{R}}^q_{top}(W\wedge \Z/2_+)_{0}  = \frac{\phom{S^n\wedge\Delta^{\bullet}_{top,+}}{\mcal{Z}/2_{0}(S^{q,q}\wedge\Z/2_+)_{0}^{G}}_{0}}
{\phom{W\wedge\Delta^{\bullet}_{top,+}}{\mcal{Z}/2_0(S^{q,q}\wedge\Z/2_+)_{0}^{av}}_{0}} = \{0\}.
\end{equation*}
\end{proof}

Recall that the action of $G$ on a based set $(Y,*)$ is said to be free if $Y^{G} = *$.

\begin{corollary}\label{freezero}
 Suppose that $W$ is a based finite $G$-$CW$ complex with free $G$-action. Then 
\begin{equation*}
 \rtop{q}(W)_{0} = \{0\}.
\end{equation*}
\end{corollary}
\begin{proof}

First we observe that $\rtop{q}(W_{n+1})_{0} \to \rtop{q}(W_{n})_{0}$ is an isomorphism 
for any $n$. Indeed since $W$ is a free $G$-$CW$ complex $W_{n}/W_{n-1}$ is a wedge of 
spheres of the form $S^{n}\wedge \Z/2_+$. By the previous proposition $\rtop{q}(W_{n+1}/W_{n})_{0} = \rtop{q}(\vee S^{n+1}\wedge\Z/2_+)_{0} = \{0\}$.

By Proposition \ref{kernsurj}, $\rtop{q}(W_{n+1}/W_{n})_{0} \twoheadrightarrow \ker(\rtop{q}(W_{n+1})_{0} \to \rtop{q}(W_{n})_{0})$ is surjective and therefore $\widetilde{\mcal{R}}^{q}_{top}(W_{n+1})_{0} \subseteq\widetilde{\mcal{R}}^{q}_{top}(W_{n})_{0}$. By Corollary \ref{rsurj} this map is 
onto as well and therefore $\widetilde{\mcal{R}}^{q}_{top}(W_{n+1})_{0}  = \widetilde{\mcal{R}}^{q}_{top}(W_{n})_{0}$.
Since $W = W_{N}$ for large $N$ we conclude that  $\rtop{q}(W)_{0} = \rtop{q}(W_{0})_{0}=\{0\}$.
\end{proof}

\begin{corollary}\label{rfix}
 Suppose that $W$ is a finite $G$-$CW$-complex. Then
\begin{equation*}
 \rtop{q}(W)_{0} \xrightarrow{\iso} \rtop{q}(W^{G})_{0}
\end{equation*}
is an isomorphism of simplicial abelian groups.
\end{corollary}
\begin{proof}
 Consider the cofibration sequence $W^{G} \hookrightarrow W \to W/W^{G}$. The space $W/W^{G}$ has a free $G$-action and so Proposition \ref{kernsurj} and Corollary \ref{freezero} imply that
\begin{equation*}
 \{0\}= \rtop{q}(W/W^{G})_{0} \twoheadrightarrow \ker[\rtop{q}(W)_{0} \to \rtop{q}(W^{G})_{0}]
\end{equation*}
is surjective. Therefore $\rtop{q}(W)_{0} \subseteq \rtop{q}(W^{G})_{0}$. Since it is also a surjection by Corollary \ref{rsurj} it is an isomorphism.
\end{proof}

For a based $G$-$CW$ complex $W$ define 
$$
\mcal{H}^{q}(W) = \phom{W^{G}\wedge\Delta^{\bullet}_{top,+}}{\mcal{R}_{0}(S^{q,q})_{0}}.
$$
The homotopy groups of $\mcal{H}^{q}(W)$ compute singular cohomology of the fixed point space, $\pi_{k}\mcal{H}^{q}(W) = H^{q-k}_{sing}(W^{G}, \Z/2)$.

\begin{theorem}
 Let $W$ be a finite $G$-$CW$-complex. Then
\begin{equation*}
 \pi_{i}\rtop{q}(W)_{0} \to \pi_{i}\mcal{H}^{q}(W)= H^{q-i}_{sing}(W^{G};\Z/2)
\end{equation*}
is an isomorphism for $i\geq 2$ and an injection for $i=0,1$.
\end{theorem}
\begin{proof}
 Since $\mcal{H}^{q}(W) = \mcal{H}^{q}(W^{G})$ and $\rtop{q}(W)_{0} = \rtop{q}(W^{G})_{0}$ by Corollary \ref{rfix} we immediately reduce to the case that $W=W^{G}$. 
Since $G$ acts trivially on $W$,  Lemma \ref{zeroav} and Proposition \ref{avfib} give
\begin{align*}
 \widetilde{\mcal{R}}^q_{top}(W)_{0}  = & \frac{\phom{W\wedge\Delta^{\bullet}_{top,+}}{\mcal{Z}/2_{0}(S^{q,q})_{0}}_{0}^G}
{\phom{W\wedge\Delta^{\bullet}_{top,+}}{\mcal{Z}/2_0(S^{q,q})_{0}}_{0}^{av}}  = \\ = & \frac{\phom{W\wedge\Delta^{\bullet}_{top,+}}{\mcal{Z}/2_{0}(S^{q,q})_{0}^{G}}_{0}}
{\phom{W\wedge \Delta^{\bullet}_{top,+}}{\mcal{Z}/2_0(S^{q,q})_{0}^{av}}_{0}}.
\end{align*}

By \cite[Proposition 8.3]{LLM:real} the short exact sequence
\begin{equation*}
0\to \mcal{Z}/2_0(S^{q,q})_{0}^{av} \to \mcal{Z}/2_0(S^{q,q})_{0}^{G} \to \mcal{R}_{0}(S^{q,q})_{0}\to 0
\end{equation*}
is a principle fibration sequence.

Finally comparing homotopy fiber sequences of simplicial abelian groups
\begin{equation*}
 \xymatrix@-1pc{
\Hom{}{W\wedge\Delta^{\bullet}_{top,+}}{\mcal{Z}^{av}}_{0} \ar[r]\ar[d]&
\Hom{}{W\wedge\Delta^{\bullet}_{top,+}}{\mcal{Z}^{G}}_{0} \ar[r]\ar[d]& 
\widetilde{\mcal{R}}^q_{top}(W)_{0}\ar[d] \\
 \Hom{}{W\wedge\Delta^{\bullet}_{top,+}}{\mcal{Z}^{av}} \ar[r]&
\Hom{}{W\wedge\Delta^{\bullet}_{top,+}}{\mcal{Z}^{G}} \ar[r]& 
\mcal{H}^{q}(W)  
}
\end{equation*}
yields the result.

\end{proof}
\begin{corollary}\label{rdual}
 Let $W$ be a finite $G$-$CW$-complex. Then
\begin{equation*}
 \pi_i\rtop{q}(W) \to \pi_{i}\mcal{H}^{q}(W)=H^{q-i}_{sing}(W^{G};\Z/2)
\end{equation*}
is an isomorphism for $i\geq 2$.
\end{corollary}
\begin{proof}
The map $\pi_{i}\rtop{q}(W)_{0}\to \pi_{i}\rtop{q}(W)$ is an isomorphism for $i\geq 2$ and an injection for $i = 0,1$
\end{proof}

\appendix
\section{Topological Monoids}
In this appendix we collect a few simple results on topological monoids. By \textit{topological monoid} we will mean a compactly generated Hausdorff topological abelian monoid (and similarly for the phrase topological group). An abelian monoid $M$ is said to have the cancellation property if for any $n,m,p\in M$ $n+p=m+p$ implies $m=n$.

\begin{lemma}\label{qmcl}
 Suppose that $M$ is a topological monoid with the cancellation property. If $+: M\times M \to M$ is closed and $N\subseteq M$ is a closed submonoid then the quotient map $\pi: M\to M/N$ is closed. 
\end{lemma}
\begin{proof}
 Suppose that $V\subseteq M$ is closed. Then since $\pi:M\to M/N$ is a quotient map to see that $\pi V$ is closed it is enough to see that $\pi^{-1}\pi V \subseteq M$ is closed. But $\pi^{-1}\pi V = (V+N) \cap (M +N)$ which is closed.
\end{proof}

\begin{lemma}\label{lemtop1}
 Let $M$ be a topological monoid with the cancellation property and let $N\subseteq M$ be a  submonoid. Suppose that $M/N$ is a topological monoid, $M^{+}$ is a topological group and $N^{+}$ is closed. Then the isomorphism of groups
\begin{equation*}
 \left(\frac{M}{N}\right)^{+} \to \frac{M^{+}}{N^{+}}
\end{equation*}
is an isomorphism of topological groups.
\end{lemma}
\begin{proof} 
The map $M \to M^{+} \to M^{+}/N^{+}$ sends $N$ to $0$ and so we obtain the monoid 
homomorphism $M/N \to M^{+}/N^{+}$ which is continuous. This induces the continuous 
group homomorphism $\phi:(M/N)^{+} \to M^{+}/N^{+}$. 

On the other hand the topological monoid quotient map $M \to M/N$ induces the  continuous 
group homomorphism $M^{+} \to (M/N)^{+}$. Since $N^{+}$ is mapped to $0$ it induces
the continuous group homomorphism $\psi:M^{+}/N^{+} \to (M/N)^{+}$. 

The continuous maps $\psi$ and $\phi$ are easily seen to be inverse to each other.
\end{proof}

Recall that if $A$ is a topological monoid with $G$-action we write $A^{av}\subseteq A$ for the image of $N=1+\sigma$, so $A^{av}\subseteq A$ is the topological submonoid consisting of elements of the form $a+\sigma a$.

\begin{proposition}\label{appcycagr}
Suppose that $M$ is a Hausdorff topological abelian monoid with the cancellation property and that $M^{+}$ is a Hausdorff group. Suppose that $G$ acts on $M$. Then the isomorphism of groups
\begin{equation*}
 (M^{G})^{+} \xrightarrow{\iso} (M^{+})^{G}
\end{equation*}
is an isomorphism of topological groups. If $(M^{+})^{av}\subseteq M^{+}$ is closed then 
\begin{equation*}
 (M^{av})^{+} \xrightarrow{\iso} (M^{+})^{av}
\end{equation*}
is an isomorphism of topological groups.
\end{proposition}
\begin{proof}
We just have to show that the ``identity'' map
\begin{equation*}
 (M^{+})^{G} \to (M^{G})^{+}
\end{equation*}
is continuous. 

The group completion $M^+$ is topologized as the quotient
\begin{equation*}
 M \times M \xrightarrow{q} M^{+}.
\end{equation*}
where $q(a,b) = a-b$.  The map $q:q^{-1}(M^{+})^{G} \to (M^{+})^{G}$ is again a quotient map since $(M^{+})^{G}$ is closed.

Consider the map $M\times M \xrightarrow{id\times \sigma\pi_{2}} M\times M\times M \xrightarrow{\Delta \times id} M^{\times 4} \xrightarrow{+} M\times M$ which sends $(a,b) \mapsto ( a, b, \sigma b, \sigma b) \mapsto (a+\sigma b, b + \sigma b)$. This is a continuous map. Its restriction to $q^{-1}(M^{+})^{G}$ is a continuous map $q^{-1}(M^{+})^{G} \to M^{G}\times M^{G}$ which induces the identity map on quotients
\begin{equation*}
( M^{+})^{G} \to (M^{G})^{+},
\end{equation*}
and therefore this is a continuous map.

The second statement for averaged cycles is proved in a similar fashion.
\end{proof}

\begin{lemma}\label{mgrp2}
 Let $M$ be a topological monoid with the cancellation property. Suppose that $M^{+}$ is a topological group and $2M^{+}$ is closed in $M^{+}$. Then 
\begin{equation*}
 \frac{M}{2M} \to \frac{M^{+}}{2M^{+}}
\end{equation*}
is an isomorphism of topological abelian groups.
\end{lemma}
\begin{proof}
For $m\in M$ write $[m]$ for its image in $M/2M$. This quotient monoid is a group since $[m]+[m] = 0$.

 For $(m,n)\in M^{+}$ write $[m,n]$ for its image in $M^{+}/2M^{+}$. The map $M/2M \to M^{+}/2M^{+}$ which sends $[m]$ to $[m,0]$ is continuous because $M\to M/2M$ is a quotient. It is an injection because if $[m,0] = [0,0]$ then there is $(2n, 2n')$ such that $m + 2n' = 2n$ which says that $[m] = 0$. It is surjective since $[m,n] = [m+n,n+n] = [m+n,0]$. The inverse $M^{+}/2M^{+} \to M/2M$ is continuous since it is the map $[m,n] \mapsto [m+n, 0]=[m+n, n+n]$.
\end{proof}

\begin{proposition}\label{qimav}
 Let $X$ be a normal quasi-projective real variety. 
\begin{enumerate}
\item The continuous homomorphism $N:\mcal{Z}^{q}/2(X_{\C})\to \mcal{Z}^{q}/2(X_{\C})^{av}$ induces an isomorphism of topological groups 
\begin{equation*}
 \frac{\mcal{Z}^{q}/2(X_{\C})}{\mcal{Z}^{q}/2(X_{\C})^{G}} \to \mcal{Z}^{q}/2(X_{\C})^{av} .
\end{equation*}
\item The continuous homomorphism $\mcal{Z}^{q}(X_{\C})^{av} \to \mcal{Z}^{q}/2(X_{\C})^{av}$ induces an isomorphism of topological groups
\begin{equation*}
 \frac{\mcal{Z}^{q}(X_{\C})^{av}}{2\mcal{Z}^{q}(X)^{G}} \to \mcal{Z}^{q}/2(X_{\C})^{av}.
\end{equation*}

\end{enumerate}
\end{proposition}
\begin{proof}
 Addition is a closed map for the monoid $\mcal{C}_{0}(\P_{\C}^{q})(X_{\C})$ (see the proof of Proposition \ref{Gcocomm}) and therefore we conclude by Lemma \ref{qmcl} that addition is also closed for both the effective cocycles $\mcal{C}^{q}(X_{\C}) = 
\mcal{C}_{0}(\P_{\C}^{q})(X_{\C})/\mcal{C}_{0}(\P_{\C}^{q-1})(X_{\C})$ and  closed for $\mcal{C}^{q}/2(X_{\C})$ . 

By Lemma \ref{mgrp2} the map $\mcal{C}^{q}/2(X_{\C}) \xrightarrow{\iso} \mcal{Z}^{q}/2(X_{\C})$ is an isomorphism of topological groups and therefore addition is closed for $\mcal{Z}^{q}/2(X_{\C})$. In particular $N=1+\sigma:\mcal{Z}^{q}/2(X_{\C}) \to \mcal{Z}^{q}/2(X_{\C})$ is closed. Since $\mcal{Z}^{q}/2(X_{\C})$ is $2$-torsion $\ker(N) = \mcal{Z}^{q}/2(X_{\C})^{G}$. It now follows that 
\begin{equation*}
 \frac{\mcal{Z}^{q}/2(X_{\C})}{\mcal{Z}^{q}/2(X_{\C})^{G}} \xrightarrow{N} \mcal{Z}^{q}/2(X_{\C})^{av} 
\end{equation*}
is an isomorphism of topological groups.
For the second statement we need to conclude that the continuous bijection $\mcal{Z}^{q}(X_{\C})^{av}/2\mcal{Z}^{q}(X_{\C})^{G} \to \mcal{Z}^{q}/2(X_{\C})^{av}$ has a continuous inverse. Write $g$ for the inverse. Then
\begin{equation*}
 \xymatrix{
\mcal{Z}^{q}(X_{\C}) \ar[r]\ar[d] & \mcal{Z}^{q}(X_{\C})^{av} \ar[d] \\
\mcal{Z}^{q}/2(X)^{av} \ar[r]^{g} & \displaystyle{\frac{\mcal{Z}^{q}(X_{\C})^{av}}{2\mcal{Z}^{q}(X_{\C})^{G}}}
}
\end{equation*}
is commutative and each map except possibly $g$ is continuous. By the first part of the proposition the composition $\mcal{Z}^{q}(X_{\C}) \to \mcal{Z}^{q}/2(X_{\C})\xrightarrow{N} \mcal{Z}^{q}/2(X_{\C})^{av}$ is a quotient map and therefore $g$ is continuous. We conclude that the map is an isomorphism of topological groups.
\end{proof}

\section{Tractable Monoid Actions}\label{tract}
We recall Friedlander-Gabber's notion of tractability for a topological monoid.

\begin{defn}\cite{FG:cyc}
\begin{enumerate}
\item If $M$ is a Hausdorff topological monoid which acts on a topological space $T$, the action is said to be $\textit{tractable}$ if $T$ is the topological union of inclusions
\begin{equation*}
\varnothing = T_{-1} \subseteq T_0 \subseteq T_1 \subseteq \cdots
\end{equation*}
such that for each $n\geq 0$ the inclusion $T_{n-1}\subseteq T_{n}$ fits into a push-out of 
$M$-equivariant maps (with $R_{0}$ empty)
\begin{equation}\label{potr}
\begin{CD}
 R_n\times M @>>> S_n\times M \\
@VVV @VVV \\
T_{n-1} @>>> T_n,
\end{CD}
\end{equation}
where the upper horizontal map is induced by a cofibration $R_n\hookrightarrow S_n$ of 
Hausdorff spaces. 

The monoid $M$ is said to be \textit{tractable} if the diagonal action of $M$ on $M\times M$ 
is tractable. 

\item  If in addition $M$, $T$ are $G$-spaces say that the action of $M$ on $T$ is 
\textit{equivariantly tractable} if the action map is $G$-equivariant, the 
$R_{n}\hookrightarrow S_{n}$ are equivariant cofibrations between $G$-spaces, and the 
pushout squares (\ref{potr}) are $G$-equivariant. 
\end{enumerate}
\end{defn}

Fixed points of an equivariant cofibration is a cofibration and fixed points preserve pushouts along a closed inclusion. Therefore if $T$ is an equivariantly tractable $M$-space then it is in particular a tractable $M$-space and $T^{G}$ is a tractable $M^{G}$-space.
 
The most important feature of tractability is that the naive group completion $M\to M^{+}$ of a tractable monoid is a homotopy group completion \cite{FG:cyc}. 

It is useful to know that all of our topological groups have reasonable equivariant homotopy 
types. Below we observe that this is the case by using essentially the same reasoning as in 
\cite[Proposition 2.5]{FW:funcspc}. The essential topological property used here is that Hironaka's triangulation theorem implies that the complexification of a real variety may be 
\textit{equivariantly} triangulated (see for example \cite[Theorem 1.3]{KW:real}). 

\begin{proposition}\label{trcw}
Suppose that $T$ is a tractable $M$-space. If  $R_{n}$, $S_{n}$ have the homotopy type of a 
$CW$-complex then so does $T/M$ 
Suppose that $T$ is an equivariantly tractable $M$-space. If  $R_{n}$, and $S_{n}$ have the equivariant homotopy type of a $G$-$CW$ complex then $T/M$ has the equivariant homotopy type of a $G$-$CW$ complex.
\end{proposition}
\begin{proof}
We prove the second statement, the first follows in the same manner by discarding equivariant considerations.
 Modding out by the $M$-action in (\ref{potr}) we obtain equivariant pushout-squares
 \begin{equation*}
\begin{CD}
 R_n @>>> S_n \\
@VVV @VVV \\
T_{n-1}/M @>>> T_n/M.
\end{CD}
\end{equation*}
By induction and homotopy invariance of pushouts along $G$-cofibrations we see that $T_{n}/M$ has the homotopy type of a $G$-$CW$ complex and that $T_{n-1}/M\to T_{n}/M$ is a $G$-cofibration. By \cite[Theorem 4.9]{Waner:miln} we conclude that $\colim_{n} T_{n}/M$ has the homotopy type of a $G$-$CW$ complex. We are done since $T/M=\colim_{n} T_{n}/M$ by  the proof of \cite[Lemma 1.3]{F:algco}.
\end{proof}

\begin{proposition}\label{concof}
Let $\mcal{E}\subseteq X$ be a constructable subset of a real projective variety.
\begin{enumerate}
\item  The space $\mcal{E}_{\C}$ has the homotopy type of a $G$-$CW$ complex.
\item Suppose that $\mcal{F}\hookrightarrow \mcal{E}$ is a closed constructable embedding.
Then $\mcal{F}_{\C}\hookrightarrow \mcal{E}_{\C}$ is an equivariant cofibration.

\end{enumerate}
\end{proposition}
\begin{proof}
 Let $\overline{\mcal{E}_{\C}}$, $\overline{\mcal{F}_{\C}}$ be closures (in $X_{\C}$) of
 $\mcal{E}_{\C}$ and $\mcal{F}_{\C}$.  There is an equivariant triangulation of $\overline{E}_{\C}$ so that 
$\overline{\mcal{F}}_{\C}$ and $\overline{\mcal{E}}_{\C}\backslash \mcal{E}_{\C}$ are 
subcomplexes \cite[Theorem 1.3]{KW:real}. Then $\mcal{E}_{\C}$ and $\mcal{F}_{\C}$ are unions of open simplices. The deformation retract of $\mcal{E}_{\C}$ onto a subsimplicial complex given in the proof of \cite[Proposition 2.5]{FW:funcspc} is an equivariant retract onto a $G$-simplicial complex (and similarly for $\mcal{F}_{\C}$) which gives the first statement. 
The construction of a deformation retract onto $\mcal{F}_{\C}$ of an open neighborhood $U$ of $\mcal{F}_{\C}$ in $\mcal{E}_{\C}$ given in \cite[Lemma C1]{FL:dual} works equivariantly which shows that $\mcal{F}_{\C}\hookrightarrow \mcal{E}_{\C}$ is a cofibration. 

\end{proof}

As shown in \cite[Proposition 1.3]{FG:cyc} the Chow monoids associated to complex varieties are tractable and in \cite[Proposition C.3]{FL:dual} these results are extended to certain constructable submonoids of Chow monoids. Their proofs work equivariantly to give the equivariant analogue of their result. A submonoid $N\subseteq M$ is said to be \textit{full} if whenever $m+m'\in N$ then both $m$, $m'\in N$. The condition below on $\mcal{E} \subseteq \mcal{C}_{k}(X)$ in the proposition is satisfied if $\mcal{E}$ is Zariski closed or if it is a full submonoid. 

\begin{proposition}\label{gtract}
Let $X$ be a projective real variety and $\mcal{E} \subseteq \mcal{C}_{k}(X)$ by a constructable submonoid such that $+:\mcal{E}\times\mcal{E} \to \mcal{E}$ is a Zariski closed mapping. Then
\begin{enumerate}
\item $\mcal{E}_{\C}$ is an equivariantly tractable monoid.
\item $\mcal{E}_{\C}^{+}$ has the homotopy type of a $G$-$CW$ complex.
\item If  $\mcal{F}\subseteq \mcal{E}$ is a closed constructable embedding then $\mcal{E}_{\C}$ is tractable as an $\mcal{F}_{\C}$-space and $\mcal{E}_{\C}/\mcal{F}_{\C}$ is an equivariantly tractable monoid.
\item Suppose that $\mcal{F}\hookrightarrow \mcal{E}$ is a closed constructable embedding.
Then $\mcal{F}^{+}\subseteq \mcal{E}^{+}$ is closed and the sequence
 \begin{equation*}
0\to  \mcal{F}_{\C}^{+}\to\mcal{E}_{\C}^{+}\to \left(\mcal{E}_{\C}/\mcal{F}_{\C}\right)^{+}\to 0
 \end{equation*}
 is an equivariant short exact sequence of groups of spaces of the homotopy type of a $G$-$CW$ complex. 

\end{enumerate}
\end{proposition}
\begin{proof}
In \cite[Proposition C.3]{FL:dual} the monoid $\mcal{E}_{\C}$ is shown to be tractable as follows. Write  $\mcal{E}(d)= \mcal{E}_{\C}\cap\mcal{C}_{k,d}(X_{\C})$ and let $\nu:\mathbb{N}^{2}\to\mathbb{N}$ be a bijection such that if $a\leq c$, $b\leq d$ then $\nu(a,b) \leq \nu(c,d)$. Define 
$$
S_{n} = \mcal{E}(a_{n})\times \mcal{E}(b_{n}), \;\;\textrm{where}\;\; \nu(a_{n},b_{n}) = n 
$$
 $$R_{n} = \im\left(\bigcup_{c\geq 0}\mcal{E}(a_{n}-c)\times \mcal{E}(b_{n}-c)\times \mcal{E}(c) \to \mcal{E}(a_{n})\times \mcal{E}(b_{n})\right) \subseteq S_{n}
$$
and 
$$
T_{n} = \im\left((\bigcup_{\nu(a,b)\leq n} \mcal{E}(a)\times \mcal{E}(b))\times \mcal{E}_{\C} \to\mcal{E}_{\C}\times\mcal{E}_{\C}\right).
$$
The spaces $R_{n}$, $S_{n}$, and $T_{n}$ are $G$-spaces, fit into the appropriate pushout 
squares, and $R_{n}\hookrightarrow S_{n}$ is a closed constructable embedding since 
addition is closed on $\mcal{E}_{\C}$ therefore by Proposition \ref{concof} the 
$R_{n}\hookrightarrow S_{n}$ are equivariant cofibrations. 
This shows that $\mcal{E}$ is equivariantly tractable. The third item follows from similar 
consideration of \cite[Proposition C.3]{FL:dual}. The second item follows by applying 
Proposition \ref{trcw}. 

For the last part write $R'_{n}$, $S'_n$, and $T'_{n}$ for the spaces above giving the 
tractability of $\mcal{F}_{\C}$. Then $R'_{n}\subseteq R_{n}$ and $S_{n}'\subseteq S_{n}$ are cofibrations. Considering the comparison of pushouts 
$$
T'_{n}/\mcal{F}_{\C}= T'_{n-1}/\mcal{F}_{\C}\bigcup_{R'_{n}}S'_{n} \to T_{n}/\mcal{E}_{\C}= T_{n-1}/\mcal{E}_{\C}\bigcup_{R_{n}}S_{n}
$$
we see by induction that $T'_{n}/\mcal{F}_{\C} \hookrightarrow T_{n}/\mcal{E}_{\C}$ is a cofibration and in particular is closed. Therefore $\mcal{F}_{\C}^{+}=\colim_{n}T'_{n}/\mcal{F}_{\C} \subseteq \colim_{n}T_{n}/\mcal{E}_{\C}= \mcal{E}_{\C}^{+}$ is a closed subspace \cite[Proposition A.5.5]{FP:cell}. Finally $\mcal{E}_{\C}/\mcal{F}_{\C})^{+}= \mcal{E}_{\C}^{+}/\mcal{F}_{\C}^{+}$ by Lemma \ref{lemtop1} which gives the displayed exact sequence.
\end{proof}

Spaces of algebraic cycles and algebraic cocycles on complex varieties are shown to have $CW$-structures or homotopy type of $CW$-spaces in \cite{LF:qproj, FW:funcspc} and in \cite{Teh:real} for real varieties.
\begin{corollary}\label{ehtype}
Let $U$ be a quasi-projective real variety. Then the spaces $\mcal{Z}_{k}(U_{\C})$, $\mcal{Z}/\ell_{k}(U_{\C})$, $\mcal{Z}^{q}(U_{\C})$, $\mcal{Z}^{q}/\ell(U_{\C})$ all have the homotopy type of a $G$-$CW$ complex. The spaces $\mcal{Z}_{k}(U_{\C})^{av}$, $\mcal{Z}/\ell_{k}(U_{\C})^{av}$, $\mcal{R}_{k}(U)$, $\mcal{Z}^{q}(U_{\C})^{av}$, $\mcal{Z}^{q}/\ell(U_{\C})^{av}$, and $\mcal{R}^{q}(U_{\C})$ all have the homotopy type of a $CW$-complex. 
\end{corollary}
\begin{proof}
 That $\mcal{Z}_{k}(U_{\C})$ has the homotopy type of a $G$-$CW$ complex follows immediately from the previous proposition. 
Let $U\subset \overline{U}$ be a projectivization. Write $\mcal{F}_{0}(\P_{\C}^{q-1})(U_{\C}) = \mcal{E}_{0}(\P_{\C}^{q-1})(U_{\C}) + \mcal{C}_{d}(\P^{q}_{\C}\times \overline{U}_{\C})$. 
This is a closed constructable submonoid $\mcal{F}_{0}(\P_{\C}^{q-1})(U_{\C}) \subseteq \mcal{E}_{0}(\P_{\C}^{q})(U_{\C})$ and 
$(\mcal{E}_{0}(\P^{q}_{\C})(U_{\C})/\mcal{F}_{0}(\P^{q-1}_{\C})(U_{\C}))^{+}\iso \mcal{Z}^{q}(U_{\C})$ has the equivariant homotopy type of a $G$-$CW$ complex. 
Since $(\ell\mcal{C}_{k}(U_{\C}))^{+}\subseteq \mcal{C}_{k}(U_{\C})^{+}$ is closed we 
easily see that $(\ell\mcal{C}_{k}(U_{\C}))^{+} = \ell(\mcal{C}_{k}(U_{\C}))^{+}\subseteq \mcal{C}_{k}(U_{\C})^{+}$. Therefore $\mcal{Z}/\ell_{k}(U_{\C}) \iso (\mcal{C}_{k}(U_{\C})/\ell\mcal{C}_{k}(U_{\C}))^{+}$ 
has the equivariant homotopy type of a $G$-$CW$ complex. 
Similarly one sees that $\mcal{Z}^{q}/\ell(U_{\C}) \iso (\mcal{E}_{0}(\P_{\C}^{q})(U_{\C})/\mcal{F}_{0}(\P^{q-1}_{\C})(U_{\C})+\ell\mcal{E}_{0}(\P_{\C}^{q})(U_{\C}))^{+}$ has the equivariant homotopy type of a $G$-$CW$ complex.

The monoid inclusions $\ell\mcal{C}_{k}(U_{\C})\cap \mcal{C}_{k}(U_{\C})^{av}\subseteq \mcal{C}_{k}(U_{\C})^{av}\subseteq \mcal{C}_{k}(U_{\C})^{G}\subseteq \mcal{C}_{k}(U_{\C})$ are all closed and so $\mcal{Z}_{k}(U_{\C})^{av}$, $\mcal{R}_{k}(U) \iso (\mcal{C}_{k}(U_{\C})^{G}/\mcal{C}_{k}(U_{\C})^{av})^{+}$, and $\mcal{Z}/\ell_{k}(U_{\C})^{av}\iso(\mcal{C}_{k}(U_{\C})^{av}/\ell\mcal{C}_{k}(U_{\C})\cap\mcal{C}_{k}(U_{\C})^{av})^{+}$ all have the homotopy type of a $CW$-complex. 

Similarly $\mcal{Z}^{q}(U_{\C})^{av} \iso (\mcal{E}_{0}(\P^{q}_{\C})(U_{\C})^{av}/\mcal{F}_{0}(\P^{q-1}_{\C})(U_{\C})\cap \mcal{E}_{0}(\P^{q}_{\C})(U_{\C}))^{+}$, $\mcal{Z}^{q}/\ell(U_{\C})^{av} \iso (\mcal{E}_{0}(\P^{q}_{\C})(U_{\C})^{av}/(\mcal{F}_{0}(\P^{q-1}_{\C})(U_{\C})+\ell\mcal{E}_{0}(\P^{q}_{\C})(U_{\C}))\cap \mcal{E}_{0}(\P^{q}_{\C})(U_{\C})^{av})^{+}$, and $\mcal{R}^{q}(U_{\C}) \iso (\mcal{E}_{0}(\P^{q}_{\C})(U_{\C})^{G}/\mcal{F}_{0}(\P^{q-1}_{\C})(U_{\C})^{G}+\mcal{E}_{0}(\P^{q}_{\C})(U_{\C})^{av})^{+}$ all have the homotopy type of a $CW$ complex.
\end{proof}

\bibliographystyle{amsalpha} 
\bibliography{remreal}
\end{document}